\documentclass[11pt,a4paper]{amsart}

\usepackage{latexsym}
\usepackage{amsmath}
\usepackage{amsthm}
\usepackage{amscd}
\usepackage{amssymb}
\usepackage{mathtools}
\usepackage[colorlinks=true,citecolor=green,linkcolor=blue,urlcolor=magenta]{hyperref}
\usepackage[capitalize,nameinlink,noabbrev]{cleveref}
\usepackage{mdframed}
\usepackage{graphicx}
\usepackage{hyperref}
\usepackage{enumerate}
\usepackage{mathrsfs}
\usepackage{url}
\usepackage{tikz}
\usepackage{tikz-cd}
\usepackage{calc}
\usepackage{color}
\usepackage{tcolorbox}
\usepackage{bm}
\usepackage[makeroom]{cancel}
\usepackage{verbatim}
\usepackage{enumitem}
\usepackage{bold-extra}

\usepackage[textwidth=27mm, textsize=scriptsize, color=green!20]{todonotes}

\newtheorem{introtheorem}{Theorem}
\newtheorem{introprop}[introtheorem]{Proposition}
\newtheorem{introcorollary}[introtheorem]{Corollary}
\newtheorem*{theorem*}{Theorem}
\newtheorem{theorem}{Theorem}[section]
\newtheorem{proposition}[theorem]{Proposition}
\newtheorem{lemma}[theorem]{Lemma}
\newtheorem{corollary}[theorem]{Corollary}
\newtheorem{conjecture}[theorem]{Conjecture}
\newtheorem{question}[theorem]{Question}
\theoremstyle{definition}
\newtheorem{defn}[theorem]{Definition}
\theoremstyle{definition}
\newtheorem{remark}[theorem]{Remark}
\theoremstyle{definition}
\newtheorem{example}[theorem]{Example}
\theoremstyle{definition}

\DeclareMathOperator{\spec}{Spec}

\DeclareMathOperator{\h}{h}

\DeclareMathOperator{\an}{an}

\DeclareMathOperator{\ord}{ord}

\DeclareMathOperator{\can}{can}

\DeclareMathOperator{\trop}{trop}

\DeclareMathOperator{\vol}{vol}

\DeclareMathOperator{\MI}{MI}

\DeclareMathOperator{\stab}{stab}

\DeclareMathOperator{\Gal}{Gal}

\DeclareMathOperator{\lcm}{lcm}

\DeclareMathOperator{\tr}{tr}
\DeclareMathOperator{\m}{m}
\DeclareMathOperator{\cotrop}{cotrop}
\DeclareMathOperator{\Ron}{Ron}
\DeclareMathOperator{\K}{K}

\newcommand{\iu}{{\mathrm{i}\mkern1mu}}
\renewcommand{\div}{\operatorname{div}}
\newcommand{\longhookrightarrow}{\lhook\joinrel\longrightarrow}
\renewcommand{\and}{\quad \text{and} \quad}

\renewcommand{\Re}{\operatorname{Re}}
\renewcommand{\Im}{\operatorname{Im}}
\newcommand{\Ker}{\operatorname{Ker}}
\newcommand{\eu}{{\mathrm{e}}}

\newcommand{\canD}{{\overline{D}\vphantom{D}^{\can}}}
\newcommand{\ronD}{{\overline{D}\vphantom{D}^{\Ron}}}

\newcommand{\canOC}{ {\overline{O(1)}} \vphantom{{O}(1)}^{\can}}
\newcommand{\ronOC}{{\overline{{O}(1)}\vphantom{{O}(1)}^{\Ron}}}
\newcommand{\trOC}{{\overline{O_{X}}\vphantom{\overline{O_{X}}}^{\tr}}}
\newcommand{\trOcurve}{{\overline{O_{\mathscr{C}(\mathbb{C})}}\vphantom{\overline{O_{\mathscr{C}(\mathbb{C})}}}^{\tr}}}
\newcommand{\trOP}{{\overline{O_{\mathbb{P}^2(\mathbb{C})}}\vphantom{\overline{O_{X}}}^{\tr}}}

\newcommand{\canOZ}{\overline{\mathscr{O}(1)}\vphantom{\mathscr{O}(1)}^{\can}}
\newcommand{\ronOZ}{\overline{\mathscr{O}(1)}\vphantom{\mathscr{O}(1)}^{\Ron}}

\newcommand{\Gm}{\mathbb{G}_{\mathrm{m}}}
\newcommand{\ce}{\mathrm{c}}
\newcommand{\chern}{\mathrm{c}_{1}}

\usepackage{amssymb}
\usepackage[all]{xy}

\allowdisplaybreaks

\begin{document} 

\title[Limit heights and special values]{Limit heights and special values of the Riemann zeta function}

\author[Gualdi]{Roberto~Gualdi}
\address{\hspace*{-6.3mm} Fakult\"at f\"ur Mathematik, Universit\"at Regensburg, 93040 Regensburg, Germany \vspace*{-2.8mm}}
\address{\hspace*{-6.3mm} {\it Email address:} {\tt roberto.gualdi@mathematik.uni-regensburg.de}}

\author[Sombra]{Mart{\'\i}n~Sombra}
\address{\hspace*{-6.3mm} ICREA, 
  08010 Barcelona, Spain \vspace*{-2.8mm}}
\address{ 
\hspace*{-6.3mm}  Departament de Matem\`atiques i
  Inform\`atica, Universitat de Barcelona,  08007
  Bar\-ce\-lo\-na, Spain \vspace*{-2.8mm}}
\address{\hspace*{-6.3mm} Centre de Recerca Matem\`atica, 08193 Bellaterra, Spain \vspace*{-2.8mm}}
\address{\hspace*{-6.3mm} {\it Email address:} {\tt sombra@ub.edu}}

\date{\today}
\subjclass[2020]{Primary 11G50; Secondary 11M06, 14G40.}
\keywords{Height of projective points, torsion points, metrized line bundles, Archimedean amoebas, Ronkin functions.}

\begin{abstract}
  We study the distribution of the height of the intersection between
  the projective line defined by the linear polynomial
  $x_{0}+x_{1}+x_{2}$ and its translate by a torsion point.  We show
  that for a strict sequence of torsion points, the corresponding
  heights converge to a real number that is a rational multiple of a
  quotient of special values of the Riemann zeta function. We also
  determine the range of these heights, characterize the extremal
  cases, and study their limit for sequences of torsion points that
  are strict in proper algebraic subgroups.

 In addition, we interpret our main result from the viewpoint of Arakelov
  geometry, showing that for a strict sequence of torsion points the
  limit of the corresponding heights coincides with an
  Arakelov height of the cycle of the projective plane over the
  integers defined by the same linear polynomial. This is a
  particular case of a conjectural asymptotic version of the
  arithmetic B\'ezout theorem.

  Using the interplay between arithmetic and convex objects from
  the Arakelov geometry of toric varieties, we show that this Arakelov
  height can be expressed as the mean of a piecewise linear function
  on the amoeba of the projective line, which  in turn can be computed
  as the aforementioned real number.
\end{abstract}

\thispagestyle{empty}

\maketitle

\vspace{-10mm}

\setcounter{tocdepth}{1}
\tableofcontents

\section*{Introduction}

Solving systems of polynomial equations in several variables is one of
the guiding problems of mathematics, and has motivated the rise of
linear algebra and of algebraic geometry, other than being applied to
disparate areas across the sciences.

While solving a given system of polynomial equations is a business of
algorithmic research, a more theoretical approach is
concerned with obtaining the finest possible information about its
solution set in terms of data that can be read directly from the
system.  Basically, this perspective aims at understanding the complexity
of the solution set in terms of the complexity of the system.

From this point of view, the prototypical result is the classical
B\'ezout's theorem, asserting that for a generic system of $n$-many
homogeneous polynomials in $n+1$ variables, the cardinality of its
solution set in the $n$-dimensional projective space equals the
product of the degrees of these polynomials.

For systems whose coefficients are algebraic numbers we can consider
not only the geometric complexity of its solution set given by its
cardinality, but also its arithmetic complexity. This latter is
usually defined as the maximal bit-length of the integers in a
representation of the solution set and can be measured in terms
of its height, see for instance \cref{rem:4}.

In this context, it is then natural to ask whether the height of the
solution set can be predicted from the arithmetic complexity of the
system, measured for example in terms of the height of the involved
polynomials. The goal of this article is to investigate this question
through the study of an explicit example.

\vspace{\baselineskip}

Let us now set our playground more precisely.  Let
$\mathbb{P}^{2}(\overline{\mathbb{Q}})$ be the projective plane over
the algebraic closure of the field of rational numbers and consider
its (canonical) height function, that is the real-valued
function
\begin{displaymath}
  \h\colon \mathbb{P}^{2}(\overline{\mathbb{Q}})\longrightarrow \mathbb{R}
\end{displaymath}
introduced and studied by Northcott and Weil \cite{Northcott, Weil2},
see also \cref{sec:preliminaries}.  For $\omega_1,\omega_2$ varying in
the set of roots of
unity~$\mu_{\infty} \subset \overline{\mathbb{Q}}$, put
$\omega=(\omega_1,\omega_2)$ and consider the system of linear
equations on $\mathbb{P}^{2}(\overline{\mathbb{Q}})$ given as
\begin{equation}\label{eq: our system in the introduction} 
  x_{0}+x_{1}+x_{2}=x_{0}+ \omega_1^{-1}x_{1}+\omega_2^{-1}x_{2}=0.
\end{equation}
Apart from the degenerate case~$\omega_1=\omega_2=1$ which we exclude
throughout, its solution set consists of a single point, that we
denote by~$P(\omega)$.

From a geometrical perspective, note that the (algebraic) torus
  $\Gm^{2}(\overline{\mathbb{Q}})=(\overline{\mathbb{Q}}^{\times})^{2}$
  has a standard action on~$\mathbb{P}^2(\overline{\mathbb{Q}})$ as
  explained in \eqref{eq:30}, and that the two equations in~\eqref{eq:
    our system in the introduction} correspond respectively to the
  projective line
  $C=Z(x_0+x_1+x_2) \subset \mathbb{P}^{2}(\overline{\mathbb{Q}})$ and
  to its translate $\omega\, C$ by the torsion point $\omega$ of this
  torus.  The intersection of these two lines coincides with the
  solution set of the system of equations in \eqref{eq: our system in
    the introduction} and hence, when~$\omega$ is nontrivial,
\[
C\cap \omega\, C=\{P(\omega)\}.
\]

Our aim is to understand how the height of the intersection
$C\cap \omega\, C$, or equivalently that of the point $P(\omega)$,
depends on the choice~of~$\omega$. The first observation is that one
cannot determine it from the usual measures for the complexity of the
system. Indeed any such measure, like the degrees or the Newton
polytopes of the defining polynomials, their Mahler measures or any
norm depending only on the absolute values of the coefficients, is
constant as $\omega$ varies. On the contrary, different choices of the
torsion point $\omega$ can produce projective points $P(\omega)$ with
very different heights \cite[Example~5.1.1]{GualdiThesis}, see also
\cref{ex: height of some solutions} or \cref{fig:figure in
  introduction} in this introduction.

In view of this situation, we turn to the study of the distribution of
the height of~$P(\omega)$, starting by determining its range of
values.

\begin{introprop}[\cref{prop: bounds for the height}]
  \label{introprop: bounds for height}
  Let  $\omega \in \Gm^{2}(\overline{\mathbb{Q}})$ be a  nontrivial torsion point. Then
    \begin{displaymath}
    0\leq\h(C\cap \omega\, C)\leq\log(2).
  \end{displaymath}
\end{introprop}

More precisely, we show that these bounds are sharp and characterize
the cases in which each of them is attained. These results
follow from an explicit expression for the projective
point~$P(\omega)$, the basic properties of the height, and a formula
for the value of a cyclotomic polynomial evaluated at~$1$.

\vspace{\baselineskip}

Having established the range of these height values, the next question
is to understand how they distribute within the
interval~$[0,\log(2)]$.  We can gain insight into it by means of
numerical experimentation, as we detail in
\cref{sec:visualizing-result} and in the accompanying {\tt SageMath}
notebook~\cite{GS_Notebook}.

In practice, let $d$ be a positive integer and denote by $\mu_{d}$ the
set of $d$-roots of unity, so that $\mu_{d}^{2}$ coincides with the
set of $d$-torsion points of~$\Gm^{2}(\overline{\mathbb{Q}})$.
Subdividing the unit square into $d^2$-many cells, we can assign to
each of them a $d$-torsion point $\omega\in \mu_{d}^{2}$ and,
excluding the case when $\omega$ is trivial, color it with a tone of
gray that is as dark as the height of~$P(\omega)$ is large within the
range~$[0,\log(2)]$.  As shown in \cref{fig:figure in introduction},
these tones distribute on the square without following any clear
pattern.
  
\begin{figure}[!htbp] 
  \includegraphics[scale=0.3]{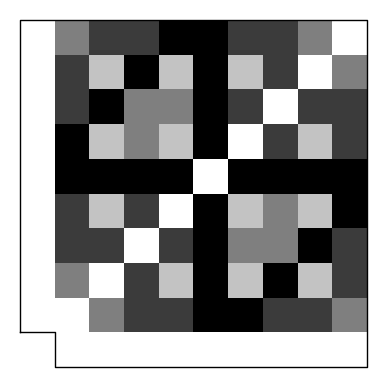} \quad\quad  \includegraphics[scale=0.3]{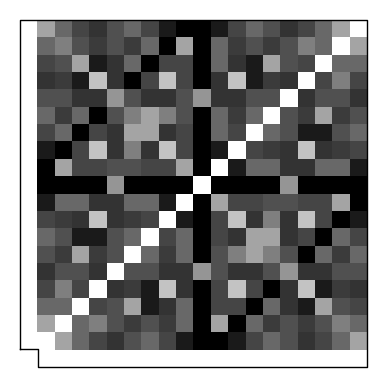}\quad\quad  \includegraphics[scale=0.3]{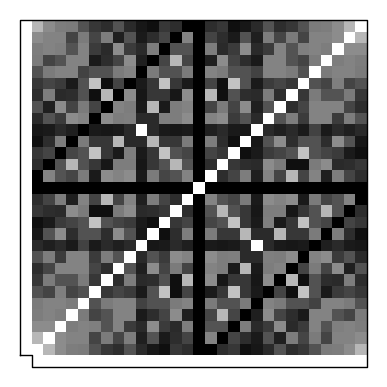}\quad\quad  \includegraphics[scale=0.3]{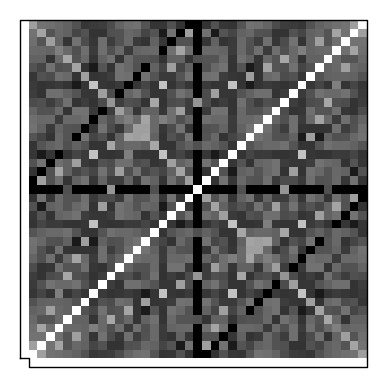}
  \caption{Heights associated to nontrivial $d$-torsion points for $d=10,20,30,40$}
\label{fig:figure in introduction}
\end{figure}

However, a careful look at these images reveals an interesting
phenomenon: as $d$ grows bigger, most of the cells get colored with a
very similar tone of gray, suggesting that most of the corresponding
heights are close to a precise real number. This becomes even more
evident for higher values of~$d$, such as those in \cref{fig:3,fig:4}.

Guided by this intuition, we focus on the study of the height
of~$P(\omega)$ for a generic torsion point~$\omega$, that is for
$\omega$ varying in a generic sequence of torsion points. Because of
the former toric Manin--Mumford conjecture, the latter is nothing else
but a \emph{strict} sequence, that is a sequence which eventually
avoids any fixed proper algebraic subgroup
of~$\Gm^{2}(\overline{\mathbb{Q}})$. The following is the main result
of the article.

\begin{introtheorem}[\cref{thm:1}]
  \label{introthm: main thm}
  Let $(\omega_{\ell})_{\ell\ge 1}$ be a strict sequence in  $ \Gm^{2}(\overline{\mathbb{Q}})$ of   nontrivial torsion points. Then
  \begin{displaymath}
\lim_{\ell\to +\infty}    \h( C\cap \omega_{\ell} C)=\frac{2\, \zeta(3)}{3\, \zeta(2)}=0.487175\dots,
  \end{displaymath}
where $\zeta$ is the Riemann zeta function.
\end{introtheorem}

This appearance of the Riemann zeta function is intriguing.  However,
there are already several known relations between heights and special
values of $L$-functions. Just to cite a few, akin links are given by
the Gross--Zagier formula for the height of Heegner points on modular
curves~\cite{GrossZagier} and by the fulfillment of similar properties
like the Northcott one~\cite{PP}. Moreover, such connections have
motivated far-reaching conjectures, like that by Colmez~\cite{Colmez}
and its generalization by Maillot and
R\"ossler~\cite{MaillotRoessler}.

Closer to our setting, the height of projective points can be extended
to projective subvarieties. In the particular case of a projective
hypersurface defined over $\mathbb{Q}$, this notion coincides with the
logarithmic Mahler measure of a primitive defining
polynomial~\cite{DP, Mail}, and there is an active line of research
relating Mahler measures with special values of~$L$-functions, see for
instance~\cite{BZ}.  For example, it follows from a classical result
of Smyth~\cite{Smyth} that the height of the projective line $C$ can
be computed as
\begin{displaymath}
  \h(C)=     \frac{3\sqrt{3}}{4\pi} L(\chi_{-3},2)= 0.323065\dots
\end{displaymath}
for the $L$-function corresponding to the odd Dirichlet character
modulo~$3$. Indeed, this value also shows up in our investigations as
the limit of the height of $P(\omega)$ for~$\omega$ varying in a
strict sequence of torsion points of a certain $1$-dimensional
algebraic subtorus (\cref{exm:3}) and as the conjectural minimal
accumulation value of the set of heights of~$P(\omega)$ for
$\omega \in \mu_{\infty}^{2}$ (\cref{conj:1}).

A formal argument allows to deduce from \cref{introthm: main thm} the
following result, which justifies our previous observation from the
numerical experiments.

\begin{introcorollary}[\cref{cor:3}]
  \label{introcorollary: cor average}
For each $\varepsilon >0$ we have that
\begin{displaymath}
  \lim_{d\to +\infty}  \frac{1}{d^{2}-1} \, \#\Big\{\omega\in \mu_{d}^{2}\setminus\{(1,1)\} \, \Big| \ \Big|\h(C\cap \omega C)
  -\frac{2\, \zeta(3)}{3\, \zeta(2)}\Big|<\varepsilon \Big\} =1.
\end{displaymath}
\end{introcorollary}

Loosely speaking, it asserts that the \emph{typical value}
of the heights corresponding to $d$-torsion points is the mentioned
rational multiple of a quotient of special values of the Riemann zeta
function. Actually, this property holds in greater generality for strict
sequences of finite subsets of torsion points (\cref{thm:2}).

\vspace{\baselineskip}

Let us now outline the proof of \cref{introthm: main thm}.  First we
write the height corresponding to a nontrivial torsion point $\omega =(\omega_1,\omega_2)$
of order~$d$ as the sum
\begin{equation}
  \label{eq:32}
\h(P(\omega))=\sum_{v}\frac{1}{\varphi(d)} \hspace{-1mm} \sum_{k \in (\mathbb{Z}/d\, \mathbb{Z})^{\times}} \hspace{-3mm} \log \max ( 
  |\iota_{v}(\omega_{2}^{k})-\iota_{v}(\omega_{1}^{k})|_{v}, |\iota_{v}(\omega_{2}^{k})-1|_{v},  |\iota_{v}(\omega_{1}^{k})-1|_{v}).
\end{equation}
Here $(\mathbb{Z}/d\, \mathbb{Z})^{\times}$ denotes the group of
modular units and $\varphi$ the Euler totient function, whereas $v$
ranges over the set of places of $\mathbb{Q}$ and
$\iota_{v}\colon \overline{\mathbb{Q}}\hookrightarrow \mathbb{C}_{v}$
is any embedding of $\overline{\mathbb{Q}}$ into the algebraically
closed complete field of $v$-adic~numbers.

In slightly more sophisticated words, the summand corresponding to a
place $v$ in the previous height formula can be viewed as the mean
over the $v$-adic Galois orbit of~$\omega$ of a certain function on
the $v$-adic torus
$\Gm^{2}(\mathbb{C}_{v})=(\mathbb{C}_{v}^{\times})^{2}$ having a
logarithmic singularity at the point~$(1,1)$.  Hence to compute the
limit of this height for $\omega$ going over a strict sequence of
torsion points, we need to prove an equidistribution result for an
adelic family of functions with logarithmic singularities, with a
simultaneous control of all the $v$-adic summands.

In our concrete situation, we achieve this by following an elementary
approach.  On the one hand, we show that the sum of the
non-Archimedean summands in~\eqref{eq:32} can be computed in terms of
the von Mangoldt and the Euler totient functions (\cref{cor:1}). This
implies that for a sequence of torsion points with diverging order,
the non-Archimedean contribution to the height converges to~$0$
(\cref{cor:2}).

On the other hand, the convergence of the Archimedean summand can be
deduced from either the logarithmic equidistribution theorem of
Chambert-Loir and Thuillier~\cite{CLT} or from that of Dimitrov and
Habegger \cite{DimitrovHabegger}. However, it is also possible to
proceed more directly and deduce this convergence from basic results
about cyclotomic polynomials and the standard equidistribution of
Galois orbits of torsion points of tori (\cref{prop:1}).  In any case,
we obtain that for a strict sequence of torsion points, the
Archimedean summand in \eqref{eq:32} tends to the integral
\[
\frac{1}{(2 \pi)^{2}} \int_{(\mathbb{R}/2 \pi \mathbb{Z})^{2}} \log \max (|\eu^{\iu u_{2}}-\eu^{\iu u_{1}}|,|\eu^{\iu u_{2}}-1|,|\eu^{\iu u_{1}}-1|)\, du_{1} du_{2}.
\]
Taking advantage of the symmetries of the integrand, we compute it as the stated quotient involving special values of the Riemann zeta function (\cref{prop:2}), thus completing the proof.

The results presented so far represent the content of
\hyperref[part 1]{Part I} of the article.  The approach therein is
explicit and the arguments employed are as self-contained as
possible. This whole part requires very little background, with the
purpose of making it accessible to non-experts.

\vspace{\baselineskip}

In \hyperref[part 2]{Part II} we raise the technical level of the
exposition to present an interpretation of \cref{introthm: main thm}
from the viewpoint of Arakelov geometry, as developed by Gillet and
Soul\'e \cite{GilletSoule} and extended by Maillot~\cite{Mail}.  This
allows to recover it in a more intrinsic way through the
interplay between arithmetic and convex objects from the Arakelov
geometry of toric varieties, studied by Burgos
Gil, Philippon and the second named author~\cite{BPS} and by the first
named author~\cite{Gualdi}.

Arakelov geometry provides a vast generalization of the notion of
height, from projective points to cycles of an arithmetic variety
equipped with a family of metrized line bundles. In this context, the
height of a point of $\mathbb{P}^{2}(\overline{\mathbb{Q}})$ with
rational coordinates coincides with the (Arakelov) height of its
associated $1$-dimensional subscheme of the projective plane over the
integers ~$\mathbb{P}^{2}_{\mathbb{Z}}$ with respect to the canonical
metrized line bundle~$\canOZ$.

Another distinguished metrized line bundle
on~$\mathbb{P}^{2}_{\mathbb{Z}}$ is the Ronkin metrized line bundle
$\ronOZ$, constructed from the Ronkin function of~$x_0+x_1+x_2$. It is
relevant in the study of the Arakelov geometry of the hypersurface
$\mathscr{C} $ of $\mathbb{P}^{2}_{\mathbb{Z}}$ defined by this linear
polynomial~\cite{Gualdi}.

Since these two metrized line bundles are semipositive, arithmetic
intersection theory allows to define the  height
\begin{equation}
  \label{eq:36}
\h_{\canOZ,\,\ronOZ}(\mathscr{C}) \in \mathbb{R}.
\end{equation}
As a consequence of the results in \hyperref[part 1]{Part I} and the
metric Weil reciprocity law (\cref{prop: metric Weil reciprocity}),
we can show that the limit value in \cref{introthm: main thm}
coincides with this height.

\begin{introtheorem}[\cref{thm: limit is Ronkin height}]
  \label{introthm: thm2}
  Let $(\omega_{\ell})_{\ell\ge 1}$ be a strict sequence in  $ \Gm^{2}(\overline{\mathbb{Q}})$ of   nontrivial torsion points. Then
\[
  \lim_{\ell\to
    +\infty}\h(C\cap \omega_{\ell}C)=\h_{\canOZ,\,\ronOZ}(\mathscr{C}).
\]
\end{introtheorem}

We can reformulate this result in the following suggestive form.
Let $ \mathbb{P}^{2}_{\overline{\mathbb{Z}}}$ be the projective plane over the integral closure of the integers. For each $\ell \ge 1$  write
$\omega_{\ell}=(\omega_{\ell,1},\omega_{\ell,2})$ with
$\omega_{\ell,1}, \omega_{\ell,2}\in \mu_{\infty}$ and consider the
$1$-dimensional integral subscheme
\begin{displaymath}
  Z(x_{0}+x_{1}+x_{2},x_{0}+\omega_{\ell,1}^{-1}x_{1}+\omega_{\ell,2}^{-1}x_{2})  \subset  \mathbb{P}^{2}_{\overline{\mathbb{Z}}}.
\end{displaymath}
It coincides with the closure of $P(\omega_{\ell}) $ in
$ \mathbb{P}^{2}_{\overline{\mathbb{Z}}}$, and so its height with
respect to $\canOZ$ agrees with the height of the point
$P(\omega_{\ell})$~(\cref{rem: heights of points over Qbar}).

The main result of \cite{Gualdi} shows that the height
in~\eqref{eq:36} coincides with the height of the ambient projective
plane with respect to a further Ronkin metrized line bundle
(\cref{prop: height of curve as Ronkin height and mixed
  integral}). Combining these results yields the following one. 

  \begin{introcorollary}[\cref{cor:5}]
    \label{cor:4}
    With notation as in \cref{introthm: thm2}, 
  \begin{displaymath}
    \lim_{\ell\to +\infty}\h_{\canOZ}(Z(x_{0}+x_{1}+x_{2},x_{0}+\omega_{\ell,1}^{-1}x_{1}+\omega_{\ell,2}^{-1}x_{2}))=
    \h_{\canOZ,\,\ronOZ,\,\ronOZ}(\mathbb{P}^{2}_{\mathbb{Z}}).
  \end{displaymath}
  \end{introcorollary}

  This limit formula can be considered as a particular case of a
  conjectural arithmetic analogue of the classical geometric fact that
  for a family of $n$-many line bundles on an $n$-dimensional
  algebraic variety, the cardinality of the zero set of a generic
  $n$-tuple of their global sections coincides with the degree of the
  variety with respect to these line bundles.  Indeed, \cref{cor:4}
  shows that a height of the zero set of a pair of global sections of
  $\mathscr{O}(1)$ with a certain arithmetic feature approaches a
  related height of the ambient space as these global sections becomes
  more and more ``generic''.


  \vspace{\baselineskip}

  Let us give some hints on how to recover \cref{introthm: main thm}
  from \cref{introthm: thm2} through the Arakelov geometry of toric
  arithmetic varieties, applied to the case of~$\mathbb{P}^{2}_{\mathbb{Z}}$.
  A fundamental ingredient of this
  theory is the classification of semipositive toric metrized line
  bundles in terms of certain concave functions on a vector space
  \cite{BPS}. Within this classification, the metrized line bundle
  $\canOZ$ on $\mathbb{P}^{2}_{\mathbb{Z}}$ corresponds to the
  piecewise linear concave
  function~$\Psi \colon \mathbb{R}^{2} \to \mathbb{R}$ defined~as
\begin{displaymath}
  \Psi(u_{1},u_{2}) =\min(0,u_{1},u_{2}). 
\end{displaymath}
On the other hand, we associate to $C$ its Archimedean amoeba
$\mathscr{A}$, that is the tentacle-shaped subset of~$\mathbb{R}^2$
given as the tropicalization the complex line $C(\mathbb{C})$
(\cref{figure amoeba 1+x+y}).

Applying the results of \cite{BPS, Gualdi} we can express the height
in~\eqref{eq:36} in terms of convex geometry (\cref{prop: height of
  curve as Ronkin height and mixed integral}).  Combining this with
the results of Passare and Rullg{\aa}rd on Ronkin functions and their
associated Monge--Amp\`ere measures~\cite{PR}, we deduce the following
relation between the considered height and the average of the concave
function $\Psi$ on the amoeba~$\mathscr{A}$.

\begin{introtheorem}[\cref{cor: limit is integral over amoeba}]
  \label{thm:intro_amoeba}
With notation as above,
\[
\h_{\canOZ,\,\ronOZ}(\mathscr{C})
=
-\frac{1}{\vol(\mathscr{A})}\int_{\mathscr{A}}\Psi(u_{1},u_{2}) \ d u_{1} d u_{2}.
\]
\end{introtheorem}

Finally we compute the integral in this equality (\cref{prop:
  zeta values and amoebas}) and thus recover \cref{introthm: main thm}
from \hyperref[introthm: thm2]{Theorems~\ref*{introthm: thm2}} and
\hyperref[thm:intro_amoeba]{\ref*{thm:intro_amoeba}}.

\vspace{\baselineskip}

Summing up, this article deals with a specific problem that is both
suitable for a down-to-earth analysis and for numerical
experimentation, as well as appearing as an instance of a much more
general situation.

The concreteness of the results in \hyperref[part 1]{Part I} allows
for an elementary and self-contained treatment that requires little
background in algebraic geometry and number theory. Moreover, the
considered problem is well-suited for computations that allow to
visualize the results and to suggest further questions, as done in
\cref{sec:visualizing-result}.

Placing the problem within the context of Arakelov geometry as in
\hyperref[part 2]{Part II}, our investigation takes a deeper
connotation, pointing towards an asymptotic version of the arithmetic
B\'ezout theorem. Indeed, after \cref{cor:4} it seems reasonable to
expect that the limit of the height of the solution set of a system of
polynomial equations sharing a certain arithmetic feature coincides
with a height of the ambient space, see for instance
\cref{conj: conjecture for curves}.  Establishing this extension would
require a substantial technical effort to solve the adelic logarithmic
equidistribution problem that arises. We plan to address this in a
subsequent article.

In a more speculative spirit, it would be also interesting to explore
whether some of these toric heights can be linked to integrals of
piecewise linear functions over amoebas like in
\cref{thm:intro_amoeba}, and try to compute them in terms of special
functions.

\vspace{\baselineskip}
\noindent\textbf{Acknowledgments.}
We are grateful to Fabien Pazuki for his encouragement to write
this article.  We also thank Francesco Amoroso, Jos\'e Ignacio Burgos Gil, Pietro
Corvaja, Xavier Guitart, Joaquim Ortega-Cerd\`a, Riccardo Pengo, Lukas
Pottmeyer, Lukas Prader and Alain Yger for several enjoyable discussions, and for
providing enthusiastic inputs and questions.  Our research was
helped by computer exploration using the open-source mathematical
software \texttt{SageMath} \cite{sagemath}.

Part of this work was done while we met at Centre de Recerca
Matem\`atica, Universitat de Barcelona and Universit\"at Regensburg.
We thank these institutions for their hospitality.

Roberto Gualdi was supported by the Alexander von Humboldt Foundation, the collaborative research center SFB 1085
``Higher Invariants'' funded by the Deutsche Forschungsgemeinschaft,
and the Fundaci\'o Ferran Sunyer i Balaguer. Mart\'{\i}n Sombra was
partially supported by the Spanish MICINN research project
PID2019-104047GB-I00, and by the Spanish AEI project CEX2020-001084-M
of the Severo Ochoa and Mar\'{\i}a de Maeztu program for centers and
units of excellence in R\&D.

\numberwithin{equation}{section}
\numberwithin{figure}{section}

\vspace{\baselineskip}
\phantomsection\label{part 1}
\begin{center} 
\textbf{\textsc{PART I}}
\end{center}

\addcontentsline{toc}{chapter}{\hspace{2.9mm} Part I}

This part is dedicated to the proof of our main results.  We start in
\cref{sec:preliminaries} by setting the notation and recalling the
principal actors of our statements, like strict sequences of torsion points of algebraic tori and the height of a projective point.

In \cref{sec: upper bounds} we give a sharp bound for the height of
the intersection of the projective line defined by the linear
polynomial $x_0+x_1+x_2$ and its translate by a nontrivial torsion
point.

The following four sections culminate with~\cref{thm:1}, which shows
that the limit of such heights for torsion points in a strict sequence
can be computed in terms of special values of the Riemann zeta
function. In \cref{sec:distribution-height} we also extend this result
to strict subsets of torsion points, to show that most of the
corresponding heights concentrate near that limit value.

The treatment throughout is down-to-earth and as self-contained as
possible, except maybe for \cref{sec:interm-sequ-tors}, which might be skipped at a first read.
Its goal is to extend the previous study to torsion points in proper subgroups of
the $2$-dimensional torus.

Finally, \cref{sec:visualizing-result} illustrates the results through
numerical calculations and graphical plottings made with the
\texttt{SageMath} notebook~\cite{GS_Notebook} accompanying this
article.

\section{Preliminaries}
\label{sec:preliminaries}

Here we discuss some of the basic constructions and properties
concerning algebraic tori and canonical heights on projective
spaces.
Our treatment is far from being complete, and we refer the
interested reader to \cite[Chapters 1 and 3]{BG} for the proofs and
more details about the explained facts.

We denote by $\overline{\mathbb{Q}}$ an algebraic closure of the field
of rational numbers~$\mathbb{Q}$.  For an integer $d\ge 1$ we write $\mu_{d}$ for the subgroup of $\overline{\mathbb{Q}}^{\times}$ of
$d$-roots of unity, and $\mu_{d}^{\circ}$ for its subset of
primitive $d$-roots. We also denote by $\mu_{\infty}$ the subgroup
of $\overline{\mathbb{Q}}^{\times}$ of all roots of unity.

For $n \ge 0$ we denote by
$\Gm^{n}(\overline{\mathbb{Q}})=(\overline{\mathbb{Q}}^{\times})^{n}$
the $n$-dimensional (split, algebraic) torus
over~$\overline{\mathbb{Q}}$. It is a group under coordinate-wise
multiplication, with torsion subgroup equal to~$\mu_{\infty}^{n}$.
For $d\geq 1$ and $\zeta\in \mu_{d}^{\circ}$, each $d$-torsion point
$\omega \in \mu_{d}^{n}$ of this torus can be uniquely written as
\begin{equation}
  \label{eq:44}
  \omega=(\zeta^{c_{1}}, \dots, \zeta^{c_{n}})
\end{equation}
with $c_{i}\in \{0,\dots, d-1\}$ for~$i=1,\dots, n$. Its order is~$\ord(\omega)=d/\gcd(c_{1},\dots, c_{n},d)$.

Given an algebraic subset $V$ of~$\Gm^{n}(\overline{\mathbb{Q}})$, a
sequence $(\gamma_{\ell})_{\ell\ge 1}$ in $V$ is \emph{strict} if it
eventually avoids any fixed algebraic subgroup $H$ of this torus not
containing~$V$, that is, if there is $\ell_{0}\ge 1$ such that
$\gamma_{\ell} \notin H$ for all~$\ell\ge \ell_{0}$. Our main case of
interest is $V=\Gm^{n}(\overline{\mathbb{Q}})$.

For each vector $a=(a_{1},\dots, a_{n})\in \mathbb{Z}^{n}$ we denote by $\chi^{a}$ the
corresponding character, so that for
$\gamma=(\gamma_{1},\dots, \gamma_{n})\in
\Gm^{n}(\overline{\mathbb{Q}})$ we have that
\begin{displaymath}
\chi^{a}(\gamma)=\gamma_{1}^{a_{1}}\cdots \gamma_{n}^{a_{n}}.
\end{displaymath}
Any proper algebraic subgroup of $\Gm^{n}(\overline{\mathbb{Q}})$ is
contained in an algebraic subgroup of codimension 1, which are those
defined by a binomial of the form $\chi^{a}-1$ with
$a\in \mathbb{Z}^{n}\setminus \{(0,\dots, 0)\}$. Hence a sequence
$(\gamma_{\ell})_{\ell\ge 1}$ in $\Gm^{n}(\overline{\mathbb{Q}})$ is
strict if and only if for every such $a$ there is $\ell_{0}\ge 1$ with
\begin{equation}
  \label{eq:22}
\chi^{a}(\gamma_{\ell})\ne 1 \quad \text{ for all } \ell\ge \ell_{0}.
\end{equation}

We denote by $\mathbb{P}^{n}(\overline{\mathbb{Q}})$ the
$n$-dimensional projective space over~$\overline{\mathbb{Q}}$.  For a
homogeneous polynomial
$f\in \overline{\mathbb{Q}}[x_{0},\dots, x_{n}]$ we denote by $Z(f)$
the algebraic subset of this projective space that it defines, and for
$\gamma\in \Gm^{n}(\overline{\mathbb{Q}})$ we consider the
\emph{twist} of $f$ by $\gamma$, which is the homogeneous polynomial
$\gamma^{*}f \in \overline{\mathbb{Q}}[x_{0},\dots, x_{n}]$ defined as
\begin{displaymath}
  \gamma^{*}f(x_{0},\dots, x_{n})= f(x_{0},\gamma_{1}x_{1},\dots, \gamma_{n}x_{n}).
\end{displaymath}
For $\gamma\in \Gm^{n}(\overline{\mathbb{Q}})$ we also consider the associated
\emph{translation map}
\begin{equation}
  \label{eq:30}
  \mathbb{P}^{n}(\overline{\mathbb{Q}})\longrightarrow
\mathbb{P}^{n}(\overline{\mathbb{Q}}), \quad
\xi=[\xi_{0}:\xi_{1}:\dots:\xi_{n}]\longmapsto \gamma\, \xi\coloneqq [\xi_{0}:\gamma_{1}\, \xi_{1}:\cdots:\gamma_{n}\, \xi_{n}].
\end{equation}
The translation by   $\gamma$ of the zero set of
$f$ coincides with the zero set of the twist~$(\gamma^{-1})^{*} f$, that is
\begin{equation}
  \label{eq:5}
\gamma Z(f)=Z((\gamma^{-1})^*f)).
\end{equation}

We next recall the basic notations and properties of canonical heights
in projective spaces.  We denote by $M_{\mathbb{Q}}$ the set of places
of $\mathbb{Q}$, that is the set of equivalence classes of nontrivial
absolute values on this field with respect to the topology that they
define. By Ostrowski's theorem, these places are represented by the
usual Archimedean absolute value on $\mathbb{Q}$ and by the $p$-adic
ones as $p$ ranges over the primes of~$\mathbb{Z}$, and so
$M_{\mathbb{Q}}$ can be identified with the set made of the symbol
$\infty$ and these primes.  For each $v\in M_{\mathbb{Q}}$ we denote
by~$|\cdot|_{v}$ its representative, and by $\mathbb{Q}_{v}$ the
completion of $\mathbb{Q}$ with respect to this absolute value. When
$v=\infty$ this complete field coincides with the field of real
numbers~$\mathbb{R}$, whereas when $v=p$ is a prime it is the field of
$p$-adic~numbers.

More generally, for a number field $K$ we denote by $M_{K}$ the set of its
places. For each $w\in M_{K}$ there is a unique $v\in M_{\mathbb{Q}}$
such that $|\cdot|_{v}$ extends to a (unique) absolute value on $K$ in
the equivalence class of $w$, a relation that is indicated by~$w\mid v$. We denote by $|\cdot|_{w}$ this absolute value on $K$, and
by $K_{w}$ the corresponding completion of~$K$.
For each~$v\in M_{\mathbb{Q}}$, the set of places of $K$ extending $v$ is
finite and moreover the sum of the corresponding local degrees
coincides with the degree of the extension:
    \begin{equation}
      \label{eq:16}
      \sum_{w\mid v} [K_{w}:\mathbb{Q}_{v}] = [K:\mathbb{Q}]. 
    \end{equation}
   
Let now
    $\xi=[\xi_{0}:\cdots:\xi_{n}]\in
    \mathbb{P}^{n}(\overline{\mathbb{Q}})$. For each  $v\in M_{\mathbb{Q}}$,
    the \emph{$v$-adic height} of the vector of homogeneous
    coordinates of this projective point is defined as
\begin{displaymath}
  \h_{v}(\xi_{0},\dots, \xi_{n}) = \sum_{w\mid v} \frac{[K_{w}:\mathbb{Q}_{v}]}{[K:\mathbb{Q}]}\, \log\max (|\xi_{0}|_{w},\dots, |\xi_{n}|_{w})
\end{displaymath}
for any number field $K$ containing all these coordinates. Its value does
not depend on the choice of this number field, and it vanishes for all
but a finite number of~$v$'s. The \emph{(canonical) height} of $\xi$
is defined as the sum of these local heights:
\begin{equation}
  \label{eq:4}
\h(\xi)= \sum_{v\in M_{\mathbb{Q}}} \h_{v}(\xi_{0},\dots, \xi_{n}).
\end{equation}
Thanks to the product formula, its value does not depend on the choice
of homogeneous coordinates.

In general $\h(\xi)\ge 0$ and, by Kronecker's theorem, $\h(\xi)=0$ if
and only if the point can be written as $\xi=[\xi_{0}:\dots:\xi_{n}]$
with $\xi_{j}$ equal to either $0$ or a root of unity.

\begin{remark}
  \label{rem:4}
  The height of a projective point is a measure of the complexity of
  its representation. For instance, a point
  $\xi\in \mathbb{P}^{n}(\overline{\mathbb{Q}})$ with rational
  homogeneous coordinates can be written as
  $\xi=[\xi_{0}:\dots:\xi_{n}]$ with coprime integer~$\xi_{i}$'s. In
  this situation, the formula in \eqref{eq:4} boils~down~to
\begin{displaymath}
  \h(\xi)=\log \max (|\xi_{0}|_{\infty},\dots, |\xi_{n}|_{\infty}),
\end{displaymath}
which gives the maximal bit-length of these integers.
\end{remark}

We next give a Galois-theoretic formula for the local heights
of the vector of homogeneous coordinates of a projective point. Set
$\Xi=(\xi_{0},\dots, \xi_{n})\in \overline{\mathbb{Q}}^{n+1} $ and let
\begin{displaymath}
O(\Xi)=  \Gal (\overline{\mathbb{Q}}/\mathbb{Q}) \cdot \Xi  
\end{displaymath}
be the orbit of this vector under the coordinate-wise action
of the absolute Galois group of~$\mathbb{Q}$. It is a finite subset of~$ \overline{\mathbb{Q}}^{n+1}$.

For each $v\in M_{\mathbb{Q}}$ we choose an algebraic closure of the
complete field~${\mathbb{Q}}_{v}$, and we denote by $\mathbb{C}_{v}$
its completion with respect to the unique extension of $|\cdot|_{v}$
to it. This field is both algebraically closed and complete with
respect to the induced absolute value, that we also denote
by~$|\cdot|_{v}$.

We also choose an embedding
    \begin{equation}
      \label{eq:12}
      \iota_{v}\colon \overline{\mathbb{Q}}\longhookrightarrow \mathbb{C}_{v},
    \end{equation}
    which induces an embedding
      $\overline{\mathbb{Q}}^{n+1}\hookrightarrow
      \mathbb{C}_{v}^{n+1}$ that we denote with the same symbol. The
      \emph{$v$-adic Galois orbit} of $\Xi$ is then defined as the
      image of $O(\Xi)$ under it, namely
    \begin{displaymath}
      O(\Xi)_{v}=\iota_{v}(O(\Xi)).
    \end{displaymath}
    It is a finite subset of $\mathbb{C}_{v}^{n+1}$ with the same
    cardinality of $O(\Xi)$, and which does not depend on the choice
    of~$\iota_{v}$.
    
    \begin{proposition}
      \label{prop:3}
      Let~$\xi\in \mathbb{P}^{n}(\overline{\mathbb{Q}})$ and
      $\Xi\in \overline{\mathbb{Q}}^{n+1}$ a corresponding vector of homogeneous
      coordinates. Then for each $v\in M_{\mathbb{Q}}$
\begin{displaymath}
  \h_{v}(\Xi)= \frac{1}{\# O(\Xi)_{v}} \sum_{\Lambda\in O(\Xi)_{v}}\log\max (|\Lambda_{0}|_{v},\dots, |\Lambda_{n}|_{v})      
\end{displaymath}
and so 
\begin{displaymath}
  \h(\xi)= \frac{1}{\# O(\Xi)} \sum_{v\in M_{\mathbb{Q}}}   \sum_{\Lambda\in O(\Xi)_{v}}\log\max (|\Lambda_{0}|_{v},\dots, |\Lambda_{n}|_{v}).      
\end{displaymath}
In particular, the height is invariant under the action of the
absolute Galois group~of~$\mathbb{Q}$.
    \end{proposition}
    
    \begin{proof}
      Let $K\subset \overline{\mathbb{Q}}$ be a finite Galois
      extension of $\mathbb{Q}$ containing all the coordinates of
      $\Xi$, and denote by~$G$ its Galois group. For
      $v\in M_{\mathbb{Q}}$ denote by $M_{K,v}$ the set of places of
      $K$ extending~$v$, and by $w_{0} \in M_{K,v}$ the place
      represented by the absolute value on $K$ induced by the absolute
      value of $\mathbb{C}_{v}$ through the embedding $\iota_{v}$ in~\eqref{eq:12}.

      The group $G$ has an action on the finite set~$M_{K,v}$,
        that can be defined by considering for each pair $\sigma\in G$ and
        $w\in M_{K,v}$ the place $\sigma(w)\in M_{K,v}$
        represented by the absolute value on $K$ given~by
      \begin{displaymath}
        |\alpha|_{\sigma(w)}=|\sigma(\alpha)|_{w}\quad\forall\alpha\in K.
      \end{displaymath}
 By \cite[Chapter II, Proposition 9.1]{Neukirch:ant}, this action
      is transitive.

      The $\mathbb{Q}$-automorphism $\sigma\colon K\to K$ extends to a
      $\mathbb{Q}_{v}$-isomorphism~$ K_{w} \to K_{\sigma(w)}$, and
      so~$[K_{\sigma(w)}:\mathbb{Q}_{v}]=[K_{w}:\mathbb{Q}_{v}]$.
      Since the action is transitive, the local degrees corresponding
      to the places in $M_{K,v}$ coincide. By the formula
      in~\eqref{eq:16}, this implies~that
      \begin{displaymath}
        \frac{[K_{w}:\mathbb{Q}_{v}]}{[K:\mathbb{Q}]}=\frac{1}{\# M_{K,v}}.
      \end{displaymath}
      Moreover, by the orbit-stabilizer theorem one also has that
      $\# M_{K,v} = {\#G}/{\#G_{w_{0}}}$ where $G_{w_{0}}$ denotes the
      stabilizer of the place~$w_{0}$. Writing $\Xi=(\xi_{0},\dots, \xi_{n})$ we have
        \begin{displaymath}
          \h_{v}(\Xi) =  \sum_{w\in M_{K,v}}\frac{1}{\# M_{K,v}}  \log\max_{j} |\xi_{j}|_{w}
          =\frac{1}{\#G} \sum_{\sigma\in G} \log\max_{j} |\xi_{j}|_{\sigma(w_{0})}.
        \end{displaymath}
On the other hand,  $G$ also acts transitively on the Galois orbit of $\Xi$ and so 
      \begin{displaymath}
        \frac{1}{\# O(\Xi)_{v}}=\frac{1}{\# O(\Xi)}=  \frac{\#G_{\Xi}}{\#G},
        \end{displaymath}
where $G_{\Xi}$ is the stabilizer of this vector. Hence
        \begin{multline*}
          \frac{1}{\#G} \sum_{\sigma\in G} \log\max_{j} |\xi_{j}|_{\sigma(w_{0})}=
          \frac{1}{\#G} \sum_{\sigma\in G} \log\max_{j} |(\iota_{v} \circ \sigma)(\xi_{j})|_{v}
\\        = \frac{1}{\# O(\Xi)_{v}} \sum_{\Lambda\in O(\Xi)_{v}}\log\max_{j} |\Lambda_{j}|_{v},
        \end{multline*}
proving the first statement. The other claims follow directly.       
    \end{proof}

    \begin{remark}
      \label{rem:3}
      A version of this result in the more general setting of Arakelov
      geometry appears in \cite[Proposition 2.3]{BPRS}.
   \end{remark}

\section{The range of the height}\label{sec: upper bounds}

In this section we start our study of the height of the intersection
of the line
\begin{equation*} 
  C=Z(x_{0}+x_{1}+x_{2}) \subset \mathbb{P}^{2}(\overline{\mathbb{Q}})
\end{equation*}
with its translate $\omega C$ by a torsion point $\omega$
of~$\Gm^{2}(\overline{\mathbb{Q}})$. As in~\eqref{eq:5}, this
translate coincides with the zero set of the twist by $\omega^{-1}$ of
the linear polynomial~$x_{0}+x_{1}+x_{2}$, that is
\begin{displaymath}
\omega C=Z(x_{0}+\omega_{1}^{-1}x_{1}+\omega_{2}^{-1}x_{2}). 
\end{displaymath}
Hence the intersection $C\cap \omega C$ coincides with the solution
set of the system of linear equations
\begin{displaymath}
x_{0}+x_{1}+x_{2}= x_{0}+\omega_{1}^{-1}x_{1}+\omega_{2}^{-1}x_{2}= 0.
\end{displaymath}
When $\omega$ in nontrivial, that is when $\omega\ne (1,1)$, it
consists of the point
\begin{equation}
  \label{eq:7}
P(\omega)= [\omega_{2}^{-1}-\omega_{1}^{-1}:1-\omega_{2}^{-1}: \omega_{1}^{-1}-1]  \in \mathbb{P}^{2}(\overline{\mathbb{Q}}).
\end{equation}
Its height depends nontrivially on~$\omega$, as the next
example shows. 

\begin{example}\label{ex: height of some solutions}
  For $d\geq 2$ let $\zeta \in \mu_{d}^{\circ}$ be a primitive
  $d$-root of unity, and consider the torsion point~$\omega_{d}=(\zeta,\zeta^{2})$.
  The corresponding intersection point
  can be written as~$P(\omega_{d})=[1:-1-\zeta:\zeta]$, and the
  Galois orbit of the vector of these homogeneous coordinates is 
  \begin{displaymath}
    \{(1,-1-\zeta^{k},\zeta^{k}) \mid k\in (\mathbb{Z}/d
  \mathbb{Z})^{\times}\}.
  \end{displaymath}
  This is a finite set of cardinality~$\varphi(d)$, where $\varphi$ is
  the Euler totient function.

  For every prime $p$ and every
  $ k\in (\mathbb{Z}/k\, \mathbb{Z})^{\times}$ we have that
  \begin{displaymath}
|\iota_{p}(\zeta^{k})|_{p}=1 \and
 |\iota_{p}(\zeta^{k}+1)|_p=|\iota_{p}(\zeta^{k})+1|_p\leq 1     
  \end{displaymath}
  because of the ultrametric inequality for the $p$-adic absolute
  value.  Hence the formula for the height of $P(\omega_{d})$ from
  \cref{prop:3} reduces to its Archimedean contribution, namely
\begin{multline*}
\h(P(\omega_{d}))=
\frac{1}{\varphi(d)}\sum_{k\in(\mathbb{Z}/d\mathbb{Z})^\times}\log\max(1,|\eu^{2\pi ik/d}+1|_{\infty}) \\ =
\frac{1}{\varphi(d)}\sum_{k\in(\mathbb{Z}/d\mathbb{Z})^\times}\log\max(1,\sqrt{2+2\cos(2\pi k/d)}).
\end{multline*}
It follows that
\begin{displaymath}
\h(P(\omega_{2}))=\h(P(\omega_{3}))=0, \ \h(P(\omega_{4}))=\frac{1}{2}\log(2),\ \h(P(\omega_{5}))=\frac{1}{4}\log\bigg(\frac{3+\sqrt{5}}{2}\bigg), \ \dots  
\end{displaymath}
\end{example}

The main result of this section is the following, in which we
establish the range of values of the height of the intersection point
$P(\omega)$ and determine the extremal cases.

\begin{proposition}\label{prop: bounds for the height} 
  Let $\omega \in \Gm^{2}(\overline{\mathbb{Q}})$ be a nontrivial torsion point. Then the corresponding height verifies the inequalities
  \begin{math}
{0\leq\h(P(\omega))\leq\log(2)}.   
  \end{math}
  The lower bound is attained exactly when
\[
{\omega}\in\{(1,\zeta),(\zeta,1),(\zeta,\zeta)\mid \zeta\in\mu_{\infty}\setminus\{1\}\}\cup\{(\zeta,\zeta^2) \mid \zeta\in \mu_{3}^{\circ}\},
\]
whereas the upper bound is attained  exactly when
\[
{\omega}\in\{(-1,\zeta),(\zeta,-1),(\zeta,-\zeta)\mid\zeta\in\mu_{\infty}\text{ with }\ord(\zeta)\neq2^k\text{ for all }k\geq0\}.
\]
\end{proposition}

Before proving the proposition, we choose the vector of homogeneous coordinates of
$P(\omega)$ given by
\begin{equation}
  \label{eq:48}
  \mathcal{P}(\omega)=
(\omega_{2}^{-1}-\omega_{1}^{-1},1-\omega_{2}^{-1}, \omega_{1}^{-1}-1) \in \overline{\mathbb{Q}}^{3}.
\end{equation}
The next lemma
gives a formula for its local heights. 

\begin{lemma}
  \label{lem:5}
  Set~$d=\ord(\omega)$. Then for each $v\in M_{\mathbb{Q}}$ we have
  that
\begin{displaymath}
  \h_{v}(\mathcal{P}(\omega))=\frac{1}{\varphi(d)}\sum_{k \in (\mathbb{Z}/d\, \mathbb{Z})^{\times}}  \log \max ( 
  |\iota_{v}(\omega_{2}^{k})-\iota_{v}(\omega_{1}^{k})|_{v}, |\iota_{v}(\omega_{2}^{k})-1|_{v},  |\iota_{v}(\omega_{1}^{k})-1|_{v}),
\end{displaymath}
where $\iota_v$ is the  embedding in~\eqref{eq:12}. 
\end{lemma}

\begin{proof}
  Both $\omega_{1}$ and $\omega_{2}$ are contained in the $d$-th
  cyclotomic extension of~$\mathbb{Q}$, and so are the coordinates
  of~$\mathcal{P}(\omega)$. Hence the Galois orbit of this vector
  coincides with its orbit under the action of the Galois group of
  this cyclotomic extension.  The latter is isomorphic
  to~$(\mathbb{Z}/d\, \mathbb{Z})^{\times}$, and by \eqref{eq:44} the
  action of each element $k$ of this group maps $\omega_{i}$ to
  $\omega_{i}^{k}$ for~$i=1,2$.  Hence the Galois orbit of
  $\mathcal{P}(\omega)$ writes down as
\begin{multline*} 
  O(\mathcal{P}(\omega))=\{ (\omega_{2}^{-k}-\omega_{1}^{-k},1-\omega_{2}^{-k}, \omega_{1}^{-k}-1) \mid k\in (\mathbb{Z}/d\mathbb{Z})^{\times}\}\\
  =\{ (\omega_{2}^{k}-\omega_{1}^{k},1-\omega_{2}^{k}, \omega_{1}^{k}-1) \mid k\in (\mathbb{Z}/d\mathbb{Z})^{\times}\}.
\end{multline*}
The elements in this last set are pairwise distinct as $k$
ranges in~$(\mathbb{Z}/d\mathbb{Z})^{\times}$. Indeed any element $k$
in the stabilizer of $\mathcal{P}(\omega)$ has to satisfy
$\omega^k=\omega$ and then it must be trivial in
$(\mathbb{Z}/d\mathbb{Z})^{\times}$ as a consequence of the hypothesis
that $\omega$ has order~$d$.  The statement then follows from
\cref{prop:3}.
\end{proof}

We will need two further auxiliary results. The first is the classical
formula for the value of a cyclotomic polynomial at~1, which will also
play an important role in \cref{sec: non-Archimedean height,sec: limit
  of archimedean heights}.  Its proof is elementary and can be found
in~\cite[page~74]{LangAlgebraicNumberTheory}.

\begin{lemma}
  \label{lem:2}
For $d \ge 2$ let  $\Phi_{d}$ be the $d$-th cyclotomic polynomial. Then
\begin{displaymath}
\Phi_d(1)=
\begin{cases}
p&\text{if $d$ is a power of a prime $p$},\\
1&\text{otherwise}.
\end{cases}
\end{displaymath}
\end{lemma}

The second auxiliary result gives the $p$-adic distance of a root of unity to
the~point~$1$.

\begin{lemma}
  \label{lem:1}
  Let $p$ be a prime and~$d\ge 2$. Then for all
  $\zeta\in \mu_{d}^{\circ}$ we have that
  \begin{displaymath}
|\iota_{p}(\zeta)-1|_{p}    =
\begin{cases}
 p^{-1/\varphi(d)}&\text{if } d \text{ is a power of } p,\\
1&\text{otherwise}.
\end{cases}
  \end{displaymath}
\end{lemma}

\begin{proof}
  The Galois conjugates of $\zeta$ are the elements of the form
  $\zeta^{k}$ for~$k\in (\mathbb{Z}/d\mathbb{Z})^{\times}$. For each
  $k$ we have that $\zeta^{k}-1= (\zeta^{k-1}+\cdots+1)(\zeta-1)$ and
  so the ultrametric inequality implies that
  \begin{displaymath}
    |\iota_{p}(\zeta^{k})-1|_{p}\le |\iota_{p}(\zeta)-1|_{p}.
  \end{displaymath}
  By symmetry, the reverse inequality also holds. Hence 
  $ |\iota_{p}(\zeta^{k})-1|_{p}= |\iota_{p}(\zeta)-1|_{p}$ for all~$k$, and so 
  \begin{displaymath}
    |\iota_{p}(\zeta)-1|_{p}^{\varphi(d)}=\prod_{k}  |\iota_{p}(\zeta^{k})-1|_{p}=|\Phi_{d}(1)|_{p}.
  \end{displaymath}
The result follows then from  \cref{lem:2}.
\end{proof}

\begin{proof}[Proof of \cref{prop: bounds for the height}]
  The lower bound comes from the fact that the height is nonnegative.
  By Kronecker's theorem, the height of $P(\omega)$ vanishes if and only if this
  point can be written as $[\xi_{0}:\xi_{1}:\xi_{2}]$ with
  $ \xi_{j}\in \mu_{\infty}\cup \{0\}$ for all~$j$.  The
  determination of such points in $C$ implies that the considered
  height is equal to  $0$ if and only if
  \begin{equation*}
P(\omega)\in\{[1:-1:0],[1:0:-1],[0:1:-1]\}\quad\text{or}\quad P(\omega)\in\{[1:\zeta:\zeta^{2}] \mid \zeta\in \mu_{3}^{\circ}\}.    
  \end{equation*}
  A comparison with the explicit form of $P(\omega) $ in \eqref{eq:7}
  shows that the first alternative holds if and only if $\omega$ is
  equal to $(1,\zeta),(\zeta,1)$ or $(\zeta,\zeta)$
  with~$\zeta\in \mu_{\infty}\setminus \{1\}$.
  For the
    second, we have that $P(\omega)=[1:\zeta:\zeta^{2}]$ with
    $\zeta\in \mu_{3}^{\circ}$ if and only if
    \begin{equation}
      \label{eq:9}
      \omega_2^{-1}-\omega_1^{-1}\ne 0, \quad \frac{1-\omega_{2}^{-1}}{\omega_2^{-1}-\omega_1^{-1}}=\zeta,
      \quad \frac{\omega_{1}^{-1}-1}{\omega_2^{-1}-\omega_1^{-1}}=\zeta^{2}.
    \end{equation}
    This implies that~$1+\zeta\omega_1^{-1}+\zeta^2\omega_2^{-1}=0$
    and so~$[1:\zeta\omega_1^{-1}:\zeta^2\omega_2^{-1}] \in
    C(\overline{\mathbb{Q}})$.  Our previous knowledge of the points
    of $C$ with homogenous coordinates that are roots of unity implies
    that either $\omega_{1}=\omega_{2}=1$ or~$\omega_{1}=\zeta^{2}$,~$\omega_{2}=\zeta$.
    Since the second
    possibility is the only one satisfying the conditions in~\eqref{eq:9}, this proves our claim concerning the points of $C$
    attaining the lower bound.

  Now let~$v \in M_{\mathbb{Q}}$.  Using the fact that
  $|\iota_{v}(\omega_{1})|_{v}=|\iota_{v}(\omega_{2})|_{v}=1$ for every~$v$, the triangular
  inequality when $v=\infty$ and the ultrametric inequality otherwise,
  we deduce from \cref{lem:5} that
  \begin{equation}
    \label{eq:6}
\h_{v}(\mathcal{P}(\omega))\le 
    \begin{cases}
      \log (2) & \text{ if } v=\infty, \\
      0 & \text{ if } v\ne \infty,
    \end{cases}
  \end{equation}
  which readily implies the upper bound.

This upper bound is attained if and only if all the
    inequalities in \eqref{eq:6} are in fact equalities.  When~$v=\infty$, this requirement forces in particular that the summand for $k=1$  in the
    formula in \cref{lem:5} coincides with~$\log(2)$, which
    happens if and only if $\omega$ is of the form
  \begin{equation}
    \label{eq:38}
(-1,\zeta), \quad (\zeta,-1) \quad \text{or} \quad (\zeta,-\zeta) \quad \text{ with } \zeta\in \mu_{\infty}.
\end{equation}
These torsion points have even order and so the indexes in that
formula are necessarily odd numbers, which implies that all the
summands coincide with~$\log(2)$. We conclude that for the Archimedean
place the inequality in \eqref{eq:6} is an equality exactly when
$\omega$ is of the form described in~\eqref{eq:38}.

Set~$d=\ord(\zeta)$. To realize the upper bound, for each of the above
possibilities for~$\omega$ we have to furthermore ensure that for
every prime $p$ we have that
\begin{displaymath}
\max(|2|_p, |\iota_{p}(\zeta^{k})-1|_p,|\iota_{p}(\zeta^{k})+1|_p)=1 \quad \text{ for all } k\in (\mathbb{Z}/d\mathbb{Z})^{\times}.
\end{displaymath}
This is nontrivial only when~$p=2$, in which case it is equivalent to
the condition that $|\iota_{2}(\zeta^{k})-1|_2=1$ for all~$k$, because
of the ultrametric property.  When $d=1$ this condition fails, whereas
when $d\ge 2$ it holds exactly when $d$ is not a
power of~$2$, by \cref{lem:1}. This completes the proof.
\end{proof}

\begin{remark}
  \label{rem:2}
  In the more general context of Arakelov geometry, the results of
  \cite{MartinezSombra} and \cite[Chapter 5]{GualdiThesis} provide
  upper bounds for the height of a complete intersection. However,
  neither of them is sharp in our particular situation.
\end{remark}

\section{The negligibility of the non-Archimedean heights}\label{sec:
  non-Archimedean height}

We now turn to our main object of study, that is the limit value of
the height of the intersection of the line $C$ with its translates by
torsion points in a strict sequence. Here we focus on the
non-Archimedean contribution to these heights, achieving an explicit
expression for it.

We first notice that these non-Archimedean local heights can
be nonvanishing.  In fact, for any non-Archimedean place of
$\mathbb{Q}$ it is easy to construct choices of~${\omega}$ for which
the corresponding local height is nontrivial, as the next example
shows.

\begin{example}
  Let $p$ be a prime and~$\zeta \in \mu_{p}^{\circ}$. Then by \cref{lem:5,lem:2},
  \begin{displaymath}
    \h_{p}(\mathcal{P}(\zeta,1))=\frac{1}{\varphi(p)} \sum_{k\in (\mathbb{Z}/p\, \mathbb{Z})^{\times}} \log|\iota_{p}(\zeta^{k})-1|_p     = \frac{\log|\Phi_{p}(1)|_{p}}{p-1}= -\frac{\log(p) }{p-1}.    
  \end{displaymath}
\end{example}

However, the situation emerging from this example already contains the
worst possible behavior of these non-Archimedean local heights, as their explicit computation in the next proposition makes evident.

\begin{proposition}\label{prop: explicit non-Archimedean height}
  Let $\omega \in \Gm^{2}(\overline{\mathbb{Q}})$ be a nontrivial
  torsion point and $p$ a prime. Then 
\[
\h_{p}(\mathcal{P}(\omega))=
\begin{cases}
-\displaystyle{\frac{\log (p)}{ p^{r-1} (p-1)}}&\text{ if } \ord({\omega})=p^r \text{ for some } r\geq1,\\
0&\text{otherwise}.
\end{cases}
\]
\end{proposition}

Before proving this proposition, we give an elementary lemma.

\begin{lemma}\label{lem: non-Archimedean absolute values for roots of unity}
  Let $(F,|\cdot|)$ be a field equipped with a non-Archimedean
  absolute value. Then for all $c\in \mathbb{Z}^{2}$ and
  $d\in \mathbb{Z}_{\ge 1}$ such that $\gcd(c_{1},c_{2},d)=1$ and all
  primitive $d$-root of unity $\zeta$ in $F$ we have that
\[
\max(|\zeta^{c_{1}}-1|,|\zeta^{c_{2}}-1|)=|\zeta-1|.
\]
\end{lemma}

\begin{proof}
  Since $\zeta$ is a root of unity, its absolute value is equal
  to~$1$. It follows from the ultrametric inequality that
  $|\zeta^e-1|= |\zeta^{e-1}+\dots+1| \, |\zeta-1| \leq |\zeta-1|$
  for every~$e\geq1$.  As the same inequality holds for
  $e \le -1$ because of the relation~$|\zeta^{e}-1|=|\zeta^{-e}-1|$
  and it also holds trivially for~$e=0$, we have that
  \begin{equation}
  \label{eq:10}
|\zeta^e-1|\le |\zeta-1| \quad \text{ for all } e\in\mathbb{Z}.
\end{equation}
Therefore, $\max(|\zeta^{c_{1}}-1|,|\zeta^{c_{2}}-1|)\le |\zeta-1|$.

Else, since~$\gcd(c_{1},c_{2},d)=1$ there is $b\in\mathbb{Z}^{2}$ such
that~$b_{1} c_{1}+b_{2}c_{2}=1 \pmod d$.  This implies that
$\zeta-1=\zeta^{b_{2}c_{2}}(\zeta^{b_{1}c_{1}}-1)+(\zeta^{b_{2}c_{2}}-1)$,
which by the ultrametric inequality and \eqref{eq:10} ensures that
\[
|\zeta-1|\leq\max(|\zeta^{b_{1}  c_{1}}-1|,|\zeta^{b_{2}  c_{2}}-1|)\leq\max(|\zeta^{c_{1}}-1|,|\zeta^{c_{2}}-1|),
\]
completing the proof.
\end{proof}

\begin{proof}[Proof of \cref{prop: explicit non-Archimedean height}]
  Set~$d=\ord({\omega})$; since $\omega$ is nontrivial, we have that~$d\geq2$.
  Choose~$\zeta \in \mu_{d}^{\circ}$.
  As~$\omega\in \mu_{d}^{2}$,
  by \eqref{eq:44} there is $c\in\mathbb{Z}^{2}$ with
  $\gcd(c_{1},c_{2},d)=1$ such that
  \begin{displaymath}
{\omega}=(\zeta^{c_{1}},\zeta^{c_{2}}).
   \end{displaymath}
   The ultrametric inequality together with \cref{lem:5,lem:
     non-Archimedean absolute values for roots of unity} then implies
   that
\begin{align*}
\h_{p}(\mathcal{P}({\omega}))
  &=\frac{1}{\varphi(d)}\sum_{k}\log\max(|\iota_{p}(\zeta^{k \, c_{2}})-\iota_{p}(\zeta^{k \, c_{1}})|_{p},|\iota_{p}(\zeta^{k\, c_{2}})-1|_{p},|
    \iota_{p}(\zeta^{k\, c_{1}})-1|_{p})\\
&=\frac{1}{\varphi(d)}\sum_{k}\log|\iota_{p}(\zeta^{k})-1|_{p}.
\end{align*}
Therefore the statement follows from \cref{lem:1}.
\end{proof}

Summing over all finite places, \cref{prop: explicit non-Archimedean height} shows that the non-Archimedean contribution to the height
of ${P}({\omega})$ is always nonpositive and that moreover it vanishes
precisely when the order of $\omega$ has at least two different prime
factors.

\begin{corollary}
  \label{cor:1}
  For every nontrivial torsion point~${\omega}$ of
  $\Gm^{2}(\overline{\mathbb{Q}})$ we have that
  \begin{displaymath}
    \sum_{v \in M_{\mathbb{Q}} \setminus \{ \infty \}}\h_{v}(\mathcal{P}(\omega))=-\frac{\Lambda(\ord(\omega))}{\varphi(\ord(\omega))},
  \end{displaymath}
 where $\Lambda$ denotes the von Mangoldt function.
\end{corollary}

\begin{remark}
  \label{rem:9}
  More explicitly, \cref{cor:1} says  that this
  part of the height of $P(\omega)$ is equal to
  \[
-\frac{\log (p)}{p^{r-1} (p-1)}
\]
if $ \ord({\omega})=p^r$  for some prime $p$  and~$r\geq1$, and to $0$ otherwise.
\end{remark}

In turn, this result implies that the non-Archimedean contribution to
the height of~${P}({\omega})$ approaches to zero when $\ord(\omega)$
is large.

\begin{corollary}
  \label{cor:2}
  Let $(\omega_{\ell})_{\ell\ge 1}$ be a sequence of nontrivial
  torsion points of $\Gm^{2}(\overline{\mathbb{Q}})$ with
  $\lim_{\ell\to+\infty}\ord(\omega_{\ell})=+\infty$. Then
  \begin{displaymath}
    \lim_{\ell\to+\infty}\sum_{v \in M_{\mathbb{Q}} \setminus \{ \infty \}}\h_{v}(\mathcal{P}(\omega_{\ell}))=0.
  \end{displaymath}
\end{corollary}

\begin{proof}
  By \cref{cor:1} and \cref{rem:9}, we can reduce without loss of
  generality to the case when
  $ \ord(\omega_{\ell})=p_{\ell}^{r_{\ell}}$ with $p_{\ell}$ a prime
  and $r_{\ell}\ge 1$ for all~$\ell$. Then the sum of the
  non-Archimedean local heights of the vector
  $\mathcal{P}(\omega_{\ell})$ is, up to sign, equal to
\begin{displaymath}
 \frac{\log(p_{\ell})}{p_{\ell}^{r_{\ell}-1}(p_{\ell}-1)} =  \frac{p_{\ell}}{r_{\ell}\, (p_{\ell}-1)} \, 
\frac{\log(p_{\ell}^{r_{\ell}})}{p_{\ell}^{r_{\ell}}}=  \frac{p_{\ell}}{r_{\ell}\, (p_{\ell}-1)} \, 
\frac{\log(\ord(\omega_{\ell}))}{\ord(\omega_{\ell})}.
\end{displaymath}
Since the first factor in the right-hand side is bounded, this
quantity tends to $0$ whenever~$\ell\to +\infty$.
\end{proof}

\section{The limit of the Archimedean height}\label{sec: limit of archimedean heights}

Having computed the non-Archimedean contribution to the height of the
intersection of the line $C$ with its translate by a nontrivial
torsion point, we turn to the limit behavior of its Archimedean
counterpart for strict sequences of such points.

Set for short~$|\cdot|=|\cdot|_{\infty}$. We denote by
\begin{displaymath}
  \mathbb{S}=(S^{1})^{2}=\{z \in (\mathbb{C}^{\times})^{2}\mid |z_{1}|=|z_{2}|=1\} 
\end{displaymath}
the compact torus of the complex torus
$\Gm^{2}(\mathbb{C})=(\mathbb{C}^{\times})^{2}$ and by $\nu$ its
probability Haar measure. Consider the function
\begin{equation}
  \label{eq:29}
F\colon (\mathbb{C}^{\times})^{2}\longrightarrow \mathbb{R}\cup \{-\infty\} ,\quad  z\longmapsto\log \max (|z_{2}-z_{1}|, |z_{2}-1|, |z_{1}-1|).
\end{equation}
Consider also the \emph{co-tropicalization map}
\begin{displaymath}
\cotrop\colon (\mathbb{C}^{\times})^{2}\longrightarrow (\mathbb{R}/2 \pi \mathbb{Z})^{2},\quad z\longmapsto(\arg(z_{1}), \arg(z_{2}))
\end{displaymath}
and the function
\begin{equation}
  \label{eq:11}
  f\colon (\mathbb{R}/2 \pi \mathbb{Z})^{2} \longrightarrow
\mathbb{R}\cup\{-\infty\}, \quad
u\longmapsto\log \max (|\eu^{\iu u_{2}}-\eu^{\iu u_{1}}|,|\eu^{\iu u_{2}}-1|,|\eu^{\iu u_{1}}-1|).
\end{equation}
The direct image measure $\cotrop_{*}\nu$ coincides with the
normalized Lebesgue measure on~$(\mathbb{R}/2 \pi \mathbb{Z})^{2}$,
and the inverse image $\cotrop^{*}f$ coincides on $\mathbb{S}$ with the restriction of~$F$.

The next result is an asymptotic Archimedean counterpart of
\cref{cor:2}.

\begin{proposition}
  \label{prop:1}
  Let $(\omega_{\ell})_{\ell\ge 1}$ be a strict sequence in
  $\Gm^{2}(\overline{\mathbb{Q}})$ of nontrivial torsion points. Then
  \begin{displaymath}
\lim_{\ell\to +\infty}    \h_{\infty}(\mathcal{P}(\omega_{\ell})) = 
\int_{\mathbb{S}} F \, d\nu =  \frac{1}{(2 \pi)^{2}} \int_{(\mathbb{R}/2 \pi \mathbb{Z})^{2}} f(u)\, du_{1} du_{2}.
  \end{displaymath}
\end{proposition}

For its proof, we need the next lemma. Given $d,e\ge 1$ with~$e\mid d$, consider the associated \emph{reduction homomorphism}
between the respective groups of modular units
\begin{displaymath}
  \pi_{d,e}\colon   (\mathbb{Z}/d\mathbb{Z})^{\times}\longrightarrow
  (\mathbb{Z}/e\mathbb{Z})^{\times}.
\end{displaymath}
Let $d= \prod_{p}p^{r_{p}}$ and $e=\prod_{p}p^{s_{p}}$ be their
respective irreducible factorizations. Under the splitting given by the Chinese reminder theorem
we can write
\begin{equation}
  \label{eq:42}
  \pi_{d,e}=\bigoplus_{p}\pi_{p^{r_{p}},p^{s_{p}}},
\end{equation}
where $\pi_{p^{r_{p}},p^{s_{p}}} \colon (\mathbb{Z}/p^{r_{p}}\mathbb{Z})^{\times}\rightarrow
(\mathbb{Z}/p^{
  s_{p}}\mathbb{Z})^{\times}$ denotes the corresponding reduction map.

\begin{lemma}
  \label{lem:6}
  The homomorphism $\pi_{d,e}$ is surjective.
\end{lemma}

\begin{proof}
  The splitting in \eqref{eq:42} allows to reduce to the case in which
  $d=p^{r}$ and $e=p^{s}$ with~$r\ge s\ge 0$. The statement follows
  then from the fact that an element $k\in \mathbb{Z}$ is a unit
  modulo $p^{r}$ if and only if it is a unit modulo $p^{s}$, since
  both conditions are equivalent to~$p\nmid k$.
\end{proof}

In the setting of \cref{prop:1}, for each $\ell\ge 1$  \cref{lem:5} shows that the
 Archimedean local height corresponding to the torsion point $\omega_{\ell}$ writes down as
   \begin{equation}
     \label{eq:19}
     \h_{\infty}(\mathcal{P}(\omega_{\ell})) = \frac{1}{\varphi(d_{\ell})}\sum_{k} F(\iota_{\infty}(\omega_{\ell}^{k}))
     \end{equation}
     with $d_{\ell}=\ord(\omega_{\ell})$ and $k$ ranging over the
     elements of~$(\mathbb{Z}/d_{\ell}\mathbb{Z})^{\times}$.
     
     Consider the uniform probability measure on the $\infty$-adic
     Galois orbit of~$\omega_{\ell}$, that is the discrete measure on the compact torus defined as
     \begin{equation}
       \label{eq:20}
\delta_{O(\omega_{\ell})_{\infty}}=\frac{1}{\varphi(d_{\ell})} \sum_{k} \delta_{\iota_{\infty}(\omega_{\ell}^{k})} ,
     \end{equation}
     where each $\delta_{\iota_{\infty}(\omega_{\ell}^{k})} $ denotes
     the Dirac delta measure on the corresponding point.
     Then the formula in~\eqref{eq:19} can be written as
     \begin{equation}
       \label{eq:21}
       \h_{\infty}(\mathcal{P}(\omega_{\ell})) = \int_{\mathbb{S}} F \, d\delta_{O(\omega_{\ell})_{\infty}}.
     \end{equation}

     It is well-known that the sequence of probability measures in
     \eqref{eq:20} converges weakly to the probability Haar measure
     $\nu$ as~$\ell\to +\infty$. Precisely, for any bounded and
     $\nu$-almost everywhere continuous real-valued function~$\phi$~on~$\mathbb{S}$,
\begin{equation}
  \label{eq:1}
\lim_{\ell\to +\infty}  \int_{\mathbb{S}} \phi \,  d\delta_{O(\omega_{\ell})_{\infty}}= \int_{\mathbb{S}} \phi\, d\nu.
\end{equation}
When $\phi$ is continuous, this is a particular case of Bilu's
equidistribution theorem for the Galois orbits of points of small
height~\cite{Bilu}, whereas its extension to the situation when $\phi$
is just bounded and $\nu$-almost everywhere continuous follows from
general results of measure theory like \cite[Lemme~6.3]{CLT}. But
since the limit in~\eqref{eq:1} concerns torsion
points and not arbitrary points of small height, it can also be proven
in an elementary way using classical facts on Gaussian exponential sums,
as it is done for its quantitative version in
\cite[Proposition~3.3]{DimitrovHabegger}.

In view of~\eqref{eq:21}, \cref{prop:1} can be seen as an equidistribution result for a test
function $\phi$ that is continuous everywhere except at the point~$z=(1,1)$,
where it has a logarithmic singularity.  Hence it is a
particular case of both Chambert-Loir and Thuillier's logarithmic
equidistribution theorem for Galois orbits of points of small height~\cite{CLT},
and of Dimitrov and Habegger's logarithmic
equidistribution theorem for Galois orbits of torsion points of
algebraic tori~\cite{DimitrovHabegger}. In spite of that, we give here
a simple proof relying solely on the more classical equidistribution
theorem in~\eqref{eq:1} and on the formula for the value of a cyclotomic
polynomial at~$1$.

\begin{proof}[Proof of \cref{prop:1}]
  For each $\ell\ge 1$ set $d_\ell=\ord(\omega_\ell)$
  and~$e_{\ell}=\ord(\omega_{\ell,1})$, the latter being a divisor of
  the first.  Since the sequence $(\omega_{\ell})_{\ell \ge 1}$ is
  strict, the condition~\eqref{eq:22} implies that
  \begin{equation}
    \label{eq:23}
    \lim_{\ell\to+\infty}e_{\ell}=+\infty.
  \end{equation}
  
  Consider the function
  $G\colon (\mathbb{C}^{\times})^{2}\to \mathbb{R}\cup\{-\infty\}$
  defined as $G(z)=\log|z_{1}-1|$.  Its integral with respect to the
   measure $\nu$ is the logarithmic Mahler measure of
  the polynomial~$z_{1}-1$, and Jensen's formula applied to it shows
  that this quantity vanishes. On the other hand, its integral with
  respect to the discrete measure in \eqref{eq:20} can be computed~as
  \begin{multline*}
    \int_{\mathbb{S}} G \, d \delta_{O(\omega_{\ell})_{\infty}} = \frac{1}{\varphi(d_{\ell})} \sum_{k\in (\mathbb{Z}/d_{\ell}\mathbb{Z})^{\times}}\log|\iota_{\infty}(\omega_{\ell,1})^{k}-1|
    \\=
    \frac{1}{\varphi(e_{\ell})} \sum_{k\in (\mathbb{Z}/e_{\ell}\mathbb{Z})^{\times}}\log|\iota_{\infty}(\omega_{\ell,1})^{k}-1|
    = \frac{1}{\varphi(e_{\ell})} \log|\Phi_{e_{\ell}}(1)|.
  \end{multline*}
  The second equality follows from the fact that the summand indexed
  by an element $k\in(\mathbb{Z}/d_{\ell}\mathbb{Z})^{\times}$ takes a
  value that depends only on its image under the reduction
  homomorphism~$\pi_{d_{\ell},e_{\ell}}$, which by \cref{lem:6} is
  surjective with fibers of
  cardinality~$\varphi(d_\ell)/\varphi(e_\ell)$.  Hence \cref{lem:2}
  together with~\eqref{eq:23} and the same argument in the proof of
  \cref{cor:2} implies that
  \begin{equation}
    \label{eq:27}
   \lim_{\ell\to +\infty}
 \int_{\mathbb{S}} G \, d \delta_{O(\omega_{\ell})_{\infty}} = 0 = \int_{\mathbb{S}} G\, d\nu.  
  \end{equation}
  
  Now let $(U_{m})_{m\ge 1}$ be the nested sequence of neighborhoods
  of the closed subset $\{z\in \mathbb{S}\mid z_{1}=1 \}$ of the
  compact torus defined as
  \[
  U_m=\Big\{z\in\mathbb{S}\mid\arg(z_1)\in\Big(-\frac{1}{m},\frac{1}{m}\Big)\pmod{2\pi}\Big\}.
  \]
  Since both $F$ and $G$ are continuous outside that closed subset,
  for each $m\ge 1$ the equidistribution theorem in~\eqref{eq:1} shows
  that
  \begin{equation}
    \label{eq:25}
      \lim_{\ell\to+\infty}
    \int_{\mathbb{S}\setminus U_{m}} \hspace{-1mm}F\, d \delta_{O(\omega_{\ell})_{\infty}} =     \int_{\mathbb{S}\setminus U_{m}}  \hspace{-1mm} F \, d\nu,\quad \lim_{\ell\to+\infty} \int_{\mathbb{S}\setminus U_{m}}  \hspace{-1mm} G \, d
  \delta_{O(\omega_{\ell})_{\infty}} = \int_{\mathbb{S}\setminus
    U_{m}}  \hspace{-1mm} G\, d\nu.
  \end{equation}
  The second limit in \eqref{eq:25} together with that in \eqref{eq:27}
  implies~that
    \begin{equation}
      \label{eq:28}
      \lim_{\ell\to+\infty}
      \int_{U_{m}} G \, d \delta_{O(\omega_{\ell})_{\infty}} =
      -      \lim_{\ell\to+\infty}
      \int_{\mathbb{S}\setminus U_{m}}  \hspace{-2mm}G \, d \delta_{O(\omega_{\ell})_{\infty}} 
      =      -      \int_{\mathbb{S}\setminus U_{m}}  \hspace{-2mm}G\, d\nu = 
      \int_{U_{m}} G\, d\nu.
    \end{equation}
We have that $G(z) \le F(z)\le \log(2)$ for all
    $z\in \mathbb{S}$ and so for each~$\ell\ge 1$,
  \begin{equation}
    \label{eq:26}
    \int_{U_{m}} G \, d \delta_{O(\omega_{\ell})_{\infty}}  \le
    \int_{U_{m}} F \, d \delta_{O(\omega_{\ell})_{\infty}}  \le 
    \int_{U_{m}} \log(2) \, d \delta_{O(\omega_{\ell})_{\infty}}.  
  \end{equation}
We deduce from \eqref{eq:25}, \eqref{eq:28} and \eqref{eq:26} that
  \begin{multline*}
    \int_{\mathbb{S}\setminus U_{m}}  \hspace{-1mm} F \, d \nu+   \int_{U_{m}} G \, d \nu \le
    \liminf_{\ell}    \int_{\mathbb{S}} F \, d \delta_{O(\omega_{\ell})_{\infty}}  \\ \le     \limsup_{\ell}    \int_{\mathbb{S}} F \, d \delta_{O(\omega_{\ell})_{\infty}}  \le 
     \int_{\mathbb{S} \setminus U_{m}}  \hspace{-1mm} F \, d\nu +   \log(2) \, \nu(U_{m}).
   \end{multline*}

   The first equality in the statement follows taking the limit for~$m\to +\infty$, using the fact that
  $ \lim_{m\to+\infty}\nu(U_{m})=0$ and the absolute continuity of the
  Lebesgue integral, whereas the second is a direct consequence of the
  first together with the change of variables formula.
  \end{proof}

\section{Computing the integral}
\label{sec:computing-integral}

Next we compute the integral giving the limit of the Archimedean
heights of the vectors $\mathcal{P}(\omega_{\ell})$ for a strict
sequence of nontrivial torsion points (\cref{prop:1}), namely 
  \begin{displaymath} 
    I=\frac{1}{(2 \pi)^{2}} \int_{(\mathbb{R}/2
      \pi \mathbb{Z})^{2}} \log \max (|\eu^{\iu u_{2}}-\eu^{\iu u_{1}}|,|\eu^{\iu u_{2}}-1|,|\eu^{\iu u_{1}}-1|)\, du_{1} du_{2}.
  \end{displaymath}
  More precisely, we will prove the following.

\begin{proposition}
  \label{prop:2}
  \begin{math} \displaystyle I= \frac{2\,
      \zeta(3)}{3\, \zeta(2)}=0.487175\dots\, .
  \end{math}
\end{proposition}

We will exploit the existing symmetries to simplify the
calculation of this integral. To this end, consider the linear
automorphisms of $\mathbb{R}^{2}$ given as
\begin{equation}
  \label{eq:8}
  \alpha(u_{1},u_{2})=(u_{2},u_{1}), \quad \beta(u_{1},u_{2})=(-u_{2},u_{1}-u_{2}), \quad \gamma(u_{1},u_{2})=(-u_{1},-u_{2}).
\end{equation}
Since they restrict to automorphisms of the lattice~$(2\pi\mathbb{Z})^2$,
each of them also induces an automorphism of the quotient
space~$(\mathbb{R}/2\pi\mathbb{Z})^{2}$, that we denote with the
corresponding overlined letter. We set $H$ for the group of
automorphisms of $(\mathbb{R}/2\pi\mathbb{Z})^{2}$ generated by them.

Recall that a \emph{fundamental domain} for the action of $H$ on
$(\mathbb{R}/2\pi\mathbb{Z})^{2}$ is a closed subset of $(\mathbb{R}/2\pi\mathbb{Z})^{2}$ whose translates by elements of $H$
cover this space, and such that the intersection of any two different
translates has empty interior.  

For the sequel, denote by $ D $ the triangle of $\mathbb{R}^{2}$ with vertices
$(0,0)$, $(\pi, 0)$ and $(4 \pi/3, 2 \pi/ 3)$, and by~$\overline{D}$
its image in~$(\mathbb{R}/2\pi\mathbb{Z})^{2}$.  Set also $f$ for the
integrand of~$I$, which coincides with the function in~\eqref{eq:11}.
    
\begin{proposition}
  \label{prop:4}
  The automorphism group $H$ verifies the following properties:
  \begin{enumerate}[leftmargin=*]
  \item \label{item:4} $H = \{ \bar\alpha^{e_{1}}\, \bar\beta^{e_{2}}\, \bar\gamma^{e_{3}}\mid e_{1},e_{3}=0,1, \ e_{2}=0,1,2\}$,
  \item \label{item:7} $\# H=12$, 
  \item \label{item:3} the set  $\overline{D}$ is a
    fundamental domain for the action of~$H$,
  \item \label{item:1} the Lebesgue measure on
    $(\mathbb{R}/2 \pi \mathbb{Z})^{2}$ is invariant under~$H$,
  \item \label{item:2} the function $f$ is invariant under~$H$, 
  \item \label{prop:6} for each $u \in \overline{D} $ we have
    that~$f(u)=\log |\eu^{\iu u_{1}}-1|$.
  \end{enumerate}
\end{proposition}

\begin{proof}
  The generators of $H$ verify the relations
  \begin{displaymath}
    \bar\alpha^{2}=\bar\beta^{3}=\bar\gamma^{2}=1, \quad \bar\gamma\, \bar\alpha=\bar\alpha\, \bar\gamma, \quad \bar\gamma\, \bar\beta= \bar\beta\, \bar\gamma, \quad \bar\beta\, \bar\alpha=\bar\alpha\, \bar\beta^{2}. 
  \end{displaymath}
  The last three imply that all the elements of $H$ are of the form
  $\bar\alpha^{e_{1}}\, \bar\beta^{e_{2}}\, \bar\gamma^{e_{3}}$ with
  $e_{1},e_{2},e_{3}\ge0$, whereas the others give the stated upper
  bounds for these exponents, proving~\eqref{item:4}.

  To prove~\eqref{item:3}, notice that the action of $\beta$ on the
  nonzero vertices of $ D $ is given by
 \begin{displaymath}
    \xymatrix{
  (\pi,0)\ar[r]^-{\beta}& (0,\pi)\ar[r]^-{\beta}& (-\pi,-\pi), 
  &\big(\frac{4\pi}{3},\frac{2\pi}{3}\big)\ar[r]^-{\beta}& \big(\frac{-2\pi}{3},\frac{2\pi}{3}\big)\ar[r]^-{\beta}& \big(\frac{-2\pi}{3},\frac{-4\pi}{3}\big)
}.
  \end{displaymath}
  The actions on $\mathbb{R}^{2}$ of $\alpha$ and $\gamma$ are both
  easy to visualize, since these linear maps are the symmetries with
  respect to the diagonal line and to the origin, respectively. The
  first picture in \cref{fig:1} describes the action on $ D $ of these
  three linear maps and their compositions, whereas the second shows
  these regions on the quotient space~$(\mathbb{R}/2\pi\mathbb{Z})^{2}$, represented by the square
  $[0,2\pi]^{2}$ with opposite edges identified.
  \begin{figure}[!htbp]
    \begin{tikzpicture}[scale=0.70]
      \draw[<->] (-4-10, 0+3) -- (4-10, 0+3);
      \draw[<->] (0-10, -4+3) -- (0-10, 4+3);
    \draw[fill=green!25] (0-10,0+3) -- (3-10,0+3) -- (4-10,2+3) -- cycle; \node at (2.4-10,0.6+3) {$\scriptscriptstyle  D $};
    \draw  (4-10-6,2-6+3) -- (6-10-6,6-6+3) -- (3-10-6,3-6+3) -- cycle; \node at (4.1-10-6,3.4-6+3) {$\scriptscriptstyle \beta^{2}  D $};
    \draw  (0-10,0+3) -- (2-10,4+3) -- (0-10,3+3) -- cycle; \node at (0.7-10,2.7+3) {$\scriptscriptstyle \alpha D$};
    \draw  (2-10-6,4-6+3) -- (3-10-6,3-6+3) -- (6-10-6,6-6+3) -- cycle; \node at (3.2-10-6,4.1-6+3) {$\scriptscriptstyle \alpha\beta^{2} D$};
    \draw  (2-10,4-6+3) -- (0-10,6-6+3) -- (3-10,6-6+3) -- cycle; \node at (1.8-10,5.2-6+3) {$\scriptscriptstyle \alpha\beta D$};
    \draw  (6-10-6,0+3) -- (6-10-6,3+3) -- (4-10-6,2+3) -- cycle; \node at (5.3-10-6,1.6+3) {$\scriptscriptstyle \beta D$};
    \draw (0-10,0+3) -- (-3-10,0+3) -- (-4-10,-2+3) -- cycle; \node at (-2.4-10,-0.6+3) {$\scriptscriptstyle  \gamma D$};
    \draw (0-10,0+3) -- (3-10,3+3) -- (2-10,4+3) -- cycle; \node at (2.2-10,3+3) {$\scriptscriptstyle  \beta^{2}\gamma D$};
    \draw  (4-10,2+3) -- (3-10,3+3) -- (0-10,0+3) -- cycle; \node at (2.8-10,2+3) {$\scriptscriptstyle \alpha \beta^{2}\gamma D$};
    \draw (3-10-6,0+3) -- (6-10-6,0+3) -- (4-10-6,2+3)  -- cycle; \node at (4.3-10-6,0.6+3) {$\scriptscriptstyle \alpha \beta \gamma D$};
    \draw  (6-10-6,3-3) -- (6-10-6,6-3) -- (4-10-6,2-3) -- cycle; \node at (5.35-10-6,3.4-3) {$\scriptscriptstyle \alpha \gamma D$};                                 \draw  (0-10,3-3) -- (2-10,4-3) -- (0-10,6-3) -- cycle; \node at (0.8-10,4.3-3) {$\scriptscriptstyle \beta\gamma D$};

    \begin{scope}[xshift=-25mm]
    \draw[fill=green!25] (3,0) -- (0,0) -- (4,2) -- cycle; \node at (2.4,0.6) {$\scriptscriptstyle \bar{D}$};
    \draw (3,0) -- (6,0) -- (4,2)  -- cycle; \node at (4.3,0.6) {$\scriptscriptstyle \bar \alpha\bar\beta \bar\gamma \bar{D}$};
    \draw  (6,0) -- (4,2) -- (6,3)  -- cycle; \node at (5.3,1.6) {$\scriptscriptstyle \bar\beta \bar{D}$};                                
    \draw  (6,3) -- (4,2) -- (6,6)  -- cycle; \node at (5.4,3.4) {$\scriptscriptstyle \bar\alpha \bar\gamma \bar{D}$};                                
    \draw  (4,2) -- (6,6) -- (3,3) -- cycle; \node at (4.1,3.4) {$\scriptscriptstyle \bar\beta^{2} \bar{D}$};                                
    \draw  (4,2) --  (0,0) -- (3,3) -- cycle; \node at (2.85,2) {$\scriptscriptstyle \bar\alpha \bar\beta^{2}\bar\gamma \bar{D}$};
    \draw (3,3) -- (0,0) -- (2,4) -- cycle; \node at (2.2,3) {$\scriptscriptstyle \bar\beta^{2}\bar\gamma \bar{D}$};
    \draw  (0,0) -- (2,4) -- (0,3) -- cycle; \node at (0.7,2.7) {$\scriptscriptstyle \bar\alpha \bar{D}$};
    \draw  (0,3) -- (2,4) -- (0,6) -- cycle; \node at (0.8,4.3) {$\scriptscriptstyle \bar\beta\bar\gamma \bar{D}$};
    \draw  (0,6) -- (2,4) -- (3,6) -- cycle; \node at (1.8,5.2) {$\scriptscriptstyle \bar\alpha\bar\beta \bar{D}$};
    \draw (6,6) -- (2,4) --  (3,6)  -- cycle; \node at (3.6,5.5) {$\scriptscriptstyle \bar\gamma \bar{D}$};
    \draw  (2,4) -- (3,3) -- (6,6) -- cycle; \node at (3.2,4.1) {$\scriptscriptstyle \bar\alpha\bar\beta^{2} \bar{D}$};
    \end{scope}
    \end{tikzpicture}
    \caption{The action of $H$ on the fundamental domain $\overline{D}$}
\label{fig:1}
\end{figure}

This second picture shows that the only product
$\bar\alpha^{e_{1}}\, \bar\beta^{e_{2}}\, \bar\gamma^{e_{3}}$ with
$e_{1},e_{3}=0,1$ and $e_{2}=0,1,2$ that fixes $ \overline{D}$ occurs
when $e_{1}=e_{2}=e_{3}=0$. This proves~\eqref{item:7}.  The statement
in \eqref{item:3} can also be checked from the picture: the translates
of $\overline{D}$ by the elements of~$H$ fit in 
$(\mathbb{R}/2\pi \mathbb{Z})^{2}$ like the pieces of a puzzle. Hence
these translates cover the whole of the space, and the intersections
of any two different translates has empty interior.

To prove the statement in~\eqref{item:1}, note that the linear
automorphisms in \eqref{eq:8} preserve the Lebesgue measure of
$\mathbb{R}^{2}$ because their determinants are equal to~$1$, and so
do the induced automorphisms of~$(\mathbb{R}/2 \pi \mathbb{Z})^{2}$.

For~\eqref{item:2}, consider the functions on
$(\mathbb{R}/2 \pi \mathbb{Z})^{2}$ defined as
  \begin{displaymath}
\phi_{1}(u)=|\eu^{\iu u_{2}}-\eu^{\iu u_{1}}|,     \quad \phi_{2}(u)=|\eu^{\iu u_{1}}-1|, \quad \phi_{3}(u)=|\eu^{\iu u_{2}}-1|. 
  \end{displaymath}
  They are invariant under~$\bar\gamma$, whereas $\bar \alpha$ leaves
  invariant $\phi_{1}$ and exchanges $\phi_{2}$ with $\phi_{3}$, and
  $\beta$ makes a cyclic permutation.  We have that
  $f=\log\max(\phi_{1},\phi_{2},\phi_{3})$ and so this function is
  invariant under~$H$.

  Finally, note that $|\eu^{\iu s}-1|=\sqrt{2-2\cos(s)}$ and so for~$s,s'\in [0,2\pi]$,
  \begin{equation}
    \label{eq:31}
    |\eu^{\iu s}-1|\le |\eu^{\iu s'}-1| \text{ if
  and only if } s\le s'\le 2\pi-s.
  \end{equation}
  On the other hand, a point $u\in \mathbb{R}^{2}$ lies in
  $ D $ if and only if it verifies the inequalities
  \begin{displaymath}
    u_{2}\ge 0,\quad u_{1}-2 u_{2}\ge 0, \quad u_{2}-2u_{1}+2\pi\ge 0.
  \end{displaymath}
  In particular, $u_1-u_2\in[0,2\pi]$ for every~$u\in D$.  These
  inequalities imply those in \eqref{eq:31} for $s'=u_{1}$ and
  $s=u_{1}-u_{2}, u_{2}$. Hence $\phi_{2}\ge \phi_{1},\phi_{3}$
  on~$\overline{D}$, which proves~\eqref{prop:6}.
\end{proof}

We also need the next integral formul\ae.

\begin{lemma}
  \label{lem:3}
The following equalities hold:
  \begin{enumerate}[leftmargin=*]
  \item \label{item:5} $  \displaystyle \int_{0}^{\pi}  s   \log |\eu^{\iu s}-1|\, ds= \frac{7}{4}\, \zeta(3)$,
  \item \label{item:6} $ \displaystyle     \int_{\pi}^{\frac{4 \pi}{3}} (4 \pi-3 s)  \log |\eu^{\iu s}-1|\, ds = \frac{11}{12}\, \zeta(3)$.    
  \end{enumerate}
\end{lemma}

\begin{proof}
  Let~$0<\varepsilon\le \pi$. The Dirichlet-Hardy test for the convergence of
  series \cite[page~42]{JeJe:mmph} implies that the sequence of
  partial sums $\sum_{k=1}^{\ell} {\eu^{\iu k s}}/{k} $, $\ell\ge 1$,
  converges uniformly on the interval $[\varepsilon , \pi]$ to the
  principal determination of~$-\log(1-\eu^{\iu s})$.
  Hence the sequence
  \begin{equation}
    \label{eq:15}
\Re\bigg(\sum_{k=1}^{\ell} \frac{\eu^{\iu k  s}}{k} \bigg), \quad \ell\ge1,
  \end{equation}
  converges uniformly on this
  interval to the function~$-\log|\eu^{\iu s}-1|$. This implies~that
  \begin{equation}
    \label{eq:14}
    \int_{\varepsilon}^{\pi}  s   \log |\eu^{\iu s}-1|\, ds= -\sum_{k\ge1}   \frac{1}{k} \Re\bigg(  \int_{\varepsilon}^{\pi} s\, \eu^{\iu k  s} \, ds \bigg).
  \end{equation}
  For each $k\ge 1$, integration by parts gives
  \begin{displaymath}
     \int_{\varepsilon}^{\pi} s\, \eu^{\iu k  s} \, ds = \frac{1}{\iu k}\Big( \Big[ s\, \eu^{\iu k s} \Big]_{\varepsilon}^{\pi} - \int_{\varepsilon}^{\pi} \eu^{\iu k  s} \, ds\Big) = \frac{1}{\iu k} \Big( \pi (-1)^{k} - \varepsilon  \, \eu^{\iu k \varepsilon} - \frac{1}{\iu k} ((-1)^{k}-\eu^{\iu k \varepsilon})\Big).
  \end{displaymath}
We deduce from \eqref{eq:14} that
  \begin{displaymath}
    \int_{\varepsilon}^{\pi}  s   \log |\eu^{\iu s}-1|\, ds= \sum_{k} \Big(\frac{\varepsilon}{k^{2}} \Im(\eu^{\iu k \varepsilon }) + \frac{1}{k^{3}} ( \Re(\eu^{\iu k \varepsilon})-(-1)^{k})\Big).
  \end{displaymath}
  Taking the limit $\varepsilon \to 0$ we obtain that
  \begin{multline*}
    \int_{0}^{\pi}  s   \log |\eu^{\iu s}-1|\, ds= \sum_{k} \frac{1-(-1)^{k}}{k^{3}} =2\, \sum_{2\nmid k} \frac{1}{k^{3}} \\=
    2\, \bigg( \sum_{k} \frac{1}{k^{3}} - \sum_{2\mid k} \frac{1}{k^{3}}\bigg) = 2\, \Big( 1-\frac{1}{8}\Big) \sum_{k} \frac{1}{k^{3}} = \frac{7}{4} \, \zeta(3),
  \end{multline*}
  proving~\eqref{item:5}. For the formula in~\eqref{item:6}, we deduce from the uniform convergence of the sequence in \eqref{eq:15} that
  \begin{equation}
    \label{eq:17}
    \int_{\pi}^{\frac{4 \pi}{3}} (4 \pi-3 s)  \log |\eu^{\iu s}-1|\, ds 
    = -\sum_{k} \frac{1}{k} \Re\bigg(\int_{\pi}^{\frac{4 \pi}{3}}  (4 \pi-3 s)  \, \eu^{\iu k s}\, ds\bigg).
  \end{equation}
  For each~$k\ge 1$, integrating by parts now gives
  \begin{multline*}
    \int_{\pi}^{\frac{4 \pi}{3}}  (4 \pi-3 s)  \, \eu^{\iu k s}\, ds =  \frac{1}{\iu k}  \Big(  \Big[ (4 \pi-3 s ) \, \eu^{\iu ks}\Big]_{\pi}^{\frac{4\pi}{3}}
    +3 \int_{\pi}^{\frac{4\pi}{3}} \eu^{ \iu k s} \, ds \Big)  \\
=  \frac{1}{\iu k}  \Big(  -\pi \, (-1)^{k}+\frac{3}{\iu k} \, (\rho^{k} -(-1)^{k})\Big)
  \end{multline*}
  with
  $\rho= \eu^{\iu \frac{4\pi}{3}} = \frac{-1}{2}-\iu
  \frac{\sqrt{3}}{2}$. By~\eqref{eq:17}, we conclude that
  \begin{align*}
    \int_{\pi}^{\frac{4 \pi}{3}} (4 \pi-3 s)  \log |\eu^{\iu s}-1|\, ds = &3 \sum_{k}  \frac{1}{k^{3}} (\Re(\rho^{k}) -(-1)^{k})\\
= &3 \bigg(  \sum_{3\mid k}\frac{1}{k^{3}} -\frac{1}{2}
    \sum_{3\nmid k}\frac{1}{k^{3}}  -  \sum_{2\mid k} \frac{1}{k^{3}} +\sum_{2\nmid k}\frac{1}{k^{3}}    \bigg) \\
= &3 \bigg(  \frac{3}{2} \sum_{3\mid k}\frac{1}{k^{3}} -\frac{1}{2}
    \sum_{k}\frac{1}{k^{3}}  -  2\, \sum_{2\mid k} \frac{1}{k^{3}} +\sum_{ k}\frac{1}{k^{3}}    \bigg) \\
    = &3\, \Big( \frac{3}{2} \, \frac{1}{27} -\frac{1}{2} - 2\, \frac{1}{8} +1 \Big) \sum_{k}\frac{1}{k^{3}} \\ = & \frac{11}{12}\, \zeta(3),
  \end{align*}
  as stated.
\end{proof}

\begin{proof}[Proof of \cref{prop:2}]
  \cref{prop:4} together with Fubini's theorem implies that
  \begin{displaymath}
I=\frac{12}{(2 \pi)^{2}} \int_{\overline{D}}\log |\eu^{\iu u_{1}}-1| \, du_{1} du_{2}
    = \frac{12}{(2 \pi)^{2}} \int_{0}^{\frac{4 \pi}{3}} \min\Big( \frac{u_{1}}{2}, 2 \pi-\frac{3 u_{1}}{2}\Big)  \log |\eu^{\iu u_{1}}-1|\, du_{1}.
  \end{displaymath}
  By  \cref{lem:3}, this quantity is equal to
  \begin{displaymath}
\frac{12}{(2 \pi)^{2}} \Big( \frac{1}{2}\, \frac{7}{4}\, \zeta(3)+\frac{1}{2}\, \frac{11}{12}\, \zeta(3)\Big)=   \frac{4\, \zeta(3)}{\pi^{2}} = \frac{2\, \zeta(3)}{3\, \zeta(2)}
  \end{displaymath}
thanks to Euler's formula~$\zeta(2)=\pi^{2}/6$. 
\end{proof}

\section{The distribution of the height}
\label{sec:distribution-height}

Finally we can join together the different pieces from the previous sections to obtain our first main
result. Set for short
\begin{displaymath}
  \eta
  =\frac{2\, \zeta(3)}{3\, \zeta(2)}
  =0.487175\dots \, .
\end{displaymath}

\begin{theorem}
  \label{thm:1}
  Let $(\omega_{\ell})_{\ell\ge 1}$ be a strict sequence in
  $\Gm^{2}(\overline{\mathbb{Q}})$ of nontrivial torsion points. Then
  \begin{displaymath}
\lim_{\ell\to +\infty}    \h( C\cap \omega_{\ell} C)=\eta.
  \end{displaymath}
\end{theorem}

\begin{proof}
  This follows readily from the definition of the height in
  \eqref{eq:4} together with \cref{cor:2} and 
  \cref{prop:1,prop:2}.
\end{proof}

The following questions are natural in this context.

\begin{question}
  \label{que:2}
  Is it possible to obtain a quantitative version of \cref{thm:1}?
  This is understood as an estimate, for a given nontrivial torsion point~$\omega$,
  of the discrepancy between the height of $C\cap \omega C$ and the
  limit value $\eta$ in terms of the strictness degree of~$\omega$,
  that is, the minimal degree of a $1$-dimensional algebraic subgroup
  containing it.
\end{question}

\begin{question}
  \label{que:1}
  Can \cref{thm:1} be extended to points of small height? It would be
  interesting to prove (or disprove) that its conclusion holds for a
  strict sequence in $\Gm^{2}(\overline{\mathbb{Q}})$ of points whose
  height converges to zero, but are not necessarily torsion.
\end{question}

The notion of strict sequence of points can be extended to include the
finite subsets that appear naturally when doing statistics on the
values of the height, as follows.

\begin{defn}\label{def: strict sequences of sets}
Let $V$ be an algebraic subset of~$\Gm^{n}(\overline{\mathbb{Q}})$.
A sequence $(E_{\ell})_{\ell\ge1}$ of nonempty finite subsets of $V$
is \emph{strict} if for every algebraic subgroup $H\subset\Gm^{n}(\overline{\mathbb{Q}})$ not containing $V$ we have that
 \begin{displaymath}
\lim_{\ell\to+\infty}   \frac{\# (E_{\ell}\cap H)}{\#E_{\ell}} =0.
 \end{displaymath}
 \end{defn}
 
 The situation of interest in this section is the
 case~$V=\Gm^{n}(\overline{\mathbb{Q}})$.  The criterion
 in~\eqref{eq:22} can be extended to this setting: a sequence
 $(E_{\ell})_{\ell\ge1}$ is strict in $\Gm^{n}(\overline{\mathbb{Q}})$
 if and only if we have that
\begin{equation}
\label{eq:33}
\lim_{\ell\to+\infty}   \frac{\#  (E_{\ell} \cap \Ker(\chi^{a}))}{\#E_{\ell}} =0 \quad \text{ for all } a\in \mathbb{Z}^{n}\setminus \{(0,\dots,0)\}.
\end{equation}

\begin{example}
  \label{exm:1}
  The sequence of the sets of $d$-torsion points $\mu_d^n$ of $\Gm^{n}(\overline{\mathbb{Q}})$ is
  strict. Indeed, for each
  $a\in \mathbb{Z}^{n}\setminus \{(0,\dots, 0)\}$ and $d\ge 1$ we have
  that
  \begin{equation}
    \label{eq:2}
    \frac{\#  (\mu_{d}^{n} \cap \Ker(\chi^{a}))}{\# \mu_{d}^{n}} =
    \frac{\#    \Ker(\chi^{a}|_{\mu_{d}^{n}})}{\# \mu_{d}^{n}} =
    \frac{1}{\# \Im(\chi^{a}|_{\mu_{d}^{n}})} = \frac{\gcd(a_{1},\dots, a_{n},d)}{d},
  \end{equation}
  where the last equality comes from the fact that the image of
  $\mu_{d}^{n}$ under $\chi^{a}$ is generated by
  $\zeta^{\gcd(a_1,\ldots,a_n)}$ for any primitive $d$-root of
  unity~$\zeta$.  Hence the quotient in \eqref{eq:2} tends to $0$
  when~$d\to+\infty$. Since this holds for every~$a$, the criterion in
  \eqref{eq:33} implies that the sequence $ (\mu_{d}^{n})_{d\ge 1}$ is
  strict.
\end{example}

We next extend \cref{thm:1} to compute the typical value of the height
of the intersection of the line $C$ with its translates by the torsion
points in a strict sequence of nonempty finite subsets.

\begin{theorem}
  \label{thm:2}
  Let $(W_{\ell})_{\ell\ge 1}$ be a strict sequence of nonempty finite
  subsets of $\Gm^{2}(\overline{\mathbb{Q}})$ of nontrivial torsion
  points. Then for~each $\varepsilon >0$ we have that
  \begin{displaymath} 
    \lim_{\ell\to+\infty}    \frac{\# \{\omega\in W_{\ell} \mid |\h(C
      \cap \omega C)- \eta|<\varepsilon \}}{\#W_{\ell}} =1.
  \end{displaymath}
Moreover 
  \begin{math}
\displaystyle{    \lim_{\ell\to+\infty}
\frac{1}{\# W_{\ell}} \sum_{\omega\in W_{\ell}}\h(C\cap \omega C)=\eta.}
  \end{math}
\end{theorem}

This is a consequence of the previous results and the following general
transfer principle.

\begin{lemma}
  \label{lem:4}
  Let $V$ be an algebraic subset of~$\Gm^{n}(\overline{\mathbb{Q}})$.
  Let $(E_{\ell})_{\ell\ge1}$ be a strict sequence of nonempty finite
  subsets of~$V$, $\phi $ a real-valued
  function on $\bigcup_{\ell\ge 1} E_{\ell}$ and $\kappa$ a real
  number such that for every strict sequence
  $(\gamma_{\ell})_{\ell\ge 1}$ in $V$ contained in this union we have that~$\lim_{\ell\to+\infty}\phi(\gamma_{\ell})=\kappa$. Then
  \begin{enumerate}[leftmargin=*]
  \item \label{item:10} for each $\varepsilon >0$ we have that
    $\displaystyle{\lim_{\ell\to+\infty} \frac{\# \{\gamma \in E_{\ell} \mid
      |\phi(\gamma)-\kappa|<\varepsilon \}}{\#E_{\ell}}=1}$,
\item \label{item:11} if $\phi$ is bounded on~$\bigcup_{\ell\ge 1} E_{\ell}$,
then
  \begin{math}
   \displaystyle{ \lim_{\ell\to+\infty}        \frac{1}{\# E_{\ell}} \sum_{\gamma\in E_{\ell}} \phi(\gamma) = \kappa.}
  \end{math}
  \end{enumerate}
\end{lemma}

\begin{proof}
  We argue by contradiction, supposing that there is $\varepsilon >0$
  for which the statement in \eqref{item:10} does not hold.
  Restricting to a subsequence, we can then assume that there is $c>0$
  such that
  \begin{displaymath}
        \frac{\# \{\gamma \in E_{\ell} \mid  |\phi(\gamma)-\kappa|\ge \varepsilon \}}{\# E_{\ell}} \ge c  \quad \text{ for all } \ell\ge 1.
  \end{displaymath}
  Let $(H_{j})_{j\geq1}$ be the complete list of algebraic subgroups
  of $\Gm^{n}(\overline{\mathbb{Q}})$ not containing~$V$.
  Unless~$V=\{(1,\ldots,1)\}$, in which case the statement is trivial,
  this is a countably infinite list.  Since $(E_{\ell})_{\ell\ge 1}$
  is strict in~$V$, we can take a subsequence $(E_{\ell_k})_{k\ge 1}$
  such~that
  \begin{displaymath}
  \frac{\#\big(    E_{\ell_k}\cap \bigcup_{j=1}^{k} H_{j} \big)}{\# E_{\ell_k}} <c \quad \text{ for all }k\ge 1.
  \end{displaymath}
  Then for each $k\ge 1$ we can choose 
  $\gamma_{k} \in E_{\ell_{k}} \setminus \bigcup_{j=1}^{k} H_{j} $ such
  that
  \begin{equation}
    \label{eq:18}
    |\phi(\gamma_{k})-\kappa|\ge \varepsilon. 
  \end{equation}
  The sequence $(\gamma_{k})_{k\ge 1}$ is strict in $V$ by construction, and
  so the inequality \eqref{eq:18} contradicts the hypothesis and
  proves~\eqref{item:10}. The statement in \eqref{item:11} follows
  easily from \eqref{item:10} and the hypothesis that $\phi$ is
  bounded.
\end{proof}

\begin{proof}[Proof of \cref{thm:2}] This is a direct
  consequence of \cref{lem:4} applied to the strict sequence
  $(W_{\ell})_{\ell\ge 1}$ of $\Gm^{2}(\overline{\mathbb{Q}})$ and to the real-valued function
  \begin{displaymath}
\bigcup_{\ell\ge1}    W_{\ell} \longrightarrow \mathbb{R}, \quad \omega\longmapsto \h(C\cap \omega C)
  \end{displaymath}
  together with \cref{thm:1} and \cref{prop: bounds for the height}.
\end{proof}

The next result is an easy consequence of \cref{exm:1} and
\cref{thm:2}.

\begin{corollary}
  \label{cor:3}
For each $\varepsilon >0$ we have that
\begin{displaymath}
  \lim_{d\to +\infty}  \frac{1}{d^{2}-1} \, \#\{\omega\in \mu_{d}^{2}\setminus\{(1,1)\} \mid |\h(C\cap \omega C)
  -\eta|<\varepsilon \} =1.
\end{displaymath}
Moreover 
  \begin{math}
\displaystyle{    \lim_{d\to+\infty}
\frac{1}{d^{2}-1} \hspace{-2mm}\sum_{\omega\in \mu_{d}^{2}\setminus\{(1,1)\} }\h(C\cap \omega C)=\eta.}
  \end{math}
\end{corollary}

\section{Intermezzo: sequences of torsion points in algebraic subgroups}
\label{sec:interm-sequ-tors}

Here we extend our previous study to sequences of torsion points which
are strict in a fixed irreducible component over the rationals of a
$1$-dimensional algebraic subgroup of the torus, thus finding other
interesting limit values for the height.

We start by recalling that the $1$-dimensional algebraic subgroups of
$\Gm^{2}(\overline{\mathbb{Q}})$ are the algebraic subsets defined by
a binomial of the form $\chi^{c}-1$
with~$c\in \mathbb{Z}^{2} \setminus \{(0,0)\}$.  Writing $c= l a$ for
a primitive vector $a\in \mathbb{Z}^{2}$ and a positive integer~$l$,
its irreducible decomposition~over~$\mathbb{Q}$~is
\begin{displaymath}
  Z(\chi^{c}-1)=\bigcup_{e\mid l}Z(\Phi_{e}(\chi^{a}))
\end{displaymath}
with $\Phi_{e}$ the $e$-th cyclotomic polynomial.

Hence for the rest of this section we fix a primitive vector
$a\in \mathbb{Z}^{2}$ and a positive integer~$e$, and we consider the
algebraic subset~$V_{a,e}=Z(\Phi_{e}(\chi^{a}))$. Its irreducible
decomposition over $\overline{\mathbb{Q}}$ is given by the disjoint union
\begin{displaymath}
    V_{a,e}=\bigcup_{\zeta\in \mu_{e}^{\circ}} Z(\chi^{a}-\zeta).
\end{displaymath}
Each irreducible component $Z(\chi^{a}-\zeta)$ is a \emph{torsion curve}, since it is the
translate of the $1$-dimensional (algebraic) subtorus
$Z(\chi^{a}-1) \simeq \Gm(\overline{\mathbb{Q}})$ by any torsion point
of it.

Let $(\omega_{\ell})_{\ell\ge 1}$ be a strict sequence in $ V_{a,e}$
of torsion points. Since any algebraic subgroup of the torus not
containing $V_{a,e}$ intersects this algebraic subset in a finite
number of points, the strictness of $(\omega_{\ell})_{\ell\ge 1}$ in
$V_{a,e}$ is equivalent to the fact that each torsion point appears at
most a finite number of times in the sequence. In turn, this is
equivalent to
\begin{equation}
  \label{eq:53}
  \lim_{\ell\to +\infty} \ord(\omega_{\ell})= +\infty.
\end{equation}
Recall that for each $\ell$ we denote by
$\delta_{O(\omega_{\ell})_{\infty}}$ the uniform probability measure
on the $\infty$-adic Galois orbit of $\omega_{\ell}$ as
in~\eqref{eq:20}.  Our first goal is to determine the limit of this
sequence of discrete measures with respect to the weak-$*$ topology.
To this end, consider the subset of the compact torus $\mathbb{S} $
defined as
  \begin{displaymath}
    \mathbb{S}_{a,e}=\{z\in \mathbb{S}\mid \Phi_{e}(\chi^{a}(z))=0\}
  = \bigcup_{\zeta\in \mu_{e}^{\circ}}  \{z\in \mathbb{S}\mid \chi^{a}(z)=\iota_{\infty}(\zeta)\}.
\end{displaymath}
It is a disjoint union of translates of the
circle~$\{z\in \mathbb{S}\mid \chi^{a}(z)=1\}\simeq S^{1}$ containing
the $\infty$-adic Galois orbit of any torsion point in~$V_{a,e}$.  For
each primitive $e$-root of unity~$\zeta$ we denote by
$\lambda_{a,e,\zeta}$ the
probability measure on $\mathbb{S}_{a,e}$ supported on the corresponding translate of~$S^1$, where it coincides with the measure induced by the
probability Haar measure of this circle. Consider then the probability
measure on $\mathbb{S}_{a,e}$ defined as
\begin{displaymath}
   \nu_{a,e}=\frac{1}{\varphi(e)} \sum_{\zeta \in \mu_{e}^{\circ}}\lambda_{a,e,\zeta}.
\end{displaymath}

The next result is the analogue in our present situation of~\eqref{eq:1}.
Its proof is done by reducing to a result of
Dimitrov and Habegger in~\cite{DimitrovHabegger} on the
equidistribution of the orbits of roots of unity under the action of
large subgroups of their Galois groups.

\begin{proposition}
  \label{prop:9}
  Let $(\omega_{\ell})_{\ell\ge 1}$ be a strict sequence in $ V_{a,e}$
  of torsion points and $\phi $ a bounded and $\nu_{a,e}$-almost
  everywhere continuous real-valued function on~$\mathbb{S}_{a,e}$. Then
  \begin{equation}\label{eq: equidistribution on Vae}
    \lim_{\ell\to +\infty} \int_{\mathbb{S}_{a,e}} \phi  \,d \delta_{O(\omega_{\ell})_{\infty}} = \int_{\mathbb{S}_{a,e}} \phi \, d\nu_{a,e}.
  \end{equation}
\end{proposition}

\begin{proof}
As $a$ is primitive, after a  change of variables we can suppose
  without loss of generality that~$a=(0,1)$,  so that
  \begin{displaymath}
    V_{a,e}=Z(\Phi_{e}(t_{2})) = \bigcup_{\zeta\in\mu_{e}^{\circ}}\overline{\mathbb{Q}}^{\times} \times \{\zeta\}
    \quad\text{and}\quad
    \mathbb{S}_{a,e}=\bigcup_{\zeta\in \mu_{e}^{\circ}}S^1\times\{\iota_\infty(\zeta)\}.
  \end{displaymath}
  
  For each $\ell \ge 1$ set $d_{\ell}=\ord(\omega_{\ell})$
  and~$e_{\ell}=\ord(\omega_{\ell,1})$. Since both $e_\ell$ and $e$
  divide $d_\ell$ we can consider the reduction homomorphisms
  \begin{displaymath}
    \pi_{d_{\ell},e_\ell}\colon
  (\mathbb{Z}/d_{\ell}\mathbb{Z})^{\times}\longrightarrow
  (\mathbb{Z}/e_\ell\mathbb{Z})^{\times}\quad\text{and}\quad\pi_{d_{\ell},e}\colon
  (\mathbb{Z}/d_{\ell}\mathbb{Z})^{\times}\longrightarrow
  (\mathbb{Z}/e\mathbb{Z})^{\times}.
\end{displaymath}
Set $J_{\ell}\subset (\mathbb{Z}/d_{\ell}\mathbb{Z})^{\times}$ for the
kernel of $\pi_{d_{\ell},e}$ and
$K_{\ell}\subset (\mathbb{Z}/e_\ell\mathbb{Z})^{\times}$ for its image
under~$\pi_{d_{\ell},e_{\ell}}$.  As~$\ord(\omega_{\ell,2})=e$, we
have that $d_{\ell}=\lcm(e_{\ell},e)$ and so~$J_{\ell}\simeq K_{\ell}$.

Because of \cite[Lemme~6.3]{CLT} we can assume without loss of
generality that $\phi$ is continuous.  Then we write each integral in
its left hand side of~\eqref{eq: equidistribution on Vae} as a sum of
integrals over the different connected components
of~$\mathbb{S}_{a,e}$ by suitably partitioning the corresponding
Galois orbit $O(\omega_{\ell})=\{\omega_{\ell}^{k} \mid k\in
(\mathbb{Z}/d_{\ell}\mathbb{Z})^{\times}\} $.  To this aim, for each
$\zeta\in\mu_e^\circ$ we choose
$r_{\ell,\zeta}\in (\mathbb{Z}/d_{\ell}\mathbb{Z})^{\times}$ such
that~$\omega_{\ell,2}^{r_{\ell,\zeta}}=\zeta$, which is possible
because~$\pi_{d_{\ell},e}$ is surjective thanks to \cref{lem:6}.  In
fact, $\omega_{\ell,2}^k=\zeta$ if and only if
$k\in J_\ell\cdot r_{\ell,\zeta}$ and so
\begin{multline}
  \label{eq:40}
  \int_{\mathbb{S}_{a,e}} \phi  \,d \delta_{O(\omega_{\ell})_{\infty}} =  \frac{1}{\varphi(d_{\ell})} \sum_{k\in(\mathbb{Z}/d_{\ell}\mathbb{Z})^{\times}} \phi(\iota_{\infty}(\omega_{\ell}^{k})) \\ 
=
\frac{1}{\varphi(e)} \sum_{\zeta\in\mu_e^\circ}\frac{1}{\#J_\ell}\sum_{j\in J_\ell} \phi\big(\iota_\infty\big(\omega_{\ell,1}^{j\cdot r_{\ell,\zeta}}\big),\iota_\infty(\zeta)\big)= 
 \frac{1}{\varphi(e)} \sum_{\zeta\in\mu_e^\circ}
\int_{S^1}\phi_\zeta \,d\delta_{\iota_{\infty}(K_{\ell} \cdot \omega_{\ell,1}^{r_{\ell,\zeta}})},
\end{multline}
where $\phi_\zeta$ denotes the continuous function on
$ S^{1}$ defined by~$\phi_{\zeta}(z)= \phi(z,\iota_\infty(\zeta))$, which is integrated against
the uniform probability measure on the $\infty$-adic orbit
of~$\omega_{\ell,1}^{r_{\ell,\zeta}}$ under the action of~$K_{\ell}$.

Since~$d_{\ell}=\lcm(e_{\ell},e)$, the condition in \eqref{eq:53}
implies that
\begin{displaymath}
  \lim_{\ell\to+\infty} \ord(\omega_{\ell,1}^{r_{\ell,\zeta}})=\lim_{\ell\to+\infty} e_{\ell}=+\infty,
\end{displaymath}
and so~$(\omega_{\ell,1}^{r_{\ell,\zeta}})_{\ell\ge 1}$ is a strict
sequence in~$\Gm(\overline{\mathbb{Q}})$.
In the terminology of \cite[Section~3]{DimitrovHabegger}, the conductor of the subgroup
$K_{\ell}$ of $(\mathbb{Z}/e_{\ell}\mathbb{Z})^{\times}$ is equal
to~$\gcd(e_{\ell},e)$, because this subgroup coincides with the kernel of the reduction homomorphism to $(\mathbb{Z}/\gcd(e_{\ell},e)\mathbb{Z})^{\times}$. In particular,
it is uniformly bounded.  Proposition
3.3(ii) of \emph{loc. cit.} then shows that the discrepancy of the
$\infty$-adic orbit
$\iota_{\infty}(K_{\ell} \cdot \omega_{\ell,1}^{r_{\ell,\zeta}})$
converges to~$0$ as~$\ell\to +\infty$.
Using standard results from measure theory like \cite[Theorem~5.4]{Harman:mnt} we deduce that
\begin{displaymath}
  \lim_{\ell\to +\infty} \int_{S^1}\phi_\zeta \,d\delta_{\iota_{\infty}(K_{\ell} \cdot \omega_{\ell,1}^{r_{\ell,\zeta}})}
=\int_{S^1}\phi_\zeta \,d \text{\rm Haar}=\int_{\mathbb{S}_{a,e}}\phi \,d \lambda_{a,e,\zeta}.
\end{displaymath}
The statement then follows from~\eqref{eq:40}.
\end{proof}

The next result is the analogue in our situation of \cref{prop:1}, and gives the limit of the Archimedean local height 
corresponding to the torsion points in a strict sequence in~$V_{a,e}$.
To state it, consider the subset of
$(\mathbb{R}/2 \pi \mathbb{Z})^{2}$ defined as
\begin{displaymath}
\mathbb{D}_{a,e}= \bigcup_{j \in (\mathbb{Z}/e\mathbb{Z})^{\times}} \Big\{u\in (\mathbb{R}/2 \pi \mathbb{Z})^{2} \, \Big|\ a_{1}u_{1}+a_{2}u_{2} = \frac{2\pi j}{e}\Big\}.
\end{displaymath}
It is a union of $\varphi(e)$ parallel segments, and we denote by
$\tau_{a,e}$ the probability measure on it induced from the Euclidean
metric on these segments. Recall that
\begin{displaymath}
F \colon (\mathbb{C}^{\times})^{2}\to \mathbb{R}\cup \{-\infty\} \and
f \colon (\mathbb{R}/2\pi\mathbb{Z})^{2} \to \mathbb{R}\cup \{-\infty\}  
\end{displaymath}
denote the functions defined in \eqref{eq:29} and in~\eqref{eq:11},
respectively.

\begin{proposition}
\label{prop:7}
Let $(\omega_{\ell})_{\ell\ge 1}$ be a strict sequence in $ V_{a,e}$
of nontrivial torsion points. Then
  \begin{displaymath}
    \lim_{\ell\to +\infty}    \h_{\infty}(\mathcal{P}(\omega_{\ell})) = 
    \int_{\mathbb{S}_{{a,e}}} F \, d\nu_{{a,e}} = \int_{\mathbb{D}_{a,e}} f\, d\tau_{a,e},
  \end{displaymath}
  where $\mathcal{P}(\omega_{\ell})$ is the vector of homogeneous
  coordinates of the projective point $C \cap \omega_{\ell} C$ as
  in~\eqref{eq:48}.
\end{proposition}

\begin{proof}
  For the first equality, note that when $e\ne 1$ the subset
  $\mathbb{S}_{a,e}$ does not contain the point~$(1,1)$. In this case, the
  restriction of the function $F$ to ${\mathbb{S}_{a,e}}$ is
  continuous, and therefore the statement follows from the formula in
  \eqref{eq:21} and \cref{prop:9}.

  When~$e=1$, the map
  $\psi\colon \Gm(\overline{\mathbb{Q}}) \to V_{a,1}$ defined as
  $\psi(s)=(s^{-a_{2}},s^{a_{1}})$ is an isomorphism of algebraic
  groups, because $a$ is primitive.  Denote also by
  $\psi\colon S^{1}\to \mathbb{S}_{a,1}$ the restriction to the unit
  circle of the complex version of this map. Then the function
  \begin{displaymath}
\psi^{*}F(s)-\log|s-1|=\log\max\Big(\frac{|s^{a_{1}}-s^{-a_{2}}|}{|s-1|},\frac{|s^{a_{1}}-1|}{|s-1|},
    \frac{|s^{-a_{2}}-1|}{|s-1|}\Big) 
  \end{displaymath}
  can be continuously extended to the point~$s=1$. Thus the formula in
  \eqref{eq:21} can be written as
  \begin{align*}
    \h_{\infty}(\mathcal{P}(\omega_{\ell})) & = 
    \int_{S^{1}}\psi^{*}F(s) \, d \delta_{O(\psi^{-1}(\omega_{\ell}))_{\infty}} \\ & =
    \int_{S^{1}}(\psi^{*}F(s)-\log|s-1|) \, d \delta_{O(\psi^{-1}(\omega_{\ell}))_{\infty}} +\int_{S^{1}}\log|s-1| \,
    d \delta_{O(\psi^{-1}(\omega_{\ell}))_{\infty}}.
  \end{align*}
  The first equality in the statement follows by applying the
  classical equidistribution theorem for roots of unity to the first
  integral in the formula above, and \cref{lem:2} to the second in a
  similar way as it was done to obtain the first equality in~\eqref{eq:27}.
  
  On the other hand, note that the image of $\mathbb{S}_{a,e}$ under
  the cotropicalization map coincides with the union of parallel
  segments $\mathbb{D}_{a,e}$ and that the direct image measure
  $\cotrop_{*}\nu_{a,e}$ coincides with~$\tau_{a,e}$. The second
  equality is then a direct consequence of the first together with the
  change of variables formula.
\end{proof}

Set for short
\begin{equation}
\label{eq:37}
\eta_{a,e}= \int_{\mathbb{D}_{a,e}} f\, d\tau_{a,e}.
\end{equation}
We have the following result in the spirit of  \cref{thm:1}. 

\begin{theorem}
  \label{thm:3}
  Let $(\omega_{\ell})_{\ell\ge 1}$ be a strict sequence in $ V_{a,e}$
  of nontrivial torsion points. Then
  \begin{displaymath}
    \lim_{\ell\to +\infty}    \h(C\cap \omega_{\ell} C) = \eta_{a,e}.
  \end{displaymath}
\end{theorem}

\begin{proof}
  This follows readily from the definition of the height together with
  \cref{prop:7} and \cref{cor:2}.
\end{proof}

Using \cref{thm:3}, we can find other interesting  limit values for the height, as the next two examples show.

\begin{example}
  \label{exm:2}
  For $a=(0,1)$ and $e=1$ the segment $\mathbb{D}_{a,e}$ can be
  parametrized with the unit interval by the map~$w\mapsto (2\pi w,0)$. Hence
  \begin{displaymath}
\eta_{a,e}= \int_{0}^{1} \log\max(|\eu^{\iu 2\pi w}-1|, |0|,|\eu^{\iu 2\pi w}-1|) \, dw
    = \int_{0}^{1} \log|\eu^{\iu 2\pi w}-1| \, dw =0.
  \end{displaymath}
  By \cref{thm:3}, the limit of the height of $C\cap \omega_{\ell}C$
  for a strict sequence $(\omega_{\ell})_{\ell\ge 1}$ in $Z(t_{2}-1)$
  of nontrivial torsion points is equal to~$0$, in agreement with
  \cref{prop: bounds for the height}. A similar observation holds for~$a \in \{(1,0), (1,-1)\}$.
\end{example}

\begin{example}
  \label{exm:3}
  For $a=(2,-1)$ and $e=1$ the segment 
  $\mathbb{D}_{a,e} $ can be parametrized with the unit interval by the map~$w\mapsto (2\pi w, 4\pi w)$. Hence
  \begin{multline*}
\eta_{a,e}= \int_{0}^{1} \log \max(|\eu^{\iu 4\pi w}-\eu^{\iu 2\pi w}|, |\eu^{\iu 4\pi w}-1|, |\eu^{\iu 2\pi w}-1|) \, dw \\=
    \int_{0}^{1} \log |\eu^{\iu 2\pi w}-1| \, dw +    \int_{0}^{1} \log \max(1, |\eu^{\iu 2\pi w}+1|) \, dw = \m(x_{0}+x_{1}+x_{2}),
  \end{multline*}
  where the last equality follows from Jensen's formula. This
  logarithmic Mahler measure was computed by Smyth as
  \begin{displaymath}
    \m(x_{0}+x_{1}+x_{2}) = \frac{3\sqrt{3}}{4\pi} L(\chi_{-3},2)= 0.323065\dots
  \end{displaymath}
  for the $L$-function associated to odd Dirichlet character modulo~$3$~\cite{Smyth}.
  By \cref{thm:3}, this quantity gives the limit of
  the height of $C\cap \omega_{\ell}C$ for a strict sequence
  $(\omega_{\ell})_{\ell\ge 1}$ in $Z(t_{1}^{2}t_{2}^{-1}-1)$ of
  nontrivial torsion points, like that in in \cref{ex: height of some
    solutions}. A similar situation occurs when
  $a\in \{(1,1), (1,-2)\}$.
\end{example}

Similarly as in \cref{sec:distribution-height}, we can extend our
study to strict sequences of nonempty finite subsets of~$V_{a,e}$, in the sense of \cref{def: strict sequences of sets}.
Each algebraic subgroup of $\Gm^{2}(\overline{\mathbb{Q}})$ not
containing $V_{a,e}$ intersects this algebraic subset in finitely many
points, and therefore any sequence of nonempty finite subsets of
$V_{a,e}$ satisfying $\lim_{\ell\to +\infty}\# E_{\ell}=+\infty$ is
automatically strict in $V_{a,e}$ (the condition is not necessary,
though).

 \begin{example}
   \label{exm:5}
   The sequence of nonempty subsets of $d$-torsion points of $V_{a,e}$
   is strict. Indeed, let $d\ge 1$ and consider the monomial map~$\chi^{a}\colon \mu_{d}^{2}\to \mu_{d}$, so that
   \begin{equation}
     \label{eq:65}
     \mu_{d}^{2}\cap V_{a,e}=(\chi^{a})^{-1}(\mu_{d}\cap \mu_{e}^{\circ}).
   \end{equation}
   If $e\nmid d$ then $\mu_{d}\cap\mu_{e}^{\circ}=\emptyset$ and so~ $\mu_{d}^{2}\cap V_{a,e}=\emptyset$. Otherwise $e\mid d$ and
   $\mu_{e}^{\circ}\subset \mu_{d}$, and so it follows from
   \eqref{eq:65} and the surjectivity of $\chi^{a}$ that
   \begin{displaymath}
     \# (\mu_{d}^{2}\cap V_{a,e})=\# \Ker(\chi^{a}) \cdot \# \mu_{e}^{\circ}= d\, \varphi(e).
   \end{displaymath}
   Thus $(\mu_{d}^{2}\cap V_{a,e})_{d\ge 1, e\mid d}$ is a strict
   sequence in~$V_{a,e}$.
 \end{example}
 
 The next two results are the analogues in our current setting of
 \cref{thm:2} and \cref{cor:3}. The first is a direct consequence of \cref{lem:4} and
 \cref{thm:3}, and the second follows using \cref{exm:5}. 

\begin{theorem}
\label{thm:4}
Let $(W_{\ell})_{\ell\ge 1}$ be a strict sequence of nonempty finite
subsets of $ V_{a,e}$ of nontrivial torsion points. Then for each
$\varepsilon >0$ we have that
  \begin{displaymath} 
    \lim_{\ell\to+\infty}    \frac{\# \{\omega\in W_{\ell} \mid |\h(C
      \cap \omega C)-\eta_{a,e}|<\varepsilon \}}{\#W_{\ell}} =1.
  \end{displaymath}
Moreover 
  \begin{math}
\displaystyle{    \lim_{\ell\to+\infty}
\frac{1}{\# W_{\ell}} \sum_{\omega\in W_{\ell}}\h(C\cap \omega C)=\eta_{a,e}.}
\end{math}
\end{theorem}

\begin{corollary}
  \label{cor:6}
For each $\varepsilon >0$ we have that
\begin{displaymath}
  \lim_{\substack{d\to +\infty\\ e\mid d}}  \frac{\#\{\omega\in \mu_{d}^{2}\cap V_{a,e}\setminus\{(1,1)\} \mid |\h(C\cap \omega C)
  -\eta_{a,e}|<\varepsilon \}}{\# (\mu_{d}^{2}\cap V_{a,e}\setminus\{(1,1)\})} =1.
\end{displaymath}
Moreover
  \begin{math}
\displaystyle{  \lim_{\substack{d\to +\infty\\ e\mid d}}  
\frac{1}{\# (\mu_{d}^{2}\cap V_{a,e}\setminus\{(1,1)\})}\hspace{-1mm} \sum_{\omega \in \mu_{d}^{2}\cap V_{a,e}\setminus\{(1,1)\} } \hspace{-5mm}\h(C\cap \omega C)=\eta_{a,e}.}
\end{math}
\end{corollary}

\section{Visualizing the results}
\label{sec:visualizing-result}

In this section we present a series of computations done with the {\tt
  SageMath} notebook~\cite{GS_Notebook} that allow to visualize our
results while at the same time suggest further intriguing questions
and conjectures.

We focus on the height values of the points of the form
$P(\omega)=C\cap\omega C$ as $\omega$ ranges in the set of nontrivial
$d$-torsion points of~$\Gm^{2}(\overline{\mathbb{Q}})$.  The next
statement specifies how to enumerate these torsion points and compute
the corresponding heights.

\begin{proposition}
  \label{prop:8}
  Let $d\ge 1$ and~$\zeta \in \mu_{d}^{\circ}$. Then
  \begin{enumerate}[leftmargin=*]
  \item \label{item:14} the map
    $(\mathbb{Z}/d\mathbb{Z})^{2} \to
    \mu_{d}^{2} $ defined as
    $c\mapsto (\zeta^{c_{1}},\zeta^{c_{2}})$ is a bijection, 
  \item \label{item:15} for each~$c\in(\mathbb{Z}/d\mathbb{Z})^{2}\setminus \{(0,0)\}$, letting
    $e=d/\gcd(c_{1},c_{2},d)$ we have that
\begin{multline*}
  \h( C\cap (\zeta^{c_{1}},\zeta^{c_{2}}) C)=-\frac{\Lambda(e)}{\varphi(e)}  \\+ \frac{1}{\varphi(e)} \sum_{k\in (\mathbb{Z}/e\mathbb{Z})^{\times}}\hspace{-2mm}\log\max(|\eu^{{2\pi\iu c_{2}k}/{d}}-\eu^{{2\pi\iu c_{1}k}/{d}}|, |\eu^{{2\pi\iu c_{2}k}/{d}}
  -1|, |\eu^{{2\pi\iu c_{1}k}/{d}} -1|)
\end{multline*}
where  $\Lambda$ and $\varphi$ denote the von Mangoldt and the Euler totient
functions, 
\item \label{item:16} the function
  $(\mathbb{Z}/d\mathbb{Z})^{2}\setminus \{(0,0)\} \to\mathbb{R}$
  defined as
  $ (c_1,c_2)\mapsto\h(C\cap (\zeta^{c_{1}},\zeta^{c_{2}})C)$ does not
  depend on the choice of $\zeta\in\mu_d^\circ$ and it is invariant
  under the transformations 
\[
(c_1,c_2)\mapsto(c_2,c_1),\quad (c_1,c_2)\mapsto(-c_2,c_1-c_2),\quad (c_1,c_2)\mapsto(-c_1,-c_2)
\]
on~$(\mathbb{Z}/d\mathbb{Z})^{2}\setminus \{(0,0)\}$.
  \end{enumerate}  
\end{proposition}

\begin{proof}
  The statement in \eqref{item:14} is a particular case
  of~\eqref{eq:44}, whereas that in \eqref{item:15} follows directly
  from \cref{lem:5}, \cref{cor:1} and the definition of the height.
  Finally, the statement in \eqref{item:16} is an easy consequence of
  that in~\eqref{item:15}.
\end{proof}

\vspace{\baselineskip}
\noindent\textbf{Displaying the computations.}
For each $d\ge1 $ we can numerically compute the height of the point~$C\cap\omega C$
for each~$\omega\in\mu_d^2\setminus\{(1,1)\}$ using the formula in
\cref{prop:8}\eqref{item:15}, and display the outputs in two images.
The first shows the obtained height values as an unordered plot on the
unit interval, whereas the second represents them in a meaningful and
organized way in the unit square.

The latter image is produced as we next explain. Choosing~$\zeta\in \mu_{d}^{\circ}$,
\cref{prop:8}\eqref{item:14} identifies
$\mu_d^2$ with the set of grid points
$ \{ 0, {1}/{d}, \dots, {(d-1)}/{d}\}^{2}$ of the unit
square~$[0,1)^2$. To visualize the behavior of the function
\begin{displaymath}
\omega\mapsto\h(C\cap \omega C) \quad \text{ for } \omega\in\mu_d^2\setminus\{(1,1)\}   
\end{displaymath}
we subdivide this square into $d^{2}$-many cells centered at these
grid points. Apart from that with center~$(0,0)$, each  of these cells
corresponds to a nontrivial $d$-torsion point~$\omega$, and we color
it with a tone of gray that is as dark as the height of
$C\cap\omega C$ is larger within the range~$[0,\log(2)]$.  By virtue
of \cref{prop:8}\eqref{item:16}, the resulting image does not depend
on the choice of the primitive $d$-root of unity~$\zeta$.

\cref{fig:3} shows the resulting images for~$d=120$. For future
considerations, we enrich the first plot with the special values
\begin{equation}
  \label{eq:62}
  \eta=\frac{2\, \zeta(3)}{3\, \zeta(2)}=0.487175\dots
\and 
  \theta= \frac{3\sqrt{3}}{4\pi} L(\chi_{-3},2)= 0.323065\dots
\end{equation}
marked with a red line and an orange line, respectively.

\begin{figure}[!htbp]
  \includegraphics[scale=0.45]{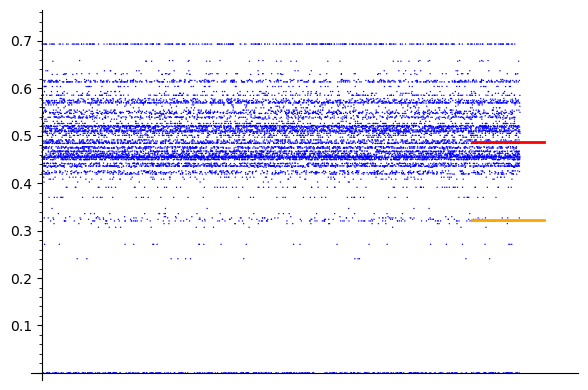} \quad\quad  \includegraphics[scale=0.55]{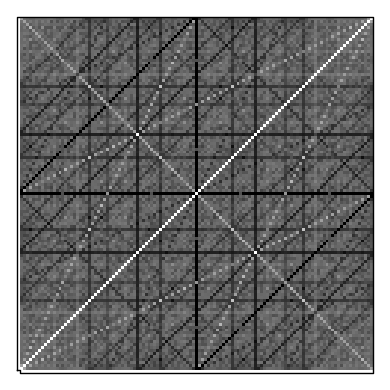}
  \caption{Distribution of heights associated to $d$-torsion points, $d=120$}
\label{fig:3}
\end{figure}

\vspace{\baselineskip}
\noindent\textbf{Bounds and symmetries.}
These figures allow to appreciate several features of the height values we study. 

As predicted by \cref{prop: bounds for the height}, the extremal
values are~$0$ and~$\log(2)$, as can be seen in the left image of
\cref{fig:3}. In the right
image, the minimal height (in white) is reached for the torsion
points corresponding to the cells in the horizontal, diagonal and
vertical lines passing through the origin of the square, plus at the grid points $(1/{3},2/{3})$ and~$(2/{3},1/{3})$
whenever~$3\mid d$.  The maximal height (in black) is only attained when
$d$ is even but not a power of~$2$. When this is the case, it is
visible on the horizontal and vertical lines through the center of the
square, and on the diagonal passing through the midpoint of an edge of
the square, excluding the cells centered at the grid points of the form
$(2^{-k}b_{1},2^{-k}b_{2})$ with $b_{1},b_2\in \mathbb{Z}$
and~$k \ge 0$.

This right image also makes apparent the invariance of the height values
under the transformations in \cref{prop:8}\eqref{item:16} and their
compositions. Up to rescaling, these are the same automorphisms on
$(\mathbb{R}/2\pi\mathbb{Z})^2$ in \cref{prop:4}, which explains the
analogy of this image with the right image in~\cref{fig:1}.  In
particular, the triangle with vertices $(0,0)$, $(1/2,0)$ and
$(2/3,1/3)$ is a fundamental domain for the height~values.

\vspace{\baselineskip}
\noindent\textbf{Limit value.}
The limit behavior of the height is also visible.  As predicted by
\cref{cor:3}, most of the points in the left image of \cref{fig:3}
cluster around the value indicated by the red line, and most of the
cells in the right image have a tone of gray that approaches the
intensity
\begin{displaymath}
  \frac{\eta}{\log(2)}= 0.702845\dots
\end{displaymath}
as~$d\to +\infty$.

We can make this phenomenon more visible with a couple of plots by taking
$\varepsilon >0$ and computing the quantities
\begin{displaymath}
  \frac{1}{d^{2}-1}   {\#\{\omega\in \mu_{d}^{2}\setminus\{(1,1)\} \mid |\h(C\cap \omega C) -\eta|<\varepsilon \}}    \and
  \frac{1}{d^{2}-1} \sum_{\omega\in \mu_{d}^{2}\setminus \{(1,1)\}} \h(C\cap \omega C)
\end{displaymath}
for a family of values of~$d$.  \cref{fig:2} collects them for all
$1\le d\le 250$ and~$\varepsilon =0.1$.

\begin{figure}[!htbp]
\includegraphics[scale=0.43]{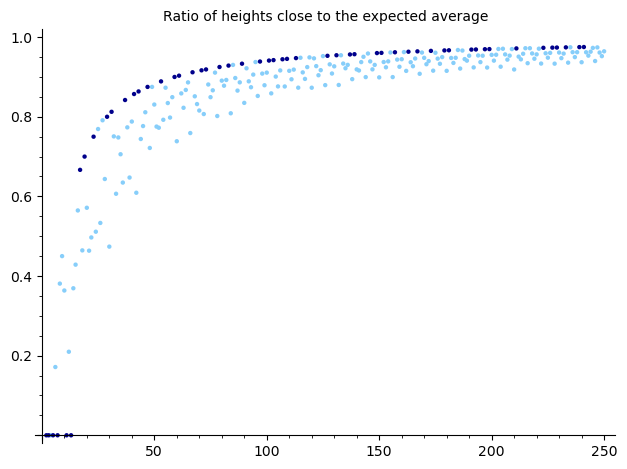}
  \quad
  \includegraphics[scale=0.43]{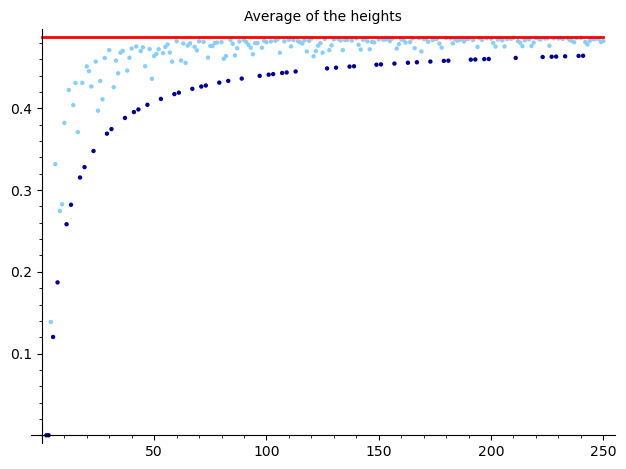}  
  \caption{Ratio of heights for $d$-torsion points around the
    limit and their mean, $1\le d\le 250$, with marked prime values}
\label{fig:2}
\end{figure}

In accordance with \cref{cor:3}, these plots confirm that most of the
considered heights cluster near $\eta$ as~$d \to +\infty$, and that
their mean converges to this limit value, marked in red in the image on the
right.

It would be interesting to understand other patterns suggested by
these computations. For instance, the right plot of
\cref{fig:2} hints to a positive answer to the following question.

\begin{question} 
  \label{conj:2}
  Does it hold that
  $\displaystyle{ \frac{1}{d^{2}-1} \hspace{-2mm} \sum_{\omega\in
      \mu_{d}^{2}\setminus \{(1,1)\}} \hspace{-3mm}\h(C\cap \omega C)
    <\eta}$ for all~$d\ge 1$?
\end{question}

\vspace{\baselineskip}
\noindent\textbf{Torsion curves.}
We can also visualize the asymptotics obtained in
\cref{sec:interm-sequ-tors}.  Recall that for a primitive vector
$a\in \mathbb{Z}^{2}$ and a positive integer~$e$ we consider the
disjoint union of torsion curves 
\begin{displaymath}
    V_{a,e}=\bigcup_{\zeta\in \mu_{e}^{\circ}} Z(\chi^{a}-\zeta) \subset \Gm^{2}(\overline{\mathbb{Q}}),
\end{displaymath}
and that the value $\eta_{a,e}$ in \eqref{eq:37} is the limit of the
heights corresponding to strict sequences of torsion points in this
algebraic subset.

\cref{cor:6} becomes apparent in the right
image of \cref{fig:3}, which shows that most of the height values
for the cells in the union of segments
\[
\bigcup_{j\in(\mathbb{Z}/e\mathbb{Z})^\times}\Big\{\Big(\frac{c_1}{d},\frac{c_2}{d}\Big)\,\Big|\,\frac{a_1c_1+a_2c_2}{d}=\frac{j}{e}\Big\}
\]
cluster around  $\eta_{a,e}$ as $d\to +\infty$
with~$e\mid d$. Indeed, the left image in this figure plus a bit of imagination
allows to see how these limit values are approached, and it would not
be difficult to produce plots analogous to those in \cref{fig:2} for a
given  algebraic subset~$V_{a,e}$.

\vspace{\baselineskip}
\noindent\textbf{Small heights.}
By Zhang's theorem proving the toric Bogomolov conjecture, the height
of the nontorsion points of the projective line $C$ is bounded below
by a positive constant, see \cite{Zagier:ancb01} for an effective
version. This can be appreciated in the left image of \cref{fig:3} as
a gap between 0 and the rest of the values.

A closer examination suggests that the first accumulation point for
the heights under consideration is the constant $\theta$ in~\eqref{eq:62}. 
As shown in \cref{exm:3}, this value is approached by
the heights corresponding to the grid points in the three segments
through the origin and orthogonal to the vectors $(2,-1)$, $(1,1)$
and~$(1,-2)$.
We can turn this observation into a
formal question as follows.

\begin{question}
  \label{conj:1}
  Let~$\varepsilon>0$. Is the set
  $\{\omega\in \mu_{\infty}^{2}\setminus\{(1,1)\} \mid 0<\h(C\cap \omega C)
  \le \theta-\varepsilon\}$ finite?
\end{question}

\vspace{-4mm}

\begin{figure}[!htbp]
  \includegraphics[scale=0.45]{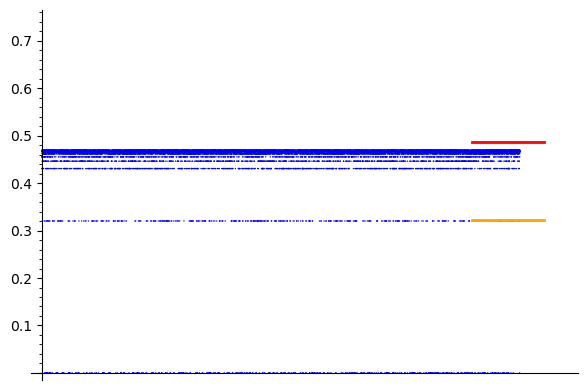} \quad\quad  \includegraphics[scale=0.55]{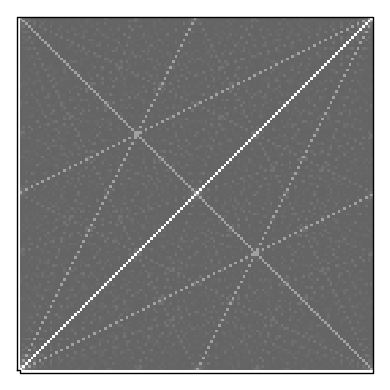}
  \caption{Distribution of heights associated to $d$-torsion points, $d=131$}
\label{fig:4}
\end{figure}

\noindent\textbf{Prime orders.}  In \cref{fig:2},  the quantities corresponding to
prime values of~$d$ are marked in dark blue, and they seem to have a
regular behavior for the two considered features. Indeed, both the
ratio of height values near~$\eta$ and their mean appear to follow
hyperbola-like patterns, at least for $d$ large enough. Moreover,
these patterns seem to bound, respectively from above and from below,
the corresponding quantity for a general~$d$. It would be interesting
to explain these behaviors.

In fact, when $d$ is prime, even the single height values corresponding to $d$-torsion points seem
to be simpler and tamer than those for composite~$d$, as can be
seen in \cref{fig:4} for~$d=131$.


This figure and similar ones suggest that for every prime~$d$ all the
height values corresponding to nontrivial $d$-torsion points lie
below~$\eta$, and that when they are nonzero, they are at least as
large as~$\theta$. In precise terms:

\begin{question}
  \label{conj:4}
  Let $p$ be a prime and $\omega\in \mu_{p}^{2} \setminus \{(1,1)\}$ with~$ \h(C\cap \omega C)\ne 0$. Does it hold that
  \begin{math}
    \theta\le \h(C\cap \omega C)<\eta?
  \end{math}
\end{question}

\phantomsection\label{part 2}
\begin{center}
\textbf{\textsc{PART II}}
\end{center}

\addcontentsline{toc}{chapter}{\hspace{2.9mm} Part II}

In this part we present an interpretation of~\cref{thm:1} from the
viewpoint of Arakelov geometry, that allows to recover it in a
more intrinsic way using the interplay between arithmetic and convex
objects from the Arakelov geometry of toric
varieties.

The first main result here is~\cref{thm: limit is Ronkin height},
linking the limit height in~\cref{thm:1} with the Arakelov height of
the subscheme of $\mathbb{P}^2_{\mathbb{Z}}$ defined by the linear
polynomial~$x_0+x_1+x_2$ with respect to suitable metrized line
bundles.  The second main result is \cref{cor: limit is integral over
  amoeba}, showing that this height agrees with the average of a
piecewise linear function over the Archimedean amoeba of this
subscheme, that can in turn be computed as a rational multiple of a
quotient of special values of the Riemann zeta function.

The treatment of this second part of the article is less elementary
and certainly not self-contained, building on the theory developed by
Gillet and Soul\'e~\cite{GilletSoule} and its refinement by
Maillot~\cite{Mail}, and on its study in the toric situation by Burgos
Gil, Philippon and the second named author~\cite{BPS} and by the first
named author~\cite{Gualdi}. We will recall the main objects and
results that we will use, assuming that the reader has some working
knowledge of these subjects, including the basics of complex analytic
geometry and of integral models of schemes.

\section{Semipositive metrics in complex geometry}\label{sec: semipositive metrics in complex geometry}

Let $X$ be a projective complex manifold with sheaf of holomorphic
functions~${O}_{X}$.  Let $L$ be a holomorphic line bundle on~$X$,
that is, a locally free sheaf of $O_{X}$-modules of rank~$1$.  A
\emph{(continuous) metric} $\|\cdot\|$ on $L$ is a rule that to every
open subset~$U$ of~$X$ and every section $s$ of $L$ on $U$ assigns
a continuous function
\begin{equation}
  \label{eq:43}
\|s(\cdot)\|\colon U\to\mathbb{R}_{\geq0}  
\end{equation}
that is compatible with restrictions to smaller open sets, and verifies
that
\begin{enumerate}[leftmargin=*]
\item \label{item:8} for every $p\in U$ we have that $\|s(p)\|=0$ if and only if~$s(p)=0$, 
\item \label{item:9} for every $p\in U$ and~$\lambda\in O_{X}(U)$ we have that~$\|(\lambda s)(p)\| =
  |\lambda(p)|\, \|s(p)\|$.
\end{enumerate}  
The pair $\overline{L}=(L,\|\cdot\|)$ is called a \emph{metrized line
  bundle} on~$X$.

\begin{remark}
  \label{rem:6}
  To define a metric on~$L$, it is enough to give a compatible
  choice of functions as those in \eqref{eq:43} for a family of
  nonvanishing sections of $L$ on open subsets that cover the whole
  of~$X$.
\end{remark}

To a metrized line bundle $\overline{L}$ on $X$ one can associate
its \emph{first Chern current}~$\chern(\overline{L})$, which is the
current of bidegree~$(1,1)$ defined  on any open subset $U$ of $X$ as
\begin{displaymath}
\chern(\overline{L})\big|_U=dd^{\ce}\big[-\log\|s\|\big|_{U}\big] 
\end{displaymath}
for any nonvanishing section $s$ on~$U$, where for a real-valued
function $f$ we denote by~$[f]$ the associated distribution, and the
operators $d$ and $d^{\ce}$ are taken in the sense of currents. This current
does not depend on the choice of the section because of the
Lelong--Poincar\'e formula.

A metrized line bundle $\overline{L}$ is \emph{semipositive}
(respectively \emph{smooth}) if for every nonvanishing local section
$s$ the function $p\mapsto -\log\|s(p)\|$ is plurisubharmonic
(respectively smooth).

\begin{example}
  \label{ex: trivial metric}
  The trivial line bundle $O_{X}$ admits a \emph{trivial
    metric}~$\|\cdot\|_{\tr}$, which is that defined by setting, for
  each holomorphic function $\lambda$ on an open subset $U$ of $ X$,
  \begin{displaymath}
    \|\lambda(p)\|_{\tr}=|\lambda(p)| \quad \text{ for all } p\in U. 
  \end{displaymath}
The corresponding metrized line bundle is
  denoted by~$\trOC$. Computing its first Chern current using the
  holomorphic function $\lambda=1$ we verify that~$\chern(\trOC)=0$,
  and so this metrized line bundle is semipositive.
\end{example}

\begin{example}
  \label{ex: canonical}
  Let~${O}(1)$ be the hyperplane line bundle of the complex projective
  space~$\mathbb{P}^{n}({\mathbb{C}})$. There is a one-to-one
  correspondence between its global sections and the linear
  polynomials in the homogeneous coordinates of~$\mathbb{P}^{n}({\mathbb{C}})$.
  For each global section $s$ we
  denote by~$l_s$ the corresponding linear polynomial and then~set
\[
\|s(p)\|_{\can}=\frac{|l_s(p_0,\ldots,p_n)|}{\max (|p_0|,\ldots,|p_n|)}  \quad \text{ for all } p=[p_{0}:\cdots:p_{n}]\in \mathbb{P}^{n}(\mathbb{C}).
\]
Since ${O}(1)$ is generated by its global sections, this assignment
determines through \cref{rem:6} a metric on this line bundle, which is
called the \emph{canonical metric} of~$O(1)$.
The corresponding metrized line bundle is denoted by~$\canOC$. It is
semipositive~\cite[Example~1.4.4]{BPS} but not smooth.
\end{example}

\begin{example}
  \label{ex: Ronkin metric}
  Consider the function
  $R\colon\mathbb{R}^{n+1}\setminus\{(0,\ldots,0)\}\to\mathbb{R}_{>0}$
  defined as
\begin{equation*}
  R(w)=
  \exp\Big(\int_{(S^{1})^{n}}\log|w_0+z_{1}w_{1}+\dots+z_{n} w_{n}|\ d\nu_{n}(z)\Big),
\end{equation*}
where $(S^{1})^{n}$ denotes the compact torus of
$(\mathbb{C}^\times)^n$ and $\nu_{n}$ its probability Haar measure.
Its value at a given point is the Mahler measure of an affine
polynomial, and so it is a well-defined real number.

The function $R$ is continuous and positive homogeneous of degree~$1$.
Similarly as for the canonical metric in the preceding example, the
\emph{Ronkin metric} of ${O}(1)$ is obtained by setting, for each
global section~$s$,
\[
\|s(p)\|_{\Ron}=\frac{|l_s(p_0,\ldots,p_n)|}{R(|p_0|,\ldots,|p_n|)} \quad \text{ for all } p=[p_{0}:\cdots:p_{n}]\in \mathbb{P}^{n}(\mathbb{C}).
\]
The corresponding metrized line bundle is denoted by~$\ronOC$.
\end{example}

\begin{remark}
  \label{rem:10}
  Ronkin metrics on line bundles on toric varieties were introduced
  in~\cite{Gualdi}, and they can be defined from any choice
  of nonzero Laurent polynomial compatible with the fan of the toric variety. The metric in \cref{ex: Ronkin
    metric} agrees with that from Definition~5.5 of
  \emph{loc. cit.}  for the Laurent polynomial~$1+t_1+\dots+t_n$. As
  explained therein, it is semipositive.
\end{remark}

Given a pure $d$-dimensional analytic cycle $Z$ of~$X$, we denote by
$\delta_Z$ its associated integration current and by $|Z|$ its
support.
Let also $\overline{L_{1}},\ldots,\overline{L_{d}}$ be a family of semipositive metrized line bundles on~$X$.
Then, Bedford--Taylor's theory
allows to construct a measure on~$X$, denoted~by
\[
\chern(\overline{L_{1}})\wedge\ldots\wedge\chern(\overline{L_{d}})\wedge\delta_{Z}
\]
and called the \emph{complex Monge--Amp\`ere measure} of
$\overline{L_{1}},\ldots,\overline{L_{d}}$ and~$Z$, see
\cite[Sections~III.3 and~III.4]{Demailly} for its definition and basic
properties.  This measure is supported on~$|Z|$, and has finite
integral against functions with at most logarithmic singularities
along analytic subsets of $|Z|$ of lower dimension. In particular, it
gives zero mass to those analytic subsets. Finally, this measure is
positive whenever~$Z$ is effective.

The next proposition is a particular case of well-known results in
analytic geometry, see for instance \cite[Lemma~8.17(ii)]{BE21}.  It
is also a direct consequence of the commutativity of the $*$-product
from Arakelov geometry shown  in \cite[Corollary~2.2.9]{GilletSoule} and
extended  in \cite[Proposition~5.3.6]{Mail} to metrized line
bundles that are not necessarily smooth.  We include its proof for
convenience and greater clarity.

\begin{proposition}[metric Weil reciprocity law]\label{prop: metric
    Weil reciprocity}
  Let $X$ be a smooth projective complex curve.  For~$i=1,2$, let also
  $\overline{L_i}=(L_i,\|\cdot\|_i)$ be a semipositive metrized line
  bundle on~$X$, and $s_i$ a nonzero rational section of~$L_i$.
  Suppose that the $0$-cycles
  \begin{displaymath}
    Z_{i}=\sum_{p\in X} \ord_{p}(s_{i}) \, [p] \quad \text{ for } i=1,2
  \end{displaymath}
  have disjoint supports. Then
  \begin{equation}
    \label{eq:47}
    \int_{X}\log\|s_1\|_1\, \big(\chern\big(\overline{L_2}\big)\wedge\delta_{X}-\delta_{Z_{2}}  \big) =
    \int_{X}\log\|s_2\|_2\, \big(
    \chern\big(\overline{L_1}\big)\wedge\delta_{X}-\delta_{Z_{1}}  \big).
  \end{equation}
\end{proposition}

\begin{proof}
  Both sides of the formula in \eqref{eq:47} are linear in the choice
  of the metrized line bundles and rational sections. Hence we can
  assume that $L_{1}$ and $L_{2}$ are ample. When this is the case,
  their metrics can be uniformly approached by smooth semipositive
  metrics on the same line bundles. Since both sides of this formula
  are continuous with respect to the uniform convergence of metrics,
  we can then reduce to the case in which both~$\|\cdot\|_{1}$ and
  $\|\cdot\|_{2}$ are smooth. We suppose this for the rest of the
  proof.
  
  For each $p\in |Z_{1}|\cup |Z_{2}|$ choose a closed neighborhood
  $B_{p}\subset X$ with smooth boundary.  Assume that all of these
  neighborhoods are disjoint and consider the open subset with smooth
  boundary
  \begin{displaymath}
    U=X\setminus \bigcup_{p\in |Z_{1}|\cup |Z_{2}|} B_{p}.
  \end{displaymath}

  Set $f_{i}=-\log\|s_{i}\|_{i}$ for each~$i$. Since the rational
  section $s_{i}$ is regular and nonvanishing on $U$ and the metric
  $\|\cdot\|_{i}$ is smooth, the restriction to $U$ of the complex
  Monge--Amp\`ere measure
  $ \chern\big(\overline{L_i}\big)\wedge\delta_{X}$ is given by the
  $(1,1)$-form~$dd^\ce f_i|_{U}$. By \cite[Chapter~III,
  Formula~3.1]{Demailly},
  \begin{multline}
    \label{eq:13}
    \int_{U}\log\|s_1\|_1\,
    \chern\big(\overline{L_2}\big)\wedge\delta_{X}- \log\|s_2\|_2\,
    \chern\big(\overline{L_1}\big)\wedge\delta_{X} \\=
    \int_{U} f_{2}\, dd^{\ce}f_{1}- f_{1}\,
    dd^{\ce}f_{2}=\int_{\partial U} f_{2} \, d^{\ce}f_{1} - f_{1} \,
    d^{\ce}f_{2}.
  \end{multline}
  Note that the integral in the left-hand side of \eqref{eq:13}
  approaches the same integral over the whole of $X$ when all the
  considered neighborhoods get arbitrarily~small.

  By spelling out the different connected components of $\partial U$
  we obtain that
\begin{equation}
  \label{eq:55}
\int_{\partial U} f_{1} \,
    d^{\ce}f_{2}  = -\sum_{p\in |Z_{1}|}\int_{\partial B_{p}} f_{1} \,
    d^{\ce}f_{2}    -\sum_{p\in |Z_{2}|}\int_{\partial B_{p}} f_{1} \,
    d^{\ce}f_{2}.     
  \end{equation}
  The first sum in the right-hand side of \eqref{eq:55} tends to $0$
  when these neighborhoods get small, as $d^{\ce}f_{2}$ is a smooth
  $(0,1)$-form on $B_{p}$ for every $p\in |Z_{1}|$ and~$f_{1}$ has a
  logarithmic singularity at every such point.
  
  Now let~$p\in |Z_{2}|$. Choosing the neighborhood $B_p$
  appropriately, we can suppose that there is a closed disc
  $V\subset \mathbb{C}$ centered at the origin together with a
  biholomorphic map $\varphi\colon V\to B_{p}$ such
  that~$\varphi(0)=p$. Then
  $\|(s_{2}\circ \varphi)(z)\|_{2}=|z|^{m}h(z) $ for~$m=\ord_{p}(s_{2})$ and a nonvanishing
  smooth function~$h$. Hence
  \begin{multline}
    \label{eq:49}
   - \int_{\partial B_{p}}f_{1}\, d^{c}f_{2} = \int_{\partial V}
    f_{1}\circ \varphi \ d^{\ce}\log\|s_{2}\circ
    \varphi\|_{2}\\= \frac{m}{2}\int_{\partial V} f_{1}\circ \varphi \
    d^{\ce}\log (z \overline z ) +
    \int_{\partial V} f_{1}\circ \varphi \
    d^{\ce}  \log (h).
  \end{multline}
  Since $d^{\ce} \log (h)$ is a smooth $(0,1)$-form on~$V$,
  the second summand in the right-hand side of \eqref{eq:49}
  tends to $0$ when $B_{p}$ shrinks to the point~$p$. For the first
  summand, recall that
  $d^{\ce}= (\partial -\overline \partial)/2\pi\iu$ with $\partial$
  and $\overline \partial$ the Dolbeault operators. Hence
  \begin{displaymath}
    d^{\ce} \log (z \overline z ) =\frac{1}{2\pi\iu} (\partial \log(z) -\overline{\partial} \log(\overline{z}))=
    \frac{1}{2\pi \iu} \Big(\frac{dz}{z}-
    \frac{d\overline z}{\overline z}\Big),
  \end{displaymath}
  and noting that the function $f_{1}$ is real-valued we deduce that
  \begin{equation}
    \label{eq:50}
    \int_{\partial V}  f_{1}\circ \varphi \ d^{\ce}\log (z \overline z ) =\frac{1}{2\pi \iu} \int_{\partial V}f_{1}\circ \varphi  \ \Big(\frac{dz}{z}-
    \frac{d\overline z}{\overline z}\Big)
    =\Re\Big( \frac{1}{\pi \iu} \int_{\partial V}f_{1}\circ \varphi \ \frac{dz}{z} \Big).
  \end{equation}
  Since $B_{p}$ contains no point of~$|Z_{1}|$, the function
  $f_1\circ\varphi$ is smooth on~$V$. Then Cauchy's formula \cite[Chapter~I,
  Formula~3.2]{Demailly} gives
  \begin{equation}
    \label{eq:51}
    \frac{1}{2\pi \iu} \int_{\partial V} f_{1}\circ \varphi \ \frac{dz}{z} =(f_{1}\circ \varphi)(0) + \int_{V} \frac{1}{\pi z} \frac{ \partial f_{1}\circ \varphi}{\partial \overline{z}} 
\,    d \lambda(z),
  \end{equation}
  where $\lambda$ denotes the Lebesgue measure of~$\mathbb{C}$.  As
  $B_{p}$ shrinks to~$p$, the integral in the right-hand side
  of~\eqref{eq:51} converges to $0$ because the function
  $z\mapsto {1}/{z}$ is integrable on~$V$, and so the integral in the
  left-hand side tends to $(f_{1}\circ \varphi)(0)=f_{1}(p)$.
  From~\eqref{eq:49} and~\eqref{eq:50} we deduce  that $ -\int_{\partial B_{p}}f_{1}\, d^{c}f_{2} $ converges to~$\ord_{p}(s_{2}) f_{1}(p)$.
  
  Thus when all the considered neighborhoods get small, it follows
  from \eqref{eq:55} that the integral
  $\int_{\partial U} f_{1} \, d^{\ce}f_{2}$ converges to the quantity
  \begin{displaymath}
    \sum_{p\in |Z_{2}|} \ord_{p}(s_{2}) f_{1}(p)= -\int_{X}
    \log\|s_{1}\| \, \delta_{Z_{2}}.  
  \end{displaymath}
  A similar consideration shows
  that~$\int_{\partial U} f_{2} \, d^{\ce}f_{1}$ converges to
  $- \int_{X} \log\|s_{2}\| \, \delta_{Z_{1}} $.  Taking the limit
  when the union of these neighborhoods converges to~$|Z_{1}|\cup |Z_{2}|$,
  we deduce from \eqref{eq:13} that
  \begin{displaymath}
        \int_{X}\log\|s_1\|_1\,
    \chern\big(\overline{L_2}\big)\wedge\delta_{X}- \log\|s_2\|_2\,
    \chern\big(\overline{L_1}\big)\wedge\delta_{X}
    = \int_{X}
  \log\|s_{1}\| \, \delta_{Z_{2}} - \int_{X}
  \log\|s_{2}\| \, \delta_{Z_{1}},
  \end{displaymath}
  as stated.
\end{proof}

\begin{remark}
  The classical Weil reciprocity law is the equality
  \begin{equation}
    \label{eq:46}
    \prod_{p\in X} f_{1}(p)^{\ord_{p}(f_{2})}
    =    \prod_{p\in X} f_{2}(p)^{\ord_{p}(f_{1})}
  \end{equation}
  for any pair of nonzero rational functions $f_{1}$ and $f_{2}$ of a
  smooth projective curve $X$ whose associated $0$-cycles have
  disjoint support \cite[Exercise~2.11]{Silverman:aec}.

  \cref{prop: metric Weil reciprocity} can be seen as a metric version
  of this reciprocity law, that essentially contains it.  Indeed
  choosing both $\overline{L_1}$ and $\overline{L_2}$ as the trivial metrized
  line bundle of~$X$ from \cref{ex: trivial metric}, this proposition
  yields the equality between the absolute values of the right-hand and
  left-hand sides of~\eqref{eq:46}.
\end{remark}

\section{The limit height of the intersection as an Arakelov height}
\label{sec:limit-height-inters}

In this section we link~\cref{thm:1} to the theory of heights from
Arakelov geometry. To this end, we first recall the basic elements of this
theory for arithmetic varieties equipped with semipositive metrized
line bundles that are not necessarily smooth.

Let $\mathscr{X}$ be an \emph{arithmetic variety}, that is a regular
integral projective flat scheme over~$\spec(\mathbb{Z})$, with sheaf
of regular functions~$\mathscr{O}_{\mathscr{X}}$. By the regularity
assumption, its set of complex points $\mathscr{X}(\mathbb{C})$ is a
projective complex manifold.
Given a line bundle $\mathscr{L}$ on~$\mathscr{X}$,
its analytification $\mathscr{L}^{\an}$ is a
holomorphic line bundle on~$\mathscr{X}(\mathbb{C})$.

A \emph{semipositive metrized
  line bundle} on $\mathscr{X}$ is a pair
\[
\overline{\mathscr{L}}=(\mathscr{L},\|\cdot\|)
\]
where $\mathscr{L}$ is a line bundle on $\mathscr{X}$ and $\|\cdot\|$
a semipositive metric on $\mathscr{L}^{\an}$ that is invariant under
the involution on $\mathscr{X}(\mathbb{C})$ induced by the complex
conjugation.  We denote
by~$\overline{\mathscr{L}^{\an}}=(\mathscr{L}^{\an},\|\cdot\|)$ the
associated semipositive metrized line bundle
on~$\mathscr{X}(\mathbb{C})$.

For a pure $d$-dimensional cycle $\mathscr{Z}$ of $\mathscr{X}$ and a
family of semipositive metrized line bundles
~$\overline{\mathscr{L}_{i}}=(\mathscr{L}_{i},\|\cdot\|_{i})$ for~$i=1, \dots, d$, on~$\mathscr{X}$, we define its \emph{(Arakelov)
  height}
\[
\h_{\overline{\mathscr{L}_1},\ldots,\overline{\mathscr{L}_{d}}}(\mathscr{Z})
\]
by means of the following recursion on the dimension:

\begin{enumerate}[leftmargin=*]
\item \label{item:12} when~$d=0$, if $\mathscr{Z}$ is a prime cycle then it is an integral closed point
  of $\mathscr{X}$ and its function field $\K(\mathscr{Z})$ is
  finite.
  In this case its height is defined as
\[
  \h(\mathscr{Z})=\log(\#\K(\mathscr{Z})).
\]
In the general case, the height of $\mathscr{Z}$  is defined by linearity.

\item \label{item:13} when~$d\geq1$, pick a rational section $s$ of
  $\mathscr{L}_d$ that is regular and nonvanishing on a dense open
  subset of the support of $\mathscr{Z}$ and set
\begin{multline*}
  \h_{\overline{\mathscr{L}_1},\ldots,\overline{\mathscr{L}_d}}(\mathscr{Z})=\h_{\overline{\mathscr{L}_1},\ldots,\overline{\mathscr{L}_{d-1}}}(\div(s)
  \cdot \mathscr{Z})
  \\
  -\int_{\mathscr{X}(\mathbb{C})}\log\|s^{\an}\|_{d}\
  \chern\big(\overline{\mathscr{L}^{\an}_{1}}\big)\wedge\ldots\wedge\chern\big(\overline{\mathscr{L}^{\an}_{d-1}}\big)\wedge\delta_{\mathscr{Z}(\mathbb{C})},
\end{multline*}
where $\div(s)\cdot\mathscr{Z}$ denotes the intersection product of
the Cartier divisor $\div(s)$ and the cycle~$\mathscr{Z}$, and
$s^{\an}$ stands for $s$ considered as a meromorphic section of~$\mathscr{L}^{\an}_d$.
\end{enumerate}
The height is a real number that does not depend on the choice of the
rational section~$s$ in \eqref{item:13} nor on the order in which the
metrized line bundles are chosen.

\begin{remark}
  \label{rem:1}
  The Arakelov height was introduced by Bost, Gillet and Soul\'e for
  smooth metrics through arithmetic intersection theory~\cite{BGS},
  and later extended by Maillot to semipositive metrics that are not
  necessarily smooth~\cite{Mail}.
\end{remark}

\begin{example}\label{ex: canonical height from canonical metric}
  Let $\mathbb{P}^n_{\mathbb{Z}}$ be the projective space over~$\spec(\mathbb{Z})$.
  The pair
\[
\canOZ=(\mathscr{O}(1),\|\cdot\|_{\can})
\]
where $\mathscr{O}(1)$ denotes the hyperplane line bundle of
$\mathbb{P}^n_{\mathbb{Z}}$ and $\|\cdot\|_{\can}$ the canonical
metric of the holomorphic line bundle $O(1)$ from \cref{ex:
  canonical} is a semipositive metrized line bundle on this
arithmetic variety, called the \emph{canonical metrized line bundle}
of~$\mathbb{P}_{\mathbb{Z}}^{n}$.
\end{example}

Let $\mathbb{P}^{n}(\overline{\mathbb{Q}})$ be the projective space
introduced in~\cref{sec:preliminaries}.  A point
$\xi \in\mathbb{P}^{n}(\overline{\mathbb{Q}})$ with rational homogeneous
coordinates can be identified with a $\mathbb{Q}$-point of
$\mathbb{P}^{n}_{\mathbb{Z}}$ or equivalently with a scheme-theoretic
point in its generic fiber. Denote by $\overline{\xi}$ its closure in~$\mathbb{P}^n_{\mathbb{Z}}$,
which is an integral subscheme of
dimension~$1$. Then
\begin{equation}
  \label{eq:52}
\h_{\canOZ}(\overline{\xi})=\h(\xi),  
\end{equation}
the quantity on the right-hand side being the height of $\xi$ in the
sense of~\eqref{eq:4}. Proving this equality is a nice exercise that
can be solved by unwrapping the corresponding definitions and by using \cref{rem:4}.

\begin{remark}\label{rem: heights of points over Qbar}
  More generally, the Arakelov height can be defined for cycles
  of~$\mathscr{X}_{\overline{\mathbb{Z}}}$, the base change of
  $\mathscr{X}$ with respect to the integral closure of the ring of
  integers.  In this setting, a point
  $\xi\in \mathbb{P}^{n}(\overline{\mathbb{Q}})$ can be identified
  with a $\overline{\mathbb{Q}}$-point of
  $\mathbb{P}^{n}_{\overline{\mathbb{Z}}}$ or equivalently with a
  scheme-theoretic point in the generic fiber of~$\mathbb{P}_{\overline{\mathbb{Z}}}^{n}$.
  The equality in
  \eqref{eq:52} extends then to $\xi$ and its closure $\overline{\xi}$
  in~$\mathbb{P}_{\overline{\mathbb{Z}}}^{n}$.
\end{remark}

Now consider the semipositive metrized line bundles on
$\mathbb{P}^2_{\mathbb{Z}}$
\begin{equation}
  \label{eq:39}
\canOZ=(\mathscr{O}(1),\|\cdot\|_{\can})
\and \ronOZ=(\mathscr{O}(1),\|\cdot\|_{\Ron})  
\end{equation}
obtained by equipping the holomorphic line bundle $O(1)$ on
$\mathbb{P}^{2}(\mathbb{C})$ with the canonical and the Ronkin metrics
from~\cref{ex: canonical,ex: Ronkin metric}, respectively.  Consider
also the subscheme $\mathscr{C}$ of $\mathbb{P}^2_{\mathbb{Z}}$
defined by the homogeneous linear polynomial~$x_0+x_1+x_2$.  The line
$C \subset \mathbb{P}^{2}(\overline{\mathbb{Q}})$ studied through
\hyperref[part 1]{Part I} coincides with the set of
$\overline{\mathbb{Q}}$-points of~$\mathscr{C}$.

The following is our main result in this section.

\begin{theorem}\label{thm: limit is Ronkin height}
  Let $(\omega_{\ell})_{\ell\ge 1}$ be a strict sequence in
  $\Gm^{2}(\overline{\mathbb{Q}})$ of nontrivial torsion points. Then
\[
\lim_{\ell\to +\infty}\h(C\cap \omega_{\ell}C) =\h_{\canOZ,\,\ronOZ}(\mathscr{C}).
\]
\end{theorem}

Its proof relies on the next complex analytic equality.  As in
\cref{sec: limit of archimedean heights} we denote by
$\mathbb{S} =(S^{1})^{2} $ the compact torus of
$(\mathbb{C}^{\times})^2$ and by $ \nu$ its probability Haar measure.

\begin{lemma}\label{lem: lemma Ronkin-integral on compact torus}
  Let $s_0$ be the global section of $\mathscr{O}(1)$ corresponding to
  the homogeneous coordinate $x_0$ of~$\mathbb{P}_{\mathbb{Z}}^{2}$.
  Then
\[
\int_{\mathbb{P}^{2} (\mathbb{C})}\log\|s_0^{\an}\|_{\Ron}\, \chern\big(\canOC\big)\wedge\delta_{\mathscr{C}(\mathbb{C})}
=
-\int_{\mathbb{S}}\log\max(|z_2-z_1|,|z_2-1|,|z_1-1|)\, d \nu(z).
\]
\end{lemma}

\begin{proof}
  Set for short
  $\mu=\chern(\canOC)\wedge\delta_{\mathscr{C}(\mathbb{C})}$.
  By construction, it is
  a measure on~$\mathbb{P}^2(\mathbb{C})$ supported on the
  line~$\mathscr{C}(\mathbb{C})$.  For each
  $z\in \mathbb{S}\setminus \{(1,1)\}$ we consider the integral
  \begin{equation}
    \label{eq:60}
F(z)=\int_{\mathbb{P}^{2}(\mathbb{C})}\log\Big|\frac{p_0+z_1p_1+z_2p_2}{p_0}\Big|\, d\mu(p),    
  \end{equation}
  which is finite because the restriction of the integrand to this
 line is a function which is continuous at all but two points, where it has at most logarithmic singularities.

 We will find a simpler expression for this integral as an application
 of the metric Weil reciprocity law. To this end, consider the inclusion
 $\jmath\colon \mathscr{C}(\mathbb{C})\to\mathbb{P}^2(\mathbb{C})$ and
 the semipositive metrized line bundles on $ \mathscr{C}(\mathbb{C})$
 given by the inverse image with respect to this map of the trivial
 and the canonical metrized line bundles on~$\mathbb{P}^2(\mathbb{C})$, that is
\begin{displaymath}
  \overline{L_1}=\jmath^*\trOP=\trOcurve
  \quad\text{and}\quad\overline{L_2}=\jmath^*\canOC.  
\end{displaymath}

On the one hand, the rational function $s_1=(x_0+z_1x_1+z_2x_2)/x_0$ of
$\mathbb{P}^{2}(\mathbb{C})$ restricts to a nonzero rational function
of~$\mathscr{C}(\mathbb{C})$, which gives a nonzero rational section $\jmath^*s_1$
of~$L_{1}=O_{\mathscr{C}(\mathbb{C})}$. On
the other hand, given complex numbers $\alpha_{1}$ and~$\alpha_{2}$
not both equal to~$1$, the linear
polynomial~$x_0+\alpha_{1} x_1+\alpha_{2} x_2$ corresponds to a global section $s_2$
of~${O}(1)$, which restricts to the nonzero global section $\jmath^*s_2$
of~$L_2=\jmath^*{O}(1)$.  The $0$-cycles of
$\mathscr{C}(\mathbb{C})$ respectively defined by these pulled-back rational
sections are
  \begin{equation}
    \label{eq:57}
Z_{1}=[z_1-z_2:z_2-1:1-z_1]-[0:1:-1] \and Z_{2}=[\alpha_{1}-\alpha_{2}:\alpha_{2}-1:1-\alpha_{1}].
  \end{equation}
  For an appropriate choice of~$\alpha_{1}$ and~$\alpha_{2}$, the
  supports of these $0$-cycles are disjoint.

  The restriction of the measure~$\mu$ to the line coincides
  with~$\chern\big(\overline{L_{2}}\big)$, by the functoriality of the
  Monge--Amp\`ere operator.
  Moreover~$\chern\big(\overline{L_1}\big)=0$, as shown in \cref{ex:
    trivial metric}.  Hence \cref{prop: metric Weil reciprocity}
  together with~\eqref{eq:60} and \eqref{eq:57} implies that
\begin{multline}
  \label{eq:58}
  F(z)=\int_{\mathscr{C}(\mathbb{C})} \log \|\jmath^*s_{1}\|_{1} \, \chern\big(\overline{L_{2}}\big) =\int_{\mathscr{C}(\mathbb{C})} \log \|\jmath^*s_{1}\|_{1} \, \delta_{Z_{2}} - \log\|\jmath^*s_{2}\|_{2} \, \delta_{Z_{1}}\\
  =
  \log\bigg(\frac{\|s_1([\alpha_{1}-\alpha_{2}:\alpha_{2}-1:1-\alpha_{1}])\|_{\tr}\cdot\|s_2([0:1:-1])\|_{\can}}{\|s_2([z_1-z_2:z_2-1:1-z_1])\|_{\can}}\bigg).
\end{multline}
We have that
\begin{align*}
  \|s_1([\alpha_{1}-\alpha_{2}:\alpha_{2}-1:1-\alpha_{1}])\|_{\tr}&= \frac{|(\alpha_{1}-\alpha_{2})+ z_{1}(\alpha_{2}-1)+z_{2}(1-\alpha_{1})|}{|\alpha_{1}-\alpha_{2}|}, \\
  \|s_2([0:1:-1])\|_{\can}&= |\alpha_{1}-\alpha_{2}|,\\
  \|s_2([z_1-z_2:z_2-1:1-z_1])\|_{\can}&= \frac{|(z_1-z_2)+\alpha_{1}(z_2-1)+\alpha_{2}(1-z_1)|}{\max(|z_2-z_1|,|z_2-1|,|z_1-1|)}.
\end{align*}
Hence  \eqref{eq:58} simplifies to 
\begin{equation}
  \label{eq:59}
  F(z) =\log \max(|z_2-z_1|,|z_2-1|,|z_1-1|),
\end{equation}
in accordance with the notation in~\eqref{eq:29}.

We now consider the functions
$g,h\colon \mathbb{S}\times \mathbb{P}^{2}(\mathbb{C})\to
\mathbb{R}\cup\{-\infty\}$ defined as
\begin{displaymath}
  g(z,p)=\log\Big|\frac{p_0+z_{1}p_1+z_{2}p_2}{p_0}\Big| \and
h(z,p)= \log \Big(\frac{|p_0|+|p_1|+|p_2|}{|p_0|}\Big)
\end{displaymath}
whenever~$(z,p)\notin Z(p_{0})$, and as and arbitrary constant
otherwise. The subset $Z(p_{0})$ has zero mass with respect to the
product measure~$\nu\times \mu$, and so the integrals of these
functions on $\mathbb{S}\times \mathbb{P}^{2}(\mathbb{C})$ do not
depend on the choice of this constant.

The function $h$ is constant with respect to the first variable,
whereas, as a function of the second one, its restriction to the line
$\mathscr{C}(\mathbb{C})$ only has a logarithmic singularity at a
point, making it integrable with respect to~$\nu\times \mu$. Moreover
$h-g$ is nonnegative, and from \eqref{eq:60} and~\eqref{eq:59} we
deduce that the iterated integral
\begin{displaymath}
\int_{\mathbb{S}}   \int_{\mathbb{P}^{2}(\mathbb{C})}   (h(z,p)-g(z,p))\, d\mu(p)  \, d\nu(z)
\end{displaymath}
is finite.  Tonelli's theorem then implies that the function~$h-g$, and
\emph{a fortiori}~$g$, is integrable with respect to the product
measure~$\nu\times \mu$, and that
\begin{equation}
  \label{eq:61}
  \int_{\mathbb{S}}   \int_{\mathbb{P}^{2}(\mathbb{C})}   g(z,p) \, d\mu(p) \, d\nu(z)
  = \int_{\mathbb{P}^{2}(\mathbb{C})}     \int_{\mathbb{S}} g(z,p) \, d\nu(z)  \, d\mu(p).
\end{equation}

Finally note that for $p=[p_0:p_1:p_2]\in\mathbb{P}^{2}(\mathbb{C})$
with~$p_0\neq0$ we have that
\begin{displaymath}
  \log\|s_0^{\an}(p)\|_{\Ron}
  =-\int_{\mathbb{S}} g(z,p)\, 
  d\nu(z).
\end{displaymath}
It follows from \eqref{eq:61} that
\begin{displaymath}
  \int_{\mathbb{P}^{2}(\mathbb{C})}\log\|s_0^{\an}(p)\|_{\Ron}\, d\mu(p)
  =-\int_{\mathbb{S}} \int_{\mathbb{P}^{2}(\mathbb{C})}g(z,p) \, d\mu(p) \, d\nu(z) = -\int_{\mathbb{S}} F(z) \, d \nu(z).
\end{displaymath}
Together with~\eqref{eq:59}, this gives the statement.
\end{proof}

\begin{proof}[Proof of \cref{thm: limit is Ronkin height}] 
  By \cref{cor:2} and~\cref{prop:1},
\[
\lim_{\ell\to +\infty}\h(C\cap \omega_{\ell}C)
=
\int_{\mathbb{S}} \log \max (|z_{2}-z_{1}|, |z_{2}-1|, |z_{1}-1|) \, d\nu(z).
\]
On the other hand, the recursive definition of the Arakelov height shows that
\[
\h_{\canOZ,\,\ronOZ}(\mathscr{C})
=
\h_{\canOZ}(\div(s_0)\cdot\mathscr{C})
-\int_{\mathbb{P}^2({\mathbb{C}})}\log\|s_0^{\an}\|_{\Ron}\, \chern\big(\canOC\big)\wedge\delta_{\mathscr{C}(\mathbb{C})}.
\]
The divisor $\div(s_0)\cdot\mathscr{C}$ coincides with the
$1$-dimensional irreducible subscheme of $\mathbb{P}_{\mathbb{Z}}^2$
arising as the Zariski closure of the point in the generic fiber of this scheme corresponding
to the point~$[0:1:-1] \in \mathbb{P}^{2}(\overline{\mathbb{Q}})$.
So \eqref{eq:52} gives
\begin{displaymath}
  \h_{\canOZ}(\div(s_0)\cdot\mathscr{C})=\h([0:1:-1] )=0.
\end{displaymath}
The statement then follows from~\cref{lem: lemma Ronkin-integral on
  compact torus}.
\end{proof}

\section{A toric perspective}\label{sec: toric perspective}

The rest of the article is concerned with expressing and computing the
Arakelov height in \cref{thm: limit is Ronkin height} in terms of convex
geometry through the theory of heights of toric varieties and
of their hypersurfaces, respectively studied in \cite{BPS} and
in~\cite{Gualdi}. Since the only toric variety that we need to
consider here is the projective space, we restrict our
presentation to this specific case. In spite of that, this is
sufficient to give a taste of the general theory.

Let $B \subset \mathbb{R}^{n}$ be a convex subset and
$g\colon B\to \mathbb{R}$ a concave function on it, that is a function
satisfying
\[
g(\lambda u+(1-\lambda)v)\geq \lambda g(u)+(1-\lambda)g(v) \quad \text{ for all } u,v\in B\text{ and } \lambda\in[0,1]. 
\]
Its \emph{stability set} is the subset of $\mathbb{R}^{n}$ defined as
\begin{displaymath}
  \stab(g)=\big\{ x\in \mathbb{R}^{n} \ \big|\,  \inf_{u\in B}(\langle x,u\rangle-g(u)) >-\infty \big\},
\end{displaymath}
where $\langle x,u \rangle= x_{1}u_{1}+\cdots+x_{n}u_{n} $ denotes the
usual scalar product. Then the \emph{Legendre--Fenchel dual} of $g$ is
the function $g^\vee\colon \stab(g)\rightarrow \mathbb{R} $ defined as
\[
g^{\vee}(x)=\inf_{u\in B}(\langle x,u\rangle-g(u)).
\]
It is concave and upper semicontinuous. If $g$ is upper
semicontinuous, then~$g^{\vee\vee}=g$.

\begin{example}\label{ex: duality indicator-support}
  Let $\Delta_{n}$ be the standard simplex of~$\mathbb{R}^n$, that is the
  $n$-dimensional simplex whose vertices are the origin and the
  vectors in the standard basis of this vector space.  Its
  \emph{support function} is the function $ \Psi_{\Delta_{n}}\colon
  \mathbb{R}^n \to \mathbb{R}$ defined as
  \begin{displaymath}
  \Psi_{\Delta_{n}}(u)= \inf_{x\in \Delta_{n}} \langle x,u\rangle =  \min(0,u_1,\ldots,u_n).    
  \end{displaymath}
  It is piecewise linear and concave, its stability set coincides with~$\Delta_{n}$,
  and its Legendre--Fenchel dual is~$0_{\Delta_{n}}$,
  the zero function on this simplex.
\end{example}

\begin{example}\label{ex: ronkin as concave function}
  Let $(S^{1})^{n}$ be the compact torus of $(\mathbb{C}^\times)^n$
  and $\nu_{n}$ its probability Haar measure.  The \emph{Ronkin
    function} is the function
  $\rho_{n}\colon\mathbb{R}^n\rightarrow\mathbb{R}$ defined as
\begin{displaymath}
\rho_{n}(u)=-\int_{(S^{1})^{n}}\log|1+z_{1}\eu^{-u_1}+\dots+z_{n}\eu^{-u_n}|\, d\nu_{n}(z).
\end{displaymath}
It is linked to the function $R$ in \cref{ex: Ronkin metric} by the relation
\begin{equation}
  \label{eq:35}
\rho_{n}(u)=-\log R(1,\eu^{-u_1},\ldots,\eu^{-u_n}).  
\end{equation}

The function $\rho_{n}$ is concave and its difference with the support
function $\Psi_{\Delta_{n}}$ is uniformly bounded
\cite[Proposition~2.10]{Gualdi}. This implies that its stability set
is the standard simplex $\Delta_{n}$ and that its Legendre--Fenchel
dual $\rho_{n}^\vee$ is a continuous concave function on it.
\end{example}

\begin{remark}
  \label{rem:5}
  Ronkin functions can be associated to any Laurent polynomial, or
  more generally to any holomorphic function on an open and
  $(S^{1})^{n}$-invariant subset of~$(\mathbb{C}^{\times})^{n}$.  As
  such, they were introduced by Ronkin in \cite{Ronkin} and further
  investigated by Passare and Rullg{\aa}rd in~\cite{PR}.

  Up to its sign, the function $\rho_{n}$ in \cref{ex: ronkin as concave
    function} is the special case corresponding to the Laurent
  polynomial~$1+t_1+\dots+t_n$.  For~$n=2$, Cassaigne and Maillot have
  found an explicit expression for it in terms of the Bloch--Wigner
  dilogarithm \cite[Proposition~7.3.1]{Mail}.  In spite of being our
  case of interest, we do not need such a description in this article.
\end{remark}

The importance of concave functions in our context stems from their
interplay with semipositive toric metrics on the hyperplane line
bundle of the complex projective space.  A metric $\|\cdot\|$ on
$O(1)$ is \emph{toric} if
\begin{equation}
  \label{eq:24}
\|s_0( z p )\| = \|s_{0}(p)\| \quad \text{ for all } p\in \mathbb{P}^{n}(\mathbb{C}) \text{ and } z\in(S^{1})^{n},
\end{equation}
where $z p $ denotes the translation of $p$ by $z$ as
in \eqref{eq:30} and $s_{0}$ the global section of~$O(1)$
corresponding to the homogeneous coordinate $x_0$
of~$\mathbb{P}^n(\mathbb{C})$.

To a toric metric $\|\cdot\|$ on $O(1)$ one can associate its
\emph{metric function}, which is the function
$ \psi_{\|\cdot\|}\colon \mathbb{R}^{n}\rightarrow \mathbb{R}$ defined
as
\begin{equation*}
 \psi_{\|\cdot\|}( u)=\log\|s_0([1:\eu^{-u_1}:\dots:\eu^{-u_n}])\|.
\end{equation*}
A toric metric on $O(1)$ is semipositive if and only if its metric
function is concave~\cite[Theorem~4.8.1(1)]{BPS}. When this is the
case, the stability set of this concave function is the standard
simplex~$\Delta_{n}$, and one also associates to the toric metric its
\emph{roof function}, which is the continuous concave function
\begin{displaymath}
 \vartheta_{\|\cdot\|} \colon \Delta_{n} \longrightarrow \mathbb{R} 
\end{displaymath}
defined as the Legendre--Fenchel dual of~$\psi_{\|\cdot\|}$.

\begin{example}\label{ex: toric metric from functions}
  The canonical metric of $O(1)$ from \cref{ex: canonical} is toric,
  as it can be readily checked from the condition in~\eqref{eq:24},
  and its metric function coincides with the support function $\Psi_{\Delta_{n}}$ from \cref{ex: duality indicator-support}. This recovers the semipositivity of $\|\cdot\|_{\can}$ and implies that its roof
  function is the zero function~$0_{\Delta_{n}}$.
\end{example}

\begin{example}
  \label{exm:6}
  The Ronkin metric of $O(1)$ from \cref{ex: Ronkin metric} is also
  toric by definition, and the associated metric function agrees with 
  the Ronkin function $\rho_{n}$ from \cref{ex: ronkin as concave function},
  by means of the relation~\eqref{eq:35}.
  In particular, it is a semipositive metric on $O(1)$ having as roof function
  its Legendre--Fenchel dual~$\rho_{n}^\vee$.
\end{example}

The main results of \cite{BPS} and \cite{Gualdi} allow to express the
height of $\mathbb{P}^n_{\mathbb{Z}}$ and of its hypersurfaces with
respect to a family of semipositive toric metrics on $\mathscr{O}(1)$
in terms of its associated family of roof functions.  This is achieved
through the \emph{mixed integral}~$\MI$, which is an operator on
families of $(n+1)$-many continuous concave functions on compact
convex subsets of $\mathbb{R}^{n}$ that is obtained by polarizing the
usual integral operator~\cite[Definition~2.7.16]{BPS}.

Since our case of interest is~$\mathbb{P}^2_{\mathbb{Z}}$, we adapt
the statements from the adelic setting of the \emph{loc. cit.} to the
schematic point of view of the present paper to obtain~the~following.

\begin{proposition}\label{prop: height of curve as Ronkin height and mixed integral}
  Let $\canOZ$ and $\ronOZ$ be the canonical and Ronkin metrized line
  bundles of $\mathbb{P}^{2}_{\mathbb{Z}}$ as in~\eqref{eq:39}, and
  ~$\mathscr{C}$ the subscheme defined by the linear
  polynomial~$x_{0}+x_{1}+x_{2}$. Then
\[
 \h_{\canOZ,\,\ronOZ} (\mathscr{C})
  =
 \h_{\canOZ,\,\ronOZ,\,\ronOZ}(\mathbb{P}^{2}_{\mathbb{Z}})
  = \MI(0_{\Delta_{2}},\rho_{2}^\vee,\rho_{2}^\vee).
\]
\end{proposition}

\begin{proof}
  This proof assumes a working knowledge of the adelic approach to the
  theory of Arakelov heights as exposed in \cite[Chapter~1]{BPS}.

  Let $\mathbb{P}^{2}_{\mathbb{Q}}$ be the projective plane over
  $\spec(\mathbb{Q})$ and $D$ the Cartier divisor of the line at
  infinity.  The pair $(\mathbb{P}_{\mathbb{Z}}^{2},\mathscr{O}(1))$
  is an integral model of~$(\mathbb{P}^{2}_{\mathbb{Q}}, O(D))$, and
  so the canonical and the Ronkin metrized line bundles on
  $\mathbb{P}_{\mathbb{Z}}^{2}$ induce quasi-algebraic semipositive
  adelic metrized divisors
  \[\canD\quad\and\quad\ronD
  \]
  on $\mathbb{P}^{2}_{\mathbb{Q}}$ in the sense of
  \cite[Section~1.5]{BPS}.  Namely, their Archimedean metrics coincide
  respectively with those of $\canOZ$ and~$\ronOZ$, and so do their
  Archimedean roof functions.  Instead, their non-Archimedean metrics
  are those induced by the integral model
  $(\mathbb{P}_{\mathbb{Z}}^{2},\mathscr{O}(1))$ as in
  \cite[Example~1.3.11]{BPS}, and then their roof functions agree with the
  zero function on~${\Delta_{n}}$.
  
  This correspondence is compatible with the associated heights, and
  thus \cite[Theorem~5.2.5]{BPS} together with \cref{ex: toric metric
    from functions,exm:6} specializes to
\begin{displaymath}
    \h_{\canOZ,\, \ronOZ,\,\ronOZ}(\mathbb{P}^{2}_{\mathbb{Z}})
 =    \h_{\canD,\, \ronD,\,\ronD}(\mathbb{P}^{2}_{\mathbb{Q}})
 = \MI(0_{\Delta_{2}},\rho_{2}^\vee,\rho_{2}^\vee),
  \end{displaymath}
  which gives the second equality in the statement.  We also have that
  \begin{displaymath}
 \h_{\canOZ,\, \ronOZ}(\mathscr{C})
 =    \h_{\canD,\, \ronD}(\mathscr{C}_{\mathbb{Q}})
  \end{displaymath}
because  $\mathscr{C}$  contains no vertical
  component, and \cite[Theorem~5.12]{Gualdi} gives
  \begin{displaymath}
    \h_{\canD,\, \ronD}(\mathscr{C}_{\mathbb{Q}})     =  \h_{\canD,\, \ronD,\, \ronD}(\mathbb{P}^{2}_{\mathbb{Q}})
  \end{displaymath}
  as $ \ronD$ coincides with the Ronkin metrized divisor of the
  Laurent polynomial $1+t_{1}+t_{2}$ in the sense of~\cite{Gualdi}.
  This gives the first equality and completes the
  proof.
\end{proof}

Combining this result with \cref{thm: limit is Ronkin height} and the
equality in \eqref{eq:52} yields the following statement.

\begin{corollary}
  \label{cor:5}
  Let $(\omega_{\ell})_{\ell\ge 1}$ be a strict sequence in
  $\Gm^{2}(\overline{\mathbb{Q}})$ of nontrivial torsion points. Then
\begin{displaymath}
\lim_{\ell\to +\infty}\h_{\canOZ}(Z(x_{0}+x_{1}+x_{2},\, x_{0}+\omega_{\ell,1}^{-1}x_{1}+\omega_{\ell,2}^{-1}x_{2}))=
    \h_{\canOZ,\,\ronOZ,\,\ronOZ}(\mathbb{P}^{2}_{\mathbb{Z}}),
  \end{displaymath}
  where
  $Z(x_{0}+x_{1}+x_{2},x_{0}+\omega_{\ell,1}^{-1}x_{1}+\omega_{\ell,2}^{-1}x_{2})$
  denotes the $1$-dimensional subscheme of~$ \mathbb{P}^{2}_{\overline{\mathbb{Z}}}$ defined by these linear
  forms.
\end{corollary}

This limit formula is a particular case of a conjectural arithmetic
analogue of the geometric fact that for a family of $n$-many line
bundles on an $n$-dimensional algebraic variety, the cardinality of
the zero set of a generic $n$-tuple of their global sections coincides
with the degree of the variety with respect to these line bundles.
Formulating this conjecture requires the language of adelic metrized
line bundles on varieties over global fields, which would take us too
far away from our setting. Instead, we content ourselves by explaining
a particular case that can be expressed with the objects at hand and
still gives a hint of the general case.

For~$d\ge 1$, consider the $d$-tensor power line bundle
$O(d)=O(1)^{\otimes d}$ on~$\mathbb{P}^{n}(\mathbb{C})$. Its global
sections are in one-to-one correspondence with the homogeneous
polynomials of degree $d$ in the homogeneous coordinates of this
projective space. Extending \cref{ex: Ronkin metric}, we define the
\emph{Ronkin metric} of $O(d)$ setting, for each global section~$s$,
\begin{displaymath}
  \|s(p)\|_{\Ron}=\frac{|l_s(p_0,\ldots,p_n)|}{R(|p_0|^d,\ldots,|p_n|^d)}
  \quad \text{ for all } p=[p_{0}:\cdots:p_{n}]\in \mathbb{P}^{n}(\mathbb{C}),
\end{displaymath}
where $l_{s}$ is the homogeneous polynomial of degree $d$
corresponding to~$s$, and $R$ is the same function as that of the
referred example.  This metric is toric and semipositive, and the
corresponding metrized line bundle on $\mathbb{P}^{2}_{\mathbb{Z}}$ is
denoted
by~$\overline{\mathscr{O}(d)}\vphantom{\mathscr{O}(d)}^{\Ron}$.

\begin{conjecture}\label{conj: conjecture for curves}
  Let $(\omega_\ell)_{\ell\geq1}$ be a strict sequence in
  $\Gm^{2}(\overline{\mathbb{Q}})$ of nontrivial torsion points. Then
  for~$d_1,d_2\geq1$ we have that
  \begin{multline*}
  \lim_{\ell\to
    +\infty}\h_{\canOZ}(Z(x_{0}^{d_{1}}+x_{1}^{d_{1}}+x_{2}^{d_{1}}, \, 
  x_{0}^{d_{2}}+
  \omega_{\ell,1}^{-d_{2}}x_{1}^{d_{2}}+\omega_{\ell,2}^{-d_{2}}x_{2}^{d_{2}}))\\
  =
  \h_{\canOZ,\,\overline{\mathscr{O}(d_1)}\vphantom{\mathscr{O}(d_1)}^{\Ron},\, \overline{\mathscr{O}(d_2)}\vphantom{\mathscr{O}(d_2)}^{\Ron}}(\mathbb{P}^{2}_{\mathbb{Z}}).    
  \end{multline*}
\end{conjecture}

In contrast with \cref{cor:5}, this statement would allow to predict
this kind of limit heights in situations where the considered systems
of equations cannot be explicitly solved.

\section{Integrals over an amoeba}
\label{sec:integr-over-amoeb}

In \cref{sec: limit of archimedean heights} we considered the
co-tropicalization map on $(\mathbb{C}^\times)^2$ that sends each
point to the arguments of its coordinates.  Looking instead at their
absolute values gives the more classical \emph{tropicalization map}:
\[
\trop\colon(\mathbb{C}^\times)^2\longrightarrow\mathbb{R}^2,\qquad z\longmapsto(-\log|z_1|,-\log|z_2|).
\]
The image of a curve of~$(\mathbb{C}^\times)^2$ under this map is
called its \emph{amoeba}. The term was coined
in~\cite[Section~6.1]{GKZ}, where these tentacle-shaped subsets were
introduced.

Here we will be concerned with the amoeba of the curve of
$(\mathbb{C}^{\times})^{2}$ defined by the Laurent polynomial~$1+t_{1}+t_{2}$,
that is
\begin{equation}\label{eq: amoeba}
\mathscr{A}
  =\{(-\log|z_{1}|,-\log|z_{2}|) \mid
  \gamma\in (\mathbb{C}^\times)^2 \text{ such that } 1+z_{1}+z_{2}=0\}.
\end{equation}
We will refer to this subset of $\mathbb{R}^{2}$ as the \emph{Archimedean
  amoeba} of~$C$, the line of $ \mathbb{P}^{2}(\overline{\mathbb{Q}})$
studied throughout \hyperref[part 1]{Part~I}. It is depicted in
\cref{figure amoeba 1+x+y}, and its computation is explained
in~\cite{GualdiAmoeba}.
\begin{figure}[!htbp]
\scalebox{.70}{
\begin{tikzpicture}
\fill [blue!10!white, domain=-2.95:3.5, variable=\x] (-3,-3)-- plot (\x,{-ln(1+exp(-\x))})--(3.5, 0)--(0,0)-- cycle;
\fill [blue!10!white, domain=-2.95:3.5, variable=\x] (-3,-3)-- plot ({-ln(1+exp(-\x))},\x)--(0,3.5)--(0,0)-- cycle;
\fill [blue!10!white, domain=3.5:ln(2), variable=\x] (0,0)-- (0,3.5)--plot ({-ln(1-exp(-\x))},\x)-- cycle;
\fill [blue!10!white, domain=ln(2):3.5, variable=\x] (0,0)-- plot (\x,{-ln(1-exp(-\x))})--(3.5,0)-- cycle;
\draw[blue!10!white,thick] (-3,-3)--(0.693,0.693);
\draw [blue,thick,domain=ln(2):3.5] plot (\x, {-ln(1-exp(-\x))});
\draw [blue,thick,domain=ln(2):3.5] plot ({-ln(1-exp(-\x))},\x);
\draw [blue,thick,domain=-2.95:3.5] plot (\x,{-ln(1+exp(-\x))});
\draw [blue,thick,domain=-2.95:3.5] plot ({-ln(1+exp(-\x))},\x);
\node [above right] at (0.7,0.7) {$\eu^{-u_1}+\eu^{-u_2}=1$};
\node [below right] at (0.3,-0.5) {$\eu^{-u_1}+1=\eu^{-u_2}$};
\node [above left] at (-0.3,0.7) {$\eu^{-u_2}+1=\eu^{-u_1}$};
\fill [blue!50!white, domain=-2.95:3.25, variable=\x] (-2.4,-2.4)-- plot (\x,{-ln(1+exp(-\x))})--(3.25, 0)--(0,0)-- cycle;
\draw (-3,0)--(3.5,0);
\draw (0,-3)--(0,3.5);
\end{tikzpicture}}
\caption{The Archimedean amoeba of $C$ with its contour lines and south region}
\label{figure amoeba 1+x+y}
\end{figure}
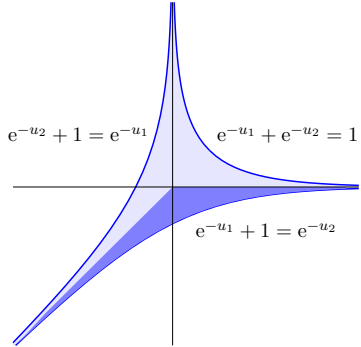
  
\begin{remark}
  \label{rem:11}
  This terminology can be justified as follows.  The line
  $C \subset \mathbb{P}^{2}(\overline{\mathbb{Q}})$ is defined over
  the rationals, and so it can be identified with an integral
  subscheme of~$\mathbb{P}^{2}_{\mathbb{Q}}$ whose Archimedean
  analytification is the complex
  line~$Z(x_{0}+x_{1}+x_{2})\subset \mathbb{P}^{2}(\mathbb{C})$. Then
  the subset $\mathscr{A}$ coincides with the amoeba of the
  restriction of this complex line to the dense open
  subset~$\mathbb{P}^{2}(\mathbb{C})\setminus Z(x_{0}x_{1}x_{2})\simeq
  (\mathbb{C}^{\times})^{2}$.
\end{remark}

For the sequel, consider   the subset of this Archimedean amoeba given as
\[
\mathscr{A}_{\mathrm{south}}=\{u\in \mathscr{A} \mid \min (0,u_1)\ge u_{2}\},
\]
which is the region colored in dark blue in~\cref{figure amoeba
  1+x+y}. 
The next result shows that the integrals on this region of the powers
of a coordinate function are given by special values of the Riemann
zeta function.


\begin{proposition}\label{prop: zeta values and amoebas}
For all~$m\in\mathbb{N}$ we have that
\[
\int_{\mathscr{A}_{\mathrm{south}}}\hspace{-5mm}u_2^m\, du_{1}du_{2}=(-1)^m \, m! \, \zeta(m+2).
\]
\end{proposition}

\begin{proof}
  The nontrivial boundary of the region $\mathscr{A}_{\mathrm{south}}$
  is given by $u_1=-\log(\eu^{-u_2}-1)$, as indicated in ~\cref{figure
    amoeba 1+x+y}. Then Tonelli's theorem and a direct
  computation~yield
\begin{multline}\label{eq: computation in the relation between zeta and integral over amoeba}
  \int_{\mathscr{A}_{\mathrm{south}}}\hspace{-5mm}u_2^m\, du_{1}du_{2}=\int_{-\infty}^0\bigg(\int_{u_2}^{-\log(\eu^{-u_2}-1)}du_1\bigg)\, u_2^m\, du_2\\
  =\int_{-\infty}^0u_2^m(-\log(\eu^{-u_2}-1)-u_2)\, du_2=\int_{-\infty}^0-u_2^m\log(1-\eu^{u_2})\, du_2.
\end{multline}

The absolute convergence of the series for the logarithm on the open
unit disk and the obvious statement for $u_2=0$ imply the pointwise
convergence
\[
-u_2^m\log(1-\eu^{u_2})=\sum_{k=1}^\infty\frac{u_2^m}{k}\eu^{ku_2} \quad \text{ for } u_2\in(-\infty,0],
\]
as functions with values in~$\mathbb{R}\cup\{\pm\infty\}$.  Since
every summand of the series has the same sign, Levi's monotone
convergence theorem implies that
\begin{multline}
\label{eq:3}
    \int_{-\infty}^0-u_2^m\log(1-\eu^{u_2})\, du_2
=
\int_{-\infty}^0\sum_{k=1}^\infty\frac{u_2^m}{k}\eu^{ku_2}\, du_2
=
\sum_{k=1}^\infty\frac{1}{k}\int_{-\infty}^0u_2^m\eu^{ku_2}\, du_2
\\=
\sum_{k=1}^\infty\frac{1}{k}\bigg[\eu^{ku_2}\sum_{\ell=0}^m\frac{(-1)^\ell m!}{k^{\ell+1}(m-\ell)!}u_2^{m-\ell}\bigg]_{-\infty}^0
\\=\sum_{k=1}^\infty(-1)^m\frac{m!}{k^{m+2}}=(-1)^m\, m!\,\zeta(m+2).
\end{multline}
The statement follows from \eqref{eq: computation in the relation
  between zeta and integral over amoeba} and~\eqref{eq:3}.
\end{proof}

This result can be employed to recover the computation of the area of
$\mathscr{A}$ by Passare and Rullg{\aa}rd in \cite[page~502]{PR}. A
similar computation appears in
\cite[page~919]{KramerPippich_communication}.

\begin{example}\label{ex: area of amoeba} 
  For~$m=0$, \cref{prop: zeta values and amoebas} gives that
  $\vol(\mathscr{A}_{\mathrm{south}})= \zeta(2)=\pi^{2}/6$.  Note that
  $\mathscr{A}$ is disjointly covered, up to subsets of measure zero,
  by the images of the south region under the map
  $(u_{1},u_{2}) \mapsto (-u_2,u_1-u_2)$ and under the symmetry along the
  diagonal. Since these maps preserve the Lebesgue measure, we deduce
  that
\[
  \vol(\mathscr{A})=3\,\vol(\mathscr{A}_{\mathrm{south}}) =3 \zeta(2) = \frac{\pi^2}{2}.
\]
\end{example}

The previous proposition also allows to compute the integral on the
Archimedean amoeba of $C$ of the support function of the
$2$-dimensional standard simplex.

\begin{example}
  \label{exm:4}
  For~$m=1$, \cref{prop: zeta values and amoebas} gives that
  $ \int_{\mathscr{A}_{\mathrm{south}}}u_2\, du_{1}du_{2}=- \zeta(3)$.
  Consider the subsets of $\mathscr{A}$ defined as
\[
  \mathscr{A}_{\mathrm{east}}=\{u\in \mathscr{A} \mid \min (u_1,u_{2}) \ge 0\}
  \and
  \mathscr{A}_{\mathrm{west}}=\{u\in \mathscr{A} \mid \min (0,u_{2})\ge u_{1}\}.
\]
Then
\begin{multline*}
  \int_{\mathscr{A}}\min(0,u_{1},u_{2}) \, du_{1}du_{2}= \int_{\mathscr{A}_{\mathrm{east}}} \hspace{-3.5mm} 0\, du_{1}du_{2}+\int_{\mathscr{A}_{\mathrm{west}}} \hspace{-3.5mm}u_{1} \, du_{1}du_{2}
  +\int_{\mathscr{A}_{\mathrm{south}}}\hspace{-5mm} u_{2} \, du_{1}du_{2}
  \\
  =2 \int_{\mathscr{A}_{\mathrm{south}}}\hspace{-5mm} u_{2} \, du_{1}du_{2}
    =- 2\,\zeta(3),
  \end{multline*}
  where the second equality follows from the symmetries of this
  integral.
\end{example}

\begin{remark}
  \label{rem:8}
  It would be interesting to explore if there are other integrals of piecewise
  polynomial functions on amoebas of curves that can be expressed in terms of
  special values of~$L$-functions.
\end{remark}


%
%
%
%
%
%
%
%
%
%
%
%
%
%
%
%
%
%
%
%
%
%
%
%
%
%
%
%
%
%
%
%
%
%
%
%
%
%
%
%
%
%
%
%

\section{Computing a mixed integral}\label{sec: mixed integral}

Finally, we show how the mixed integral in \cref{prop: height of curve as Ronkin height and mixed integral} relates to the
integral of a piecewise linear function on the amoeba treated in
\cref{sec:integr-over-amoeb}. The obtained equality offers a further
expression for the limit height under study, and it can be used to
recover our main result.

Set $\Delta=\Delta_{2}$ for the standard simplex of $\mathbb{R}^2$ and
denote by $0_{\Delta}$ the zero function on it.  Set also
$\Psi=\Psi_{\Delta_{2}}$ for the support function of this simplex
as in \cref{ex: duality indicator-support}, and~$\rho=\rho_{2}$ for the
Ronkin function of $1+t_1+t_2$ according to \cref{ex: ronkin as concave
  function}. Recall that $\mathscr{A}$ denotes the Archimedean amoeba
of the line $C$ as in~\eqref{eq: amoeba}.

\begin{theorem}
  \label{thm: mixed integral is integral over amoeba}
With notation as above, we have that
\[
\MI(0_\Delta,\rho^\vee,\rho^\vee)
=
-\frac{1}{\vol(\mathscr{A})}\int_{\mathscr{A}}\Psi(u) \, d u_{1} d u_{2},
\]
where
$\vol$ denotes the Lebesgue measure of~$\mathbb{R}^{2}$.
\end{theorem}

To prove this result, we first need to understand the behavior of the
Legendre--Fenchel dual of this Ronkin function on the boundary of its
domain.

\begin{lemma}\label{lem: Ronkin dual vanishes on boundary}
  The function $\rho^\vee$ vanishes on the boundary of~$\Delta$.
\end{lemma}

\begin{proof}
It is a consequence of Jensen's formula that
\begin{equation}\label{eq: Ronkin outside amoeba}
\rho(u)=\Psi(u)\quad\text{for all }u\in\mathbb{R}^2\setminus\mathscr{A},
\end{equation}
see for instance \cite[Proposition~7.3.1(2)]{Mail}. Since $\rho$ is
concave, this implies that the inequality $\rho\leq\Psi$ holds on the
entire real plane.  Taking Legendre--Fenchel duals we obtain that
$\rho^\vee(x)\geq 0_\Delta(x)=0$ for all~$x\in\Delta$.

Hence, to complete the proof it is enough to show that $\rho^\vee$ is
nonpositive on the boundary of~$\Delta$.  To this aim, first consider
a point $x=(\lambda,0)\in\partial\Delta$ with~$\lambda\in[0,1]$.  For
$\varepsilon>0$ the point
$u_{\varepsilon}=(\varepsilon,-\log(1-\eu^{-\varepsilon}))$ belongs to
the east border of~$\mathscr{A}$, as shown in \cref{figure amoeba
  1+x+y}.  By \eqref{eq: Ronkin outside amoeba} and the continuity of
the Ronkin function we necessarily have
that~$\rho(u_{\varepsilon})=0$.  Hence
\[
\rho^\vee(x)= \inf_{u\in \mathbb{R}^{2}} (\langle x,u\rangle -\rho(u)) 
\leq 
\lim_{\varepsilon\to 0^{+}}(\langle x,u_\varepsilon\rangle-\rho(u_\varepsilon))
= \lim_{\varepsilon\to 0^{+}}\lambda \varepsilon =0,
\]
which implies that~$\rho^\vee(x)=0$. The fact that this also holds
when~$x=(0,\lambda)$ with $\lambda\in[0,1]$ can be proven similarly.

To conclude, let $x=(\lambda,1-\lambda)$ with~$\lambda\in[0,1]$.
For~$\kappa>0$ consider the
point~$u_\kappa=(-\log(\eu^\kappa-1),-\kappa)$, which lies in the
southern border of~$\mathscr{A}$.  Again by \eqref{eq: Ronkin outside
  amoeba} and the continuity of the Ronkin function we have
that~$\rho(u_\kappa)=-\kappa$. Hence
\begin{multline*}
\rho^\vee(x) 
\leq
\lim_{\kappa\to+\infty}(\langle(x,u_\kappa\rangle-\rho(u_\kappa))
=
\lim_{\kappa\to+\infty} (\lambda \cdot (-\log(\eu^\kappa-1)) +(1-\lambda) \cdot (-\kappa) -(-\kappa)) \\
= \lim_{\kappa\to+\infty} -\lambda \log(1-\eu^{-\kappa})= 0,
\end{multline*}
which implies that~$\rho^\vee(x)=0$ also holds in this case, as
stated.
\end{proof}

We can now move to the proof of \cref{thm: mixed integral is integral
  over amoeba}, which relies on a relation between the mixed integral
and an integral on the dual space with respect to the real
Monge--Amp\`ere measure of a concave function.  The latter is a
positive measure having higher density where the function is more
concave \cite[Definition~2.7.1]{BPS}.

\begin{proof}[Proof of \cref{thm: mixed integral is integral over amoeba}]
Since both $0_\Delta$ and $\rho^\vee$ are continuous concave functions
  on the simplex $\Delta$ of~$\mathbb{R}^2$, it  is a consequence
  of \cite[Theorem~1.6]{Gualdi} together with the vanishing of
  $\rho^\vee$ on the boundary of $\Delta$ shown in \cref{lem: Ronkin dual
    vanishes on boundary} that
  \begin{equation}
    \label{eq:41}
\MI(0_\Delta,\rho^\vee,\rho^\vee)=- 2\int_{\mathbb{R}^2}\Psi\, d\mathcal{M}(\rho),    
  \end{equation}
where $\mathcal{M}(\rho)$ stands for the real Monge--Amp\`ere measure
of the Ronkin function.

By \cite[Theorem~7 and Example on page~502]{PR} the real
Monge--Amp\`ere measure of $\rho$ agrees with the Lebesgue measure of
$\mathbb{R}^{2}$ restricted to the amoeba~$\mathscr{A}$, and scaled by
the constant~$\pi^{-2}$, and so
\begin{equation}
  \label{eq:45}
  \int_{\mathbb{R}^2}\Psi\, d\mathcal{M}(\rho) =  \frac{1}{\pi^{2}} \int_{\mathscr{A}}\Psi(u)\, du_{1}d u_{2}.
\end{equation}
The statement follows from \eqref{eq:41} and \eqref{eq:45} together
with~\cref{ex: area of amoeba}.
\end{proof}

The following corollary is an immediate consequence of the chain of
equalities given by \cref{thm: limit is Ronkin height}, 
\cref{prop: height of curve as Ronkin height and
  mixed integral} and  \cref{thm: mixed integral is integral over amoeba}.

\begin{corollary}\label{cor: limit is integral over amoeba}
  Let $(\omega_{\ell})_{\ell\ge 1}$ be a strict sequence in
  $ \Gm^{2}(\overline{\mathbb{Q}})$ of nontrivial torsion points. Then
\[
\lim_{\ell\to +\infty}\h(C\cap \omega_{\ell}C)
=
\h_{\canOZ,\,\ronOZ}(\mathscr{C})
=
-\frac{1}{\vol(\mathscr{A})}\int_{\mathscr{A}}\Psi(u) \,  d u_{1} du_{2},
\]
where  $\mathscr{C}$ is the subscheme of $\mathbb{P}^2_{\mathbb{Z}}$
  defined by~$x_0+x_1+x_2$, whereas
  $\canOZ$ and
  $\ronOZ$ are the canonical and Ronkin
  semipositive metrized line bundles on~$\mathbb{P}^2_{\mathbb{Z}}$.
\end{corollary}

This result highlights the role played by our particular choice of
both height function and curve. Indeed, the considered limit height
turns out to be computed by the average of $\Psi$, which is metric
function associated  to the canonical height in the toric
correspondence from~\cite{BPS}, on the Archimedean amoeba
$\mathscr{A}$ of the line~$C$.

Finally, the  right-hand side in the equalities in \cref{cor: limit is integral over amoeba}
can be directly computed using \cref{ex: area of amoeba,exm:4}, namely 
\[
-\frac{1}{\vol(\mathscr{A})}\int_{\mathscr{A}}\Psi(u)\, du_{1}d{u_{2}}
=
\frac{2\, \zeta(3)}{3\, \zeta(2)},
\]
thus recovering \cref{thm:1} through this more conceptual point of
view.

\addtocontents{toc}{\protect\vspace{7pt}}

\bibliographystyle{amsalpha}
\bibliography{bibliography}

\end{document}